\documentclass[a4paper,10pt,reqno]{amsart}

\usepackage{amsmath,amsfonts,amscd,amssymb,amsthm, graphicx,mathrsfs,eufrak,enumitem,upgreek,fp}
\usepackage{extarrows}
\usepackage[mathscr]{euscript}
\usepackage{enumitem}
\usepackage{stmaryrd}

\usepackage{setspace}
\newcommand{\marginparstretch}{0.6}
\let\oldmarginpar\marginpar
\renewcommand\marginpar[1]{\-\oldmarginpar[\framebox{\setstretch{\marginparstretch}\begin{minipage}{\marginparwidth}{\raggedleft\tiny #1}\end{minipage}}]{\framebox{\setstretch{\marginparstretch}\begin{minipage}{\marginparwidth}{\raggedright\tiny #1}\end{minipage}}}}

\usepackage[colorlinks]{hyperref}
\usepackage{tikz,mathrsfs}
\usepackage{relsize}

\usetikzlibrary{arrows,decorations.pathmorphing,decorations.pathreplacing,positioning,shapes.geometric,shapes.misc,decorations.markings,decorations.fractals,calc}
\usetikzlibrary{matrix}

        \tikzset{
        cvertex/.style={circle,draw=black,inner sep=1pt,outer sep=3pt},
        vertex/.style={circle,fill=black,inner sep=1pt,outer sep=3pt},
        DBs/.style={circle,draw=black,circle,fill=black,inner sep=0pt, minimum size=3pt},
        DB/.style={circle,draw=black,circle,fill=black,inner sep=0pt, minimum size=4pt},
         DWs/.style={circle,draw=black,circle,fill=white,inner sep=0pt, minimum size=3pt},
         DWds/.style={circle,draw=black,densely dotted,circle,fill=white,inner sep=0pt, minimum size=3pt},
        DW/.style={circle,draw=black,inner sep=0pt, minimum size=4pt},
        tvertex/.style={inner sep=1pt,font=\scriptsize},
        gap/.style={inner sep=0.5pt,fill=white},
        Ggap/.style={inner sep=0.5pt,fill=green!40!black!20}}

\tikzstyle{mybox} = [draw=black, fill=blue!10, very thick,
   rectangle, rounded corners, inner sep=10pt, inner ysep=20pt]
\tikzstyle{boxtitle} =[fill=blue!50, text=white,rectangle,rounded corners]

\makeatletter
\def\calcLength(#1,#2)#3{%
\pgfpointdiff{\pgfpointanchor{#1}{center}}%
             {\pgfpointanchor{#2}{center}}%
\pgf@xa=\pgf@x%
\pgf@ya=\pgf@y%
\FPeval\@temp@a{\pgfmath@tonumber{\pgf@xa}}%
\FPeval\@temp@b{\pgfmath@tonumber{\pgf@ya}}%
\FPeval\@temp@sum{(\@temp@a*\@temp@a+\@temp@b*\@temp@b)}%
\FProot{\FPMathLen}{\@temp@sum}{2}%
\FPround\FPMathLen\FPMathLen5\relax
\global\expandafter\edef\csname #3\endcsname{\FPMathLen}
}
\makeatother

\usepackage{colonequals}

\newtheorem{theorem}{Theorem}[section]

\newtheorem{corollary}[theorem]{Corollary}
\newtheorem{lemma}[theorem]{Lemma}
\newtheorem{proposition}[theorem]{Proposition}
\newtheorem{definition-proposition}[theorem]{Definition-}

\newtheorem{setup}[theorem]{Setup}

\newtheorem{definition}[theorem]{Definition}

\newtheorem{example}[theorem]{Example}

\theoremstyle{definition}
\newtheorem{remark}[theorem]{Remark}
\newtheorem{notation}[theorem]{Notation}

\newcommand{\N}{\mathbb{N}}

\newcommand{\Ext}{\operatorname{Ext}\nolimits}
\newcommand{\Tor}{\operatorname{Tor}\nolimits}
\newcommand{\Hom}{\operatorname{Hom}\nolimits}
\newcommand{\uHom}{\operatorname{\underline{Hom}}\nolimits}
\newcommand{\uEnd}{\operatorname{\underline{End}}\nolimits}

\newcommand{\End}{\operatorname{End}\nolimits}

\newcommand{\RHom}{\mathbf{R}\strut\kern-.2em\operatorname{Hom}\nolimits}
\newcommand{\RshHom}{\mathbf{R}\strut\kern-.2em\mathscr{H}\strut\kern-.3em\operatorname{om}\nolimits}
\newcommand{\shHom}{\mathscr{H}\strut\kern-.3em\operatorname{om}\nolimits}

\newcommand{\Image}{\operatorname{Im}\nolimits}
\newcommand{\Kernel}{\operatorname{Ker}\nolimits}

\newcommand{\Spec}{\operatorname{Spec}\nolimits}

\newcommand{\Ker}{\operatorname{Ker}\nolimits}

\newcommand{\con}{\operatorname{con}\nolimits}
\newcommand\fundgp{\uppi_{\hspace{0.5pt}1}\hspace{-0.5pt}}

\newcommand{\tT}{\EuScript{T}}

\DeclareMathOperator{\moduleCategory}{\mathsf{mod}} \renewcommand{\mod}{\moduleCategory}
\DeclareMathOperator{\Mod}{\mathsf{Mod}}
\DeclareMathOperator{\proj}{\mathsf{proj}}

\DeclareMathOperator{\thick}{\mathsf{thick}}
\DeclareMathOperator{\coh}{\mathsf{coh}}

\DeclareMathOperator{\add}{\mathsf{add}}

\DeclareMathOperator{\cm}{\mathrm{CM}}
\DeclareMathOperator{\ucm}{\underline{\mathrm{CM}}}

\newcommand{\cmr}{\cm R}
\newcommand{\ucmr}{\ucm R}
\def\twotilt{\mathop{\text{$2$-}\sf tilt}\nolimits}

\newcommand{\cut}{\ar@{-}@[|(5)]}

\newcommand{\iso}{\cong}

\numberwithin{equation}{section}

\def\Db{\mathop{\mathrm{D^b}\kern -0.05em}\nolimits}
\def\Dnb{\mathop{\mathrm{D}\kern -0.05em}\nolimits}
\def\Kb{\mathop{\mathrm{K^b}\kern -0.05em}\nolimits}

\newcommand{\cC}{\EuScript{C}}
\newcommand{\cD}{\mathcal{D}}

\newcommand{\cG}{\mathcal{G}}
\newcommand{\cH}{\mathcal{H}}

\newcommand{\cS}{\mathcal{S}}

\newcommand{\id}{\mathrm{id}}

\setenumerate{label={\normalfont(\arabic*)}}

\begin{document}
\title{\textsc{The Tilting Theory of Contraction Algebras}}
\author{Jenny August}
\address{Max Planck Institute for Mathematics,
Vivatsgasse 7,
53111 Bonn,
Germany.}
\email{jennyaugust@mpim-bonn.mpg.de}
\begin{abstract}

To every minimal model of a complete local isolated cDV singularity Donovan--Wemyss associate a finite dimensional symmetric algebra known as the contraction algebra. We construct the first known standard derived equivalences between these algebras and then use the structure of an associated hyperplane arrangement to control the compositions, obtaining a faithful group action on the bounded derived category. Further, we determine precisely those standard equivalences which are induced by two-term tilting complexes and show that any standard equivalence between contraction algebras (up to algebra isomorphism) can be viewed as the composition of our constructed functors. Thus, for a contraction algebra, we obtain a complete picture of its derived equivalence class and, in particular, of its derived autoequivalence group.

\end{abstract}
\subjclass[2010]{Primary 16E35; Secondary 14E30, 16G50}
\maketitle
\parindent 15pt
\parskip 0pt

\section{Introduction}

Contraction algebras were introduced by Donovan--Wemyss as an invariant of certain birational maps \cite{DefFlops, twists}. For a special class of these maps, known as \emph{3-fold flopping contractions}, these algebras are always finite dimensional, symmetric and are known to control certain aspects of the associated geometry \cite{DefFlops, invsing, HomMMP}. This paper provides further evidence that, in these special cases, the derived category of a contraction algebra actually controls all of the geometry; in particular, we prove that the group structure of certain derived symmetries arising via \emph{flops} in the geometry descends to the derived category of the associated contraction algebra. 

Algebraically, this involves first constructing \emph{standard derived equivalences}, those of the form $\RHom(T,-)$ for some bimodule complex $T$, between the contraction algebras, and then secondly, understanding how to compose them. Both of these are generally difficult problems but using the associated geometry, we obtain something that is very rare in the literature: a class of finite dimensional algebras where we fully understand all the standard derived equivalences between them and in particular, understand the structure of their derived autoequivalence groups.

\subsection{Background and Motivation}
Given a complete local isolated cDV singularity $X= \Spec R$ (see \S2.1 for a definition of such 3-folds), the Minimal Model Programme outputs certain contractions $f \colon Y \to X$, known as \textit{minimal models}. Although minimal models of $X$ are not unique, any two minimal models are connected by a sequence of codimension two modifications called \textit{simple flops} (see e.g.\ \cite{[K]}) and further, if $f \colon Y \to X$ and $g \colon Z \to X$ are two minimal models of $X$ related by a simple flop, then there are associated derived equivalences
\begin{align*}
\Db(\coh Z) \to \Db(\coh Y) \quad \text{and} \quad \Db(\coh Y) \to \Db(\coh Z), 
\end{align*}
known as \emph{Bridgeland--Chen flop functors} \cite{bridgeland, chen}. As these are not inverse to each other, a nontrivial autoequivalence of $\Db(\coh Y)$ can be obtained by considering their composition. More generally, we think of any autoequivalence of $\Db(\coh Y)$ which is obtained as the composition of flop functors as a derived symmetry arising from birational geometry. \\

The Homological Minimal Model Programme studies this geometry using techniques from noncommutative algebra \cite{HomMMP}. The key observation is that $Y$ (as above) is derived equivalent to an algebra of the form $\End_R(M)$ which is known as the \emph{maximal modification algebra (MMA)} of $Y$. Sending $Y$ to $M$ induces a bijection between minimal models of $X$ and basic maximal rigid objects in the singularity category of $R$, written $\ucmr$, under which simple flops correspond to mutations at indecomposable summands \cite[4.10]{HomMMP}. Furthermore, mutation of maximal rigid objects induces derived equivalences of the MMAs which are isomorphic to the flop functors. Thus, questions about minimal models can be translated into algebraic problems involving only the MMAs. In this way, MMAs are known to control all the geometry but it is conjectured by Donovan--Wemyss \cite[1.3]{rigidequiv} that the same is true when we pass to the stable endomorphism algebra $\uEnd_R(M)$, known as the \emph{contraction algebra}. This motivates the goal of this paper which is to fully understand the derived equivalences between the contraction algebras of all minimal models of $\Spec R$. As well as being interesting algebraically, this will show the contraction algebras retain information about the corresponding flop functors.

\subsection{New Results}

As contraction algebras are symmetric, it is well known that they have no tilting modules and hence derived equivalences must be induced by tilting complexes of higher length. However, producing \emph{two-sided tilting complexes}, which are needed to obtain standard equivalences, is often very difficult. The first result of this paper uses the link with MMAs to construct such a complex.

\begin{theorem}[Theorems \ref{square}, \ref{equivalence}] \label{introexplicit}
Suppose that  $\Spec R$ is a complete local isolated cDV singularity and $M,N \in \ucmr$ are basic maximal rigid objects with $R$ as a summand. Writing $\Lambda$ and $\Gamma$ for the MMAs and $\Lambda_{\con}$ and $\Gamma_{\con}$ for the contraction algebras, define 
\begin{align}
\tT \colonequals \tau_{\scriptscriptstyle{ \geq -1}} ( \Gamma_{\!\con} \otimes^{\bf L}_\Gamma \Hom_R(M,N) \otimes^{\bf L}_{\Lambda} \Lambda_{\con}) \label{bimodcomplex}
\end{align}
which is a $\Gamma_{\con}$-$\Lambda_{\con}$-bimodule complex. If $M$ and $N$ are related by a single mutation at an indecomposable summand $M_i \in \add(M)$ then $\tT$ induces a standard equivalence $F_{_i} \colonequals \RHom_{\Lambda_{\con}}(\tT, -)$ such that the following diagram commutes.
\begin{center}\hspace{1cm}
\begin{tikzpicture}
  \matrix (m) [matrix of math nodes,row sep=3em,column sep=4em,minimum width=2em] {  
\Db(\Lambda_{\con})  & \Db(\Lambda)   \\
\Db(\Gamma_{\con})  & \Db(\Gamma)   \\};
 \path[-stealth]
    (m-1-1.east|-m-1-2) edge node [above] { $\scriptstyle - \otimes_{\Lambda_{\con}} \Lambda_{\con}$} (m-1-2) 
    (m-2-1.east|-m-2-2) edge node [above] {$\scriptstyle - \otimes_{\Gamma_{\!\con}} \Gamma_{\!\con}$} (m-2-2)  
 (m-1-1) edge node [left] {$\scriptstyle F_{_i} \colonequals \RHom_{\Lambda_{\con}}(\tT, -)$} (m-2-1)
 (m-1-2) edge node [right] {$ \scriptstyle G_{_i} \colonequals \RHom_{\Lambda}(\Hom_\Lambda(M,N), -)$} (m-2-2);

\end{tikzpicture}
\end{center}
\end{theorem}

This results also holds more generally for basic rigid objects but for maximal rigid objects, the functor $G_{i}$ is precisely the mutation functor between MMAs (and hence is isomorphic to the flop functor) and so we think of $F_{i}$ as also induced by the flop. 
Moreover, the commutative diagram is key to the rest of the paper as it makes it possible to understand the composition of the $F_i$. Composing directly would involve taking the derived tensor products of complexes of bimodules, which is difficult to compute, but with the commutative diagram we can use known results about the $G_i$ from \cite{faithful} to bypass these difficulties. 

Associated to each minimal model of $\Spec R$ (and hence to each maximal rigid object in $\ucmr$) there is a real hyperplane arrangement, and this provides the key topological data needed to control compositions. Although this data comes from the geometry (or equivalently from the MMAs), we show in Theorem \ref{intrinsic} that it can be recovered completely from the two-term tilting theory of the corresponding contraction algebra. 
The hyperplane arrangement produces a directed graph where the vertices can be viewed as maximal rigid objects and an arrow $M \to N$ corresponds to the derived equivalence $F_i$ between the corresponding contraction algebras as in Theorem \ref{introexplicit} (see \S \ref{composingsection} for details). Viewing paths in this combinatorial structure as the composition of arrows and hence of these functors, we obtain the following result about paths of shortest length. 

\begin{theorem}[Theorem \ref{composing}] \label{introcomposing}
Suppose $\Spec R$ is a complete local isolated cDV singuarity and $M \in \ucmr$ is a basic maximal rigid object with associated hyperplane arrangement $\cH$. If $N \in \ucmr$ is any other maximal rigid object then the composition along any positive minimal path from $M$ to $N$ is isomorphic to the functor $F_{MN} \colonequals \RHom(\tT, -)$, where $\tT$ is constructed as in \eqref{bimodcomplex}; in particular, the functor is independent of the minimal path chosen. 
\end{theorem}

An immediate consequence of this is that the functors $F_i$ satisfy higher length braid relations (see Corollary \ref{braidthm}). Moreover, this theorem provides the first standard derived equivalence between any two contraction algebras. We next show that the $F_{MN}$ are precisely the standard equivalences induced by two-term tilting complexes.
\begin{theorem}[Corollary \ref{twotermareatoms}] \label{twotermintro}
Suppose $\Spec R$ is a complete local isolated cDV singularity and $M \in \ucmr$ is a basic maximal rigid object. Set $\Lambda_{\con} \colonequals \uEnd_R(M)$. Then positive minimal paths starting at $M$ determine precisely (up to algebra isomorphism) the standard equivalences from $\Db(\Lambda_{\con})$ induced by two-term tilting complexes of $\Lambda_{\con}$.
\end{theorem} 

In the process of proving Theorem \ref{twotermintro}, we upgrade the bijection \cite[4.7]{[AIR]} between maximal rigid objects in $\ucmr$ and two-term tilting complexes of a contraction algebra to additionally include the two-sided two-term tilting complexes, and further show that the bijection preserves endomorphism algebras in this case (see Corollary \ref{twosidedctbij}).

A further consequence of Theorem \ref{introcomposing} is that there is a well-defined functor between two groupoids: one associated to the hyperplane arrangement and the other consisting of standard derived equivalences between contraction algebras (see \S\ref{composingsection} for details). In \S \ref{faithfulness}, we show this functor is faithful and hence obtain the following, purely algebraic, corollary.

\begin{theorem}[Corollary \ref{groupaction}] \label{groupactionintro}
Suppose $R$ is a complete local isolated cDV singularity and $M \colonequals \bigoplus_{i=1}^n M_i \in \ucmr$ is a basic maximal rigid object with associated real hyperplane arrangement $\cH$. Writing $\Lambda_{\con} \colonequals \uEnd_R(M)$, then there is an injective group homomorphism
\begin{align*}
\pi_1(\mathbb{C}^n \backslash \cH_{\mathbb{C}}) \to \mathrm{Auteq}(\Db(\Lambda_{\con}))
\end{align*}
where  $\cH_{\mathbb{C}}$ is the complexification of $\cH$.
\end{theorem}
This shows that the group structure on the subgroup of derived autoequivalences of a contraction algebra obtained by composing the $F_{i}$  is precisely the same as that of the flop functors (as seen in \cite[6.7]{faithful}). We then show in Theorem \ref{intrinsic} that $\cH$, and hence the group of flop functors $\pi_1(\mathbb{C}^n \backslash \cH_{\mathbb{C}})$, can be constructed from a given contraction algebra without any knowledge of the geometry. Amongst other things, this gives further evidence towards the conjecture \cite[1.3]{rigidequiv} of Donovan--Wemyss. 

Finally, by combining with the author's previous work \cite{rigidequiv}, we establish the following.
\begin{theorem}[Theorem \ref{summary}(2)]
Suppose $R$ is a complete local isolated cDV singularity and that $\Lambda_{\con} \colonequals \uEnd_R(M)$ is an associated contraction algebra. Then every standard derived equivalence from $\Lambda_{\con}$ (up to some algebra isomorphism) can be obtained as some composition of our constructed standard equivalences and their inverses. 
\end{theorem}

Forgetting the geometry, this is used in \S\ref{picture} to show that contraction algebras are interesting in their own right. The hyperplane arrangement $\cH$ constructed from two term tilting theory can be viewed as a complete picture of the derived equivalence class: the two-term tilting complexes sit in the chambers with their endomorphism rings providing all basic members of the derived equivalence class. Further, paths, along with an equivalence relation determined by the groupoid associated to $\cH$, control all the derived equivalences between them.

\subsection{Conventions} Throughout, $k$ will denote an algebraically closed field of characteristic zero. For a ring $A$, we write $\Mod A$ for the category of right $A$-modules and $\mod A$ for the category of finitely generated right $A$-modules. For $M \in \mod A$, we let $\add M$ be the full subcategory consisting of summands of finite direct sums of copies of $M$ and we let $\proj A \colonequals \add A$ be the category of finitely generated projective modules. Further, write $\Kb(\proj A)$ for the homotopy category of bounded complexes of finitely generated projectives, $\Db(A) \colonequals \Db(\mod A)$ for the bounded derived category of $\mod A$ and $\Dnb(A) \colonequals \Dnb(\Mod A)$ for the derived category of $\Mod A$. Finally, we denote by $(-)^*$ the map $\Hom_A(-,A) \colon \mod A \to \mod A^{\mathrm{op}}$ and by $D$ the map $\Hom_k(-,k) \colon \mod A \to \mod A^{\mathrm{op}}$.

\subsection{Acknowledgments}

While this work was carried out, the author was a student at the University of Edinburgh and the material contained in this paper formed part of her PhD thesis. The author would like to thank her supervisor Michael Wemyss for his helpful guidance and the Carnegie Trust for the Universities of Scotland for their financial support. Further the author would like to thank the Max Planck Institute for Mathematics in Bonn for their support during edits to this paper. 

\section{Preliminaries}
\subsection{General Setting}
Throughout this paper, we will restrict our attention to the following 3-fold singularities.
\begin{definition}
A three dimensional complete local $\mathbb{C}$-algebra $R$ is a \emph{complete local compound Du Val (cDV) singularity} if $R$ is isomorphic to 
\begin{align*}
\mathbb{C}\llbracket u,v,x,y \rrbracket / (f(u,v,x) + yg(u,v,x,y)) 
\end{align*}
where $\mathbb{C}\llbracket u,v,x \rrbracket /(f(u,v,x))$ is a Du Val surface singularity and $g$ is arbitrary. 
\end{definition}
\noindent Associated to such a singularity is the following category.
\begin{definition} 
Let $(R, \mathfrak{m})$ be a commutative noetherian local ring and choose $M \in \mod R$. The depth of $M$ is defined to be
\begin{align*}
\mathrm{depth}_R(M)=\mathrm{min}\{ i \geq 0 \mid \Ext^i_R(R/ \mathfrak{m},M) \neq 0\}.
\end{align*}
We say $M$ is \emph{maximal Cohen--Macaulay} $(\mathrm{CM})$ if $\mathrm{depth}_R(M)=\mathrm{dim}(R)$ and we write $\cmr$ for the full subcategory of $\mod R$ consisting of maximal Cohen--Macaulay modules.
\end{definition}
For each pair $M,N$ of objects in $\cmr$, define $\uHom_R(M,N)$ to be $\Hom_R(M,N)$ factored out by the set of morphisms which factor through $\proj (\cmr)=\add(R)$. The stable category of $\cmr$, denoted $\ucmr$, is then defined to have the same objects as $\cmr$, but where 
\begin{align*}
\Hom_{\ucmr}(M,N) \colonequals \uHom_R(M,N).
\end{align*} 
In particular, any projective module in $\cmr$ is isomorphic to zero in $\ucmr$. Given an object $M \in \cmr$, a \emph{syzygy} of $M$, denoted $\Omega M$, is obtained by taking the kernel of a surjective morphism $R^n \to M$. This is only defined up to projective summands, but as such, it is well-defined on $\ucmr$ and in fact, induces a functor $\Omega \colon \ucmr \to \ucmr$ which is an equivalence. 
The following theorem summarises known results about $\cmr$ and its stable category. For details, and full references, see e.g.\ \cite[\S1]{symmetric}.
\begin{proposition}   \label{prop} 
If $R$ is a complete local isolated cDV singularity, $\cmr$ is a Frobenius category. Moreover, the stable category $\ucmr$ is a Krull-Schmidt, Hom-finite, 2-Calabi-Yau triangulated category with shift functor $\Omega^{-1}$ satisfying $\Omega^2 \cong \id$. 
\end{proposition}
\noindent In particular, it is easy to see that there are isomorphisms
\begin{align*}
\Ext_R^1(M,\Omega N) \cong \uHom_R(M, N) \cong  \Ext_R^1(\Omega M, N)
\end{align*}
which we will use freely throughout.
\begin{definition}
Let $\cC = \cmr$ or $\ucmr$. An object $M \in \cC$ is called:
\begin{enumerate}[leftmargin=0.3cm,itemindent=.6cm,labelwidth=\itemindent,labelsep=0cm,align=left]
\item \emph{rigid} if $\Ext^1_R(M,M)=0$;
\item \emph{maximal rigid} if $M$ is rigid and if $M \oplus X$ is rigid for some $X \in \cC$, then $X \in \add(M)$.
\end{enumerate}
\end{definition}
As $R$ is both projective and injective in $\cmr$, it is clear that any maximal rigid object in $\cmr$ must contain $R$ as a summand. Further, $M$ is a basic maximal rigid object in $\ucmr$ if and only if $R \oplus M$ is a basic maximal rigid object in $\cmr$. We will consider two algebras associated to each maximal rigid object.
\begin{definition}
Let $M \in \cmr$ be a basic maximal rigid object.
\begin{enumerate}[leftmargin=0.3cm,itemindent=.6cm,labelwidth=\itemindent,labelsep=0cm,align=left]
\item The \emph{maximal modification algebra} (MMA) associated to $M$ is $\End_R(M)$.
\item The \emph{contraction algebra} associated to $M$ is $\uEnd_R(M)$.
\end{enumerate}
\end{definition}
Although MMAs are known to control all the geometry of $\Spec R$, they are very large and it is conjectured that the contraction algebra, which in contrast is finite dimensional by Proposition \ref{prop}, should also be enough \cite[1.3]{rigidequiv}. A further property of contraction algebras is the following.
\begin{theorem} \cite[7.1]{symmetric} \label{symmalg}
Let $\Spec R$ be a complete local isolated cDV singularity. Then for any $N \in \cmr$, the algebra $\Lambda \colonequals \uEnd_{R}(N)$ is a symmetric algebra i.e.\ $\Lambda \cong D\Lambda$ as $\Lambda$-$\Lambda$-bimodules.
\end{theorem}
The goal of this paper is to fully understand the derived equivalences between the contraction algebras of $\Spec R$ and in particular, their derived autoequivalence groups. 
\subsection{Mutation} \label{mutation}
Rigid objects are largely studied for their mutation properties which we describe here, following \cite[\S6]{mmas}. To begin, we require the following.
Suppose that $\cD$ is an additive category and that $\cS$ is a class of objects in $\cD$.
\begin{enumerate}
\item A morphism $f \colon X \to Y$ is called a \emph{right $\cS$-approximation} of $Y$ if $X \in \cS$ and the induced morphism $\Hom(Z,X) \to \Hom(Z,Y)$ is surjective for any $Z \in \cS$.
\item A morphism $f \colon X \to Y$ is said to be \emph{right minimal} if for any $g \colon X \to X$ such that $f \circ g= f$, then $g$ must be an isomorphism.
\item A morphism $f \colon X \to Y$ is a \emph{minimal right $\cS$-approximation} if $f$ is both right minimal and a right $\cS$-approximation of $Y$.
\end{enumerate}
There is also the dual notion of a left minimal $\cS$-approximation. Now let $M=\bigoplus\limits_{i=0}^nM_i$ be a basic rigid object in $\cmr$ with $M_0 \cong R$. To mutate at $M_i$, where $i \neq 0$, take a minimal right $\add((M/M_i)^*)$-approximation 
\begin{align*}
V_i^* \xrightarrow{a_i} M_i^*,
\end{align*}
and let $J_i\colonequals\Kernel \ a_i$ so that there is an exact sequence
\begin{align}
0 \to J_i \xrightarrow{c_i} V_i^* \xrightarrow{a_i} M_i^*. \label{partexchange1}
\end{align}
Define $K_i \colonequals J_i^*$ and $\upnu_iM \colonequals M/M_i \oplus K_i$ which is again a rigid object in $\cmr$ \cite[6.10, 5.12]{mmas}. Further, if $M$ is maximal rigid then $\upnu_iM$ is also maximal rigid \cite[6.10, 5.12]{mmas}.
\begin{lemma} \label{exactexchnage}
The sequence \eqref{partexchange1} induces an exact sequence 
\begin{align}
0 \to M_i \xrightarrow{b_i} V_i \xrightarrow{d_i} K_i \to 0 \label{exchangesequence1}
\end{align}
such that $d_i$ is a minimal right $\add(M/M_i)$-approximation of $K_i$ and $b_i$ is a minimal left $\add(M/M_i)$-approximation of $M_i$.
\end{lemma}
\begin{proof}
As $R \in \add (M)$ and $M$ is basic, it must be that $R \in \add((M/M_i)^*)$ and hence applying $\Hom_R(R,-)$ to \eqref{partexchange1} shows that 
\begin{align} \label{exchangeproof}
0 \to J_i \xrightarrow{c_i} V_i^* \xrightarrow{a_i} M_i^* \to 0
\end{align}
is exact. 
Further, by \cite[6.4(2)]{mmas}, $c_i$ is a minimal left $\add((M/M_i)^*)$-approximation. Thus, applying $\Hom_R(-,R)$ to \eqref{exchangeproof} gives an exact sequence
\begin{align*}
0 \to M_i^{**} \xrightarrow{a_i^*} V_i^{**} \xrightarrow{c_i^*} K_i \to 0.
\end{align*} 
However, $M_i$ and $V_i$ are both rigid and hence are reflexive by \cite[5.12]{mmas} (meaning that the natural map $M_i \to M_i^{**}$ is an isomorphism and similarly for $V_i$). Thus, we get an exact sequence
 \begin{align*}
0 \to M_i \xrightarrow{a_i^*} V_i \xrightarrow{c_i^*} K_i \to 0.
\end{align*} 
As $\Ext^1_R(M/M_i, M_i)=0$ by the rigidity of $M$, it is clear that $d_i \colonequals c_i^*$ is a right $\add(M/M_i)$-approximation. Similarly, $b_i \colonequals a_i^*$ is a left $\add( M/M_i)$-approximation if $\Ext^1_R(J_i^*, M/M_i)=0$ which holds by the rigidity of $\upnu_i M$. Minimality of both follows from minimality of $a_i$ and $c_i$.
\end{proof}
We refer to the sequence \eqref{exchangesequence1} as the \textit{exchange sequence} defining $\upnu_iM$.
\begin{remark}
Mutation could also have been defined in $\ucmr$ using exchange triangles, as in \cite[1.10]{[AIR]}. However, it is easy to show that such an exchange triangle is induced by the exchange sequence \eqref{exchangesequence1} and so gives the same result upon mutation.
\end{remark}
For two basic maximal rigid objects $M,N \in \ucmr$, it is known that $M \cong  \upnu_i N$ for some $i \in \{1, \dots, n\}$ if and only if $M$ and $N$ differ by exactly one indecomposable summand \cite[5.3]{yoshino}. Thus, for these objects we have $\upnu_i \upnu_i M \cong M$ and we can use the following graph to record all the information about maximal rigid objects and their mutation.
\begin{definition}
The \emph{mutation graph} of maximal rigid objects in $\ucmr$ has the basic maximal rigid objects as vertices and an edge between two vertices $M$ and $N$ if they differ by exactly one indecomposable summand.
\end{definition}

\subsection{Relationship to Minimal Models}
For a complete local isolated cDV singularity $\Spec R$, a \emph{minimal model} is a certain crepant projective birational morphism $f \colon X \to \Spec R$ where $X$ need not be smooth but instead must be $\mathbb{Q}$-factorial terminal (see e.g.\ \cite[\S2]{HomMMP}). With this definition, $X$ is considered to be a `nicest' member of the birational equivalence class of $\Spec R$. The following are well known.
\begin{enumerate}
\item \cite{reid} Any minimal model $f \colon X \to \Spec R$ is an isomorphism away from the unique singular point of $\Spec R$, where the preimage, called the \textit{exceptional locus}, consists of a finite chain of curves.
\item \cite{[KM]} There are finitely many minimal models of $\Spec R$.
\item \cite{[K]} Any two minimal models of $\Spec R$ are connected by a sequence of \emph{simple flops}; certain codimension two modifications.  
\end{enumerate}
To explain $(3)$ in more detail, choose a curve $C_i$ in the exceptional locus of a minimal model $f \colon X \to \Spec R$. Since $R$ is complete local, $f$ can be factorised as 
\begin{align*}
X \xrightarrow{g} X_{\con} \xrightarrow{h} \Spec R
\end{align*}
where $g(C_i)$ is a single point and $g$ is an isomorphism elsewhere. For any such factorisation, there exists a certain birational map $g^+ \colon X^+ \to X_{\con}$, satisfying some technical conditions detailed in \cite[2.6]{HomMMP}, which fits into a commutative diagram
\begin{center}
\begin{tikzpicture}[scale=1.1]
\node (C1) at (0, 0) {\scriptsize{$X_{\con}$}};
\node (C2) at (0, -1.2) {\scriptsize{$\Spec R$}};
\node (C3) at (-1.2, 1.2) {\scriptsize{$X$}};
\node (C4) at (1.2, 1.2)    {\scriptsize{$X^+$}};
\node at (0.55, 0.71)    {\scriptsize{$g^+$}};
\node at (-0.55, 0.71)    {\scriptsize{$g$}};
\node at (-0.1, -0.5)    {\scriptsize{$h$}};
\draw[->] (C1) -- (C2); 
\draw[->] (C3) to node[left] {$\scriptstyle f$} (C2);
\draw[->] (C4) to node[right] {$\scriptstyle f^+$} (C2);
\draw[->] (C3) -- (C1);
\draw[->] (C4) -- (C1);
\draw[->,dashed] (C3) to node[above] {$\scriptstyle \phi$} (C4);
\end{tikzpicture}
\end{center} 
where $\phi$ is a birational equivalence (see e.g.\ \cite[p25]{kollar} or \cite[\S 2]{Schroer}). We call $f^+ \colon X^+ \to \Spec R$ the \emph{simple flop} of $f$ at the curve $C_i$. In this case, there are derived equivalences
\begin{align}
 \Db(\coh X) \to \Db(\coh X^+) \quad \text{and} \quad  \Db(\coh X^+) \to \Db(\coh X), \label{flops}
\end{align}
known as \emph{Bridgeland--Chen flop functors} \cite{bridgeland, chen}. 
\begin{definition}
Given a complete local isolated cDV singularity, define the \emph{simple flops graph} to have a vertex for each minimal model and an edge between two vertices if the corresponding minimal models are connected by a simple flop.
\end{definition}
The following is one of the fundamental theorems in the Homological Minimal Model Programme as it provides the link between the study of rigid objects and the geometry.
\begin{theorem}\cite[4.10]{HomMMP} \label{HMMPbij}
Let $\Spec R$ be a complete local isolated cDV singularity. Then there is a bijection
\begin{align*}
\{\text{basic maximal rigid objects in $\cmr$}\} \longleftrightarrow \{\text{minimal models $f \colon X \to \Spec R$}\}.
\end{align*}
Moreover,
\begin{enumerate}
\item There is a one-to-one correspondence between summands of a maximal rigid object and curves in the exceptional locus of the corresponding minimal model;
\item The mutation graph of maximal rigid objects coincides with the simple flops graph of the minimal models.
\end{enumerate}
\end{theorem}
In particular, this theorem shows that basic maximal rigid objects in $\cmr$ do exist and further, there are finitely many such objects.

\subsection{Derived Equivalences}
It is a well known result of Rickard \cite{morita} that two $k$-algebras $\Lambda$ and $\Gamma$ are derived equivalent if and only if there exists a tilting complex $T \in \Kb(\proj \Lambda)$ such that $\End_{\Kb(\proj \Lambda)}(T) \cong \Gamma$. 
\begin{definition}
For a $k$-algebra $\Lambda$, we say a complex $T \in \Kb(\proj \Lambda)$ is tilting if:
\begin{enumerate}
\item $\Hom_{\Kb(\proj \Lambda)}(T,T[n])=0$ for all $ n \neq 0$.
\item $\thick (T)= \Kb(\proj \Lambda)$, where $\thick (T)$ denotes the smallest triangulated full subcategory of $\Kb(\proj \Lambda)$ containing $T$ and closed under forming direct summands.
\end{enumerate}
\end{definition}
Although such a tilting complex $T$ ensures the existence of an equivalence from $\Db(\Lambda)$ to $\Db(\End_{\Kb(\proj \Lambda)}(T))$, to obtain a natural functor requires extra structure on the tilting complex, as in part $(3)$ of the following.
\begin{theorem} \cite[8.1.4]{keller} \label{kel}
Suppose $\Lambda$ and $\Gamma$ are two rings and $X$ is a complex of $\Gamma$-$\Lambda$-bimodules. The following are equivalent:
\begin{enumerate}
\item $- \otimes^{\bf L}_\Gamma X \colon \Dnb(\Gamma) \to \Dnb(\Lambda)$  is an equivalence;
\item $- \otimes^{\bf L}_\Gamma X \colon \Kb(\proj \Gamma) \to \Kb(\proj \Lambda)$ is an equivalence;
\item Viewing $X$ as a complex of right $\Lambda$-modules; \begin{enumerate}[label=\normalfont(\alph*)]
\item The map $\Gamma \to \Hom_{\Dnb(\Lambda)}(X,X)$ induced by $- \otimes^{\bf L}_\Gamma X$ is an isomorphism and further $\Hom_{\Dnb(\Lambda)}(X,X[n]) =0$ for $n \neq 0$.
\item $X$ is quasi-isomorphic to some complex $T$ in $\Kb(\proj \Lambda)$.
\item $\thick (T)= \Kb(\proj \Lambda)$.
\end{enumerate}
\end{enumerate} 
\end{theorem}
In this case, $X$ is called a \textit{two-sided} tilting complex and it induces a \textit{standard equivalence}
\begin{align*}
\RHom_\Lambda(X,-) \colon \Dnb(\Lambda) \to \Dnb(\Gamma)
\end{align*}
with inverse $ - \otimes^{\bf L}_\Gamma X$. Note that the conditions in part $(3)$ ensure that any such $X$ is quasi-isomorphic to a tilting complex for $\Lambda$. 
Keller further showed that for any tilting complex $T \in \Kb(\proj \Lambda)$, there exists a two-sided tilting complex $\tT$ such that $\tT$ is quasi-isomorphic to $T$ in $\Db(\Lambda)$ \cite[8.3.1]{keller}. As $\tT$ is two-sided, there is a standard equivalence
 \begin{align*}
\RHom_\Lambda(\tT,-) \colon \Dnb(\Lambda) &\to \Dnb(\Gamma)
\end{align*}
which maps $T \mapsto \Gamma$. It was further shown in \cite[2.1]{kellerhomotopy} that such a $\tT$ is unique in a suitable sense. For our purposes, the following suffices.
\begin{proposition}\cite[2.3]{rouquier} \label{rouq}
Suppose that $\Lambda$, $\Gamma$ and $\Gamma'$ are $k$-algebras and that $\tT$ (respectively $\tT'$) is a two-sided tilting complex of $\Lambda$-$\Gamma$-bimodules (respectively $\Lambda$-$\Gamma'$-bimodules) with $\End_\Lambda(\tT) \cong \End_\Lambda(\tT')$. Then $\tT \cong \tT'$ in $\Db(\Lambda)$ if and only if there exists an isomorphism $\upgamma \colon \Gamma \to \Gamma'$ such that 
\begin{align*}
\tT \cong {}_\upgamma \Gamma' \otimes_{\Gamma'} \tT'
 \end{align*} 
 in the derived category of $\Gamma$-$\Lambda$ bimodules.
\end{proposition}
In this way, we can say any tilting complex $T \in  \Kb(\proj \Lambda)$ induces a unique (up to algebra isomorphism) standard equivalence. 

As well as these general results about derived equivalences, we will also make use of the following known result about derived equivalences of MMAs.

\begin{theorem} \cite[4.17, 6.14]{mmas} \label{demma}
Let $\Spec R$ be a complete local isolated cDV singularity and suppose that $M \colonequals \bigoplus\limits_{i=0}^n M_i \in \cmr$ is a basic rigid object with $M_0 \cong R$. Writing $\Lambda \colonequals \End_R(M)$, the following statements hold.
\begin{enumerate}
\item For any $i \neq 0$, the bimodule $\Hom_R(M,\upnu_i M)$ is tilting of projective dimension one which gives rise to a standard equivalence
\begin{align*}
G_i \colonequals \RHom_\Lambda(\Hom_R(M,\upnu_i M),-) \colon \Db(\Lambda) \xrightarrow{\iso} \Db(\End_R(\upnu_i M)).
\end{align*} 
\item If further $M$ is a maximal rigid object then, for any other maximal rigid object $N$, the bimodule $\Hom_R(M,N)$ is tilting of projective dimension one which gives rise to a standard derived equivalence
\begin{align*}
\RHom_\Lambda(\Hom_R(M,N),-) \colon \Db(\Lambda) \xrightarrow{\iso} \Db(\End_R(N)).
\end{align*} 
\end{enumerate}
\end{theorem}
We wish to study the derived equivalences coming from the geometry, namely the Bridgeland--Chen flop functors \eqref{flops} induced by flops. However, it is shown in \cite[4.2]{HomMMP}, that, if $M$ and $\upnu_iM$ correspond to $X$ and $X^+$ respectively, then $G_i$ is isomorphic to the inverse of the flop functor from \eqref{flops}. Thus, to study the flop functors, it is equivalent to study the equivalences $G_i$ between the MMAs. 

As contraction algebras are symmetric by Theorem \ref{symmalg}, it is well known that they have no tilting modules and thus we have to look for tilting complexes of higher length. The bijection of \cite[4.7]{[AIR]} provides a method of construction for those of length two.
\begin{theorem} \label{decon}
Let $\Spec R$ be a complete local isolated cDV singularity and suppose that $M \colonequals \bigoplus_{i=0}^n \in \cmr$ is a basic rigid object with $M_0 \cong R$. Set $\Lambda_{\con} \colonequals \uEnd_R(M)$ and, for any $i \neq 0$, mutate $M$ at $M_i$ via the exchange sequence \eqref{exchangesequence1}. Then the complex
\begin{align*}
P \colonequals \big(\uHom_R(M,M_i) \xrightarrow{b_i \circ -} \uHom_R(M,V_i)\big) \oplus \big(0 \to \bigoplus_{j\neq i} \uHom_R(M,M_j)\big)
\end{align*}
is a tilting complex for $\Lambda_{\con}$.
\end{theorem}
\begin{proof}
The exchange sequence \eqref{exchangesequence1} descends to a triangle
\begin{align*}
M_i \xrightarrow{b_i} V_i \xrightarrow{d_i} K_i \to \Omega^{-1} M_i
\end{align*} 
in $\ucmr$ where $d_i$ is a minimal right $\add(M/M_i)$-approximation by Lemma \ref{exactexchnage}. An easy check using that $M$ is rigid shows further that $d_i$ is a minimal right $\add(M)$-approximation and hence in the triangle 
\begin{align*}
M_i \xrightarrow{\scriptsize \begin{pmatrix} b_i \\  0 \end{pmatrix}} V_i \oplus M/M_i \xrightarrow{\scriptsize \begin{pmatrix} d_i & 0 \\ 0 & \id \end{pmatrix}} K_i \oplus M/M_i \to \Omega^{-1} M_i,
\end{align*}
 the map $ \scriptsize\begin{pmatrix} d_i & 0 \\ 0 & \id \end{pmatrix}$ is a minimal right $\add(M)$-approximation of $\upnu_i M \colonequals  K_i \oplus M/M_i$. Thus, by the bijection of \cite[2.18]{rigidequiv} (a generalisation of \cite[4.7]{[AIR]}), the complex $P$ is a silting complex. As \cite[2.8]{siltingmutation} shows any silting complex for a symmetric algebra is in fact tilting, and $\Lambda_{\con}$ is symmetric by Theorem \ref{symmalg}, this completes the proof.
\end{proof}
\begin{remark}
It is further known that $\End_{\Kb(\proj \Lambda_{\con})}(P) \cong \uEnd_R(\upnu_iM)$ and thus $\Lambda_{\con}$ and $ \uEnd_R(\upnu_iM)$ are derived equivalent \cite[4.1, 5.5]{Dugas}. However, we will not use this and will actually recover this result in \S\ref{explicit}.
\end{remark}
Although the complex $P$ is tilting, it has no two-sided structure and hence to obtain a standard equivalence, and so understand the derived equivalences of a contraction algebra, we must lift $P$ to a two-sided tilting complex. 
\section{Standard Equivalences of Contraction Algebras} \label{explicit} 
In this section, we explicitly construct standard derived equivalences between contraction algebras which then, in later sections, we will show how to compose. For this, we require the following setup where $\ref{setupcom}(2)$ is our key new object. 

Let $\Spec R$ be a complete local isolated cDV singularity and suppose that $M \colonequals \bigoplus\limits_{i=0}^n M_i$ is a basic rigid object in $\cmr$ with $M_0 \cong R$. Recall from \S\ref{mutation} that we can mutate $M$ at the summand $M_i$, where $i \neq 0$, via an exchange sequence, 
\begin{align}
0 \to M_i \xrightarrow{b_i} V_i \xrightarrow{d_i} K_i \to 0, \label{exseq}
\end{align}
to get $\upnu_iM \colonequals M/M_i \oplus K_i$. Consider the following algebras:
\[
\begin{array}{ccc}
\Lambda \colonequals \End_R(M) & \ & \Gamma \colonequals \End_R(\upnu_iM)\\
\Lambda_{\con} \colonequals \uEnd_R(M) & \ & \Gamma_{\!\con} \colonequals \uEnd_R(\upnu_iM).
\end{array}\]

\begin{setup} \label{setupcom}
With notation as above, set:
\begin{enumerate}
\item $T_{i} \colonequals \Hom_R(M, \upnu_i M)$ which, by Theorem \ref{demma}, is a $\Gamma$-$\Lambda$-bimodule giving an equivalence
\begin{align*}
- \otimes_\Gamma^{\bf L} T_i \colon \Db(\Gamma) \to \Db(\Lambda);
\end{align*} 
\item $\tT_{\!i} \colonequals \tau_{\scriptscriptstyle{ \geq -1}} ( \Gamma_{\!\con} \otimes^{\bf L}_\Gamma T_i \otimes^{\bf L}_{\Lambda} \Lambda_{\con})$ which is a complex of $\Gamma_{\!\con}$-$\Lambda_{\con}$-bimodules.
\end{enumerate}
\end{setup}
\noindent Here, $\tau_{\scriptscriptstyle \geq-1}$ is the truncation functor taking a complex 
\begin{align*}
X \colonequals \dots \to X_{i-1} \xrightarrow{d_{i-1}} X_i \xrightarrow{d_i} X_{i+1} \to \dots 
\end{align*}
to the complex
\begin{align*}
 \dots 0 \to 0 \to X_{-1}/\Image(d_{-2}) \xrightarrow{d_{-1}} X_{0} \xrightarrow{d_0} X_{1} \to \dots. 
\end{align*}
Note that, if $X$ has zero homology in degrees $-2$ and lower, then there is a quasi-isomorphism $X \to \tau_{\scriptscriptstyle \geq -1}X$.\\

In section \ref{proofofthm1}, we will establish the following theorem.
\begin{theorem} \label{square}
With the setup of \ref{setupcom}, there is an isomorphism 
\begin{align*}
\Gamma_{\!\con} \otimes^{\bf L}_{\Gamma} T_{i} \cong \tT_{\!i} \otimes_{\Lambda_{\con}} {\Lambda_{\con}}_{\Lambda}
\end{align*}
 in the derived category of $\Gamma_{\!\con}$-$\Lambda$-bimodules. Consequently, there is a commutative diagram
\begin{equation}\label{commutingsquare}
\begin{array}{c}\begin{tikzpicture} 
  \matrix (m) [matrix of math nodes,row sep=3em,column sep=4em,minimum width=2em] {  
\Dnb(\Lambda_{\con})  & \Dnb(\Lambda)   \\
\Dnb(\Gamma_{\!\con})  & \Dnb(\Gamma)   \\};
 \path[-stealth]
    (m-1-1.east|-m-1-2) edge node [above] {$\scriptstyle - \otimes_{\Lambda_{\con}} \Lambda_{\con}$ } (m-1-2)
    (m-2-1.east|-m-2-2) edge node [above] {$ \scriptstyle - \otimes_{\Gamma_{\!\con}} \Gamma_{\!\con}$} (m-2-2)
 (m-2-1) edge node [left] {$\scriptstyle  -\otimes^{\bf L}_{\Gamma_{\!\con}} \scalebox{0.9}{$\mathscr{T}$}_{\!i}$} (m-1-1)
 (m-2-2) edge node [right] {$ \scriptstyle - \otimes^{\bf L}_\Gamma T_i$} (m-1-2);
\end{tikzpicture}
\end{array}
\end{equation}
where the functor on the right hand side is an equivalence.
\end{theorem}
In section \ref{equivproof}, the diagram \eqref{commutingsquare} will be used to prove the following, which shows that $\tT_{\!i}$ is in fact a two-sided tilting complex, inducing an equivalence between contraction algebras.
\begin{corollary} \label{equivalence}
With the setup of \ref{setupcom}, the functor $  -\otimes^{\bf L}_{\Gamma_{\!\con}} \tT_{\!i} \colon \Db(\Gamma_{\!\con}) \to \Db(\Lambda_{\con})$ is an equivalence.
\end{corollary}
In summary, each mutation of maximal rigid objects induces a standard derived equivalence between the corresponding contraction algebras which, moreover, can be thought of as a contraction algebra analogue of the flop functor. 
\subsection{One-sided Results} \label{onesidedsection}
To establish the isomorphism of Theorem \ref{square} in the category of bimodules, we first need to establish the isomorphism as complexes of $\Lambda$-modules. The general idea is to show that both $\Gamma_{\!\con} \otimes^{\bf L}_\Gamma T_i$ and $\tT_i$ are quasi-isomorphic to the $\Lambda_{\con}$-tilting complex $P$ from Theorem \ref{decon}. As the proofs of the following three results are largely computational, most are postponed to the appendix. 
 
\begin{proposition} \label{onesided}
In the setup of \ref{setupcom}, $ \Gamma_{\!\con} \otimes^{\bf L}_{\Gamma} T_i \cong P_{\Lambda}$ in $\Db(\Lambda)$, where $P$ is as in Theorem \ref{decon}. In particular, the homology of $ \Gamma_{\!\con} \otimes^{\bf L}_{\Gamma} T_i$ is isomorphic as right $\Lambda$-modules to
\begin{align*}
\left\{
	\begin{array}{ll}
		 \uHom_R(M,\upnu_i M) & \mbox{in degree 0; } \\
		\uHom_R(M, \Omega K_i) & \mbox{in degree -1; } \\
		0 & \mbox{elsewhere.}
	\end{array}
\right.
\end{align*}
\end{proposition}

\begin{proof}
The first statement is \ref{appenmain}, proved in the appendix. For the second statement it is enough to show the homology of
\begin{align*}
\uHom_R(M,M_i) \xrightarrow{b_i \circ -} \uHom_R(M,V_i)
\end{align*}
is $\uHom_R(M,K_i)$ in degree $0$ and $\uHom_R(M,\Omega K_i)$ in degree $-1$. This is shown at the beginning of the proof of \ref{quasiso}.  
\end{proof}
As well as giving an explicit form for $ \Gamma_{\!\con} \otimes^{\bf L}_{\Gamma} T_i$, Proposition \ref{onesided} can further be used to find an explicit form for $\tT_i$.
\begin{proposition} \label{onesided2}
Under the setup of \ref{setupcom}, $\tT_{\!i} \cong P$ in $\Db(\Lambda_{\con})$, where $P$  is as in Theorem \ref{decon}.
\end{proposition}
\begin{proof}
In \ref{defS}, a complex $\EuScript{P}$ of projective $\Lambda$-modules is constructed which is isomorphic to $P$ in $\Db(\Lambda)$.  Moreover, for any summand $M_j$ of $M$, it is easily checked there is an isomorphism 
\begin{align*}
\uHom_R(M,M_j) \cong \Hom_R(M,M_j) \otimes_\Lambda \Lambda_{\con}
 \end{align*} 
 and thus it is clear, using the explicit form of $\EuScript{P}$, that
\begin{align*}
 P \otimes^{\bf L}_{\Lambda} \Lambda_{\con} \cong \EuScript{P} \otimes_{\Lambda} \Lambda_{\con} \cong  P \oplus P[3]
\end{align*}
 in $\Db(\Lambda_{\con})$. Combining this with Proposition \ref{onesided}, there are isomorphisms
\begin{align*}
\Gamma_{\!\con} \otimes^{\bf L}_{\Gamma} T_i \otimes^{\bf L}_{\Lambda} \Lambda_{\con} \cong  P \otimes^{\bf L}_{\Lambda} \Lambda_{\con} \cong P \oplus P[3], 
\end{align*}
all in $\Db(\Lambda_{\con})$. Applying $\tau_{\scriptscriptstyle \geq -1}$ to both sides, using that $P$ has non-zero terms only in degrees $0$ and $-1$, gives the result. 
\end{proof}
\begin{corollary} \label{dimension}
Under the setup of \ref{setupcom}, $\Gamma_{\!\con} \otimes^{\bf L}_{\Gamma} T_i \cong \tT_{\!i} \otimes_{\Lambda_{\con}} {\Lambda_{\con}}_{\Lambda}$ in $\Db(\Lambda)$. In particular, the vector space dimension of the homologies of $\tT_{\!i}$ and $ \Gamma_{\!\con} \otimes^{\bf L}_{\Gamma} T_i$ agree in each degree, are finite dimensional, and are zero outside degrees $-1$ and $0$.
\end{corollary}
\begin{proof}
By Propositions \ref{onesided} and \ref{onesided2}, both  $\Gamma_{\!\con} \otimes^{\bf L}_{\Gamma} T_i$ and $\tT_{\!i} \otimes_{\Lambda_{\con}} {\Lambda_{\con}}_{\Lambda}$ are isomorphic to the the complex $P_\Lambda = P \otimes_{\Lambda_{\con}} {\Lambda_{\con}}_{\Lambda}$ in $\Db(\Lambda)$. Since $- \otimes_{\Lambda_{\con}} {\Lambda_{\con}}_{\Lambda}$ preserves homology, the second statement then follows easily.
\end{proof}

\subsection{Proof of Theorem \ref{square}} \label{proofofthm1}
Recall from Corollary \ref{dimension} that $\Gamma_{\!\con} \otimes^{\bf L}_{\Gamma} T_i$ and $\tT_{\!i} \otimes_{\Lambda_{\con}} {\Lambda_{\con}}_{\Lambda}$ have the same finite dimensional homology in each degree (which is zero outside of degrees $0$ and $-1$). Therefore, to prove Theorem \ref{square}, it is enough to show there is a map in the derived category of $\Gamma_{\!\con}$-$\Lambda$-bimodules between these two complexes which is injective on homology. 

Suppose that
\begin{align}
Q\colonequals \dots \to Q_{-3} \xrightarrow{d_{-3}} Q_{-2} \xrightarrow{d_{-2}} Q_{-1} \xrightarrow{d_{-1}} Q_0 \to 0 \label{Q}
\end{align}
is a complex of $\Gamma_{\!\con}$-$\Lambda$-bimodules, projective as $\Lambda$-modules, which is quasi-isomorphic to $\Gamma_{\!\con} \otimes^{\bf{L}}_{\Gamma} T_i$. For example, taking the Cartan--Eilenberg resolution of $\Gamma_{\!\con} \otimes^{\bf{L}}_{\Gamma} T_i$ in the derived category of $\Gamma_{\!\con}$-$\Lambda$-bimodules would suffice. 

Since the $Q_i$ are projective as $\Lambda$-modules, the complex $\Gamma_{\!\con} \otimes^{\bf{L}}_{\Gamma} T_i \otimes_{\Lambda}^{\bf{L}} \Lambda_{\con}$ is simply 
\begin{align*}
\dots \to Q_{-3} \otimes_{\Lambda} \Lambda_{\con} \xrightarrow{d_{-3} \otimes \id} Q_{-2} \otimes_{\Lambda} \Lambda_{\con} \xrightarrow{d_{-2} \otimes \id} Q_{-1} \otimes_{\Lambda} \Lambda_{\con} \xrightarrow{d_{-1} \otimes \id} Q_0 \otimes_{\Lambda} \Lambda_{\con} \to 0.
\end{align*}

There are natural maps $\partial_i \colon Q_i \to Q_i \otimes_{\Lambda} \Lambda_{\con}$ given by $q \mapsto q \otimes 1$ and these induce a map of complexes $Q \to Q \otimes_\Lambda \Lambda_{\con}$. Composing this map with the natural map from $Q \otimes_\Lambda \Lambda_{\con}$ to the truncation $\tau_{\geq -1}(Q \otimes_\Lambda \Lambda_{\con})$ gives the following map of complexes: 
\begin{center}
\begin{tikzpicture}
  \matrix (m) [matrix of math nodes,row sep=3em,column sep=2em,minimum width=2em] {  
\dots & Q_{-3} & Q_{-2} & Q_{-1} & Q_0 & 0  \\
\dots & Q_{-3} \otimes_{\Lambda} \Lambda_{\con} & Q_{-2} \otimes_{\Lambda} \Lambda_{\con} & Q_{-1} \otimes_{\Lambda} \Lambda_{\con} & Q_0 \otimes_{\Lambda} \Lambda_{\con} & 0 \\
\dots & 0 & 0 & \frac{Q_{-1} \otimes_{\Lambda} \Lambda_{\con}}{ \Image(d_{-2} \otimes \id)} & Q_0 \otimes_{\Lambda} \Lambda_{\con} & 0.  \\};
 \path[-stealth]
  (m-1-1.east|-m-1-2) edge node [above] {} (m-1-2)
    (m-1-2.east|-m-1-3) edge node [above] {$\scriptstyle d_{-3}$} (m-1-3)
    (m-1-3.east|-m-1-4) edge node [above] {$\scriptstyle d_{-2}$} (m-1-4)
    (m-1-4.east|-m-1-5) edge node [above] {$\scriptstyle d_{-1}$} (m-1-5)
    (m-1-5.east|-m-1-6) edge node [above] {} (m-1-6)
    (m-2-1.east|-m-2-2) edge node [above] {} (m-2-2)
    (m-2-2.east|-m-2-3) edge node [above] {$\scriptstyle d_{-3} \otimes \id$} (m-2-3)
    (m-2-3.east|-m-2-4) edge node [above] {$\scriptstyle d_{-2} \otimes \id$} (m-2-4)
    (m-2-4.east|-m-2-5) edge node [above] {$\scriptstyle d_{-1} \otimes \id$} (m-2-5)
    (m-2-5.east|-m-2-6) edge node [above] {} (m-2-6)
(m-3-1.east|-m-3-2) edge node [above] {} (m-3-2)
    (m-3-2.east|-m-3-3) edge node [above] {} (m-3-3)
    (m-3-3.east|-m-3-4) edge node [above] {$$} (m-3-4)
    (m-3-4.east|-m-3-5) edge node [above] {$\scriptstyle d_{-1} \otimes \id$} (m-3-5)
    (m-3-5.east|-m-3-6) edge node [above] {} (m-3-6)
(m-1-2) edge node [right] {$\scriptstyle  \partial_{-3}$} (m-2-2)
 (m-2-2) edge node [right] {} (m-3-2)
(m-1-3) edge node [right] {$\scriptstyle  \partial_{-2}$} (m-2-3)
 (m-2-3) edge node [left] {} (m-3-3)
(m-1-4) edge node [right] {$\scriptstyle  \partial_{-1}$} (m-2-4)
 (m-2-4) edge node [left] {} (m-3-4)
(m-1-5) edge node [right] {$\scriptstyle  \partial_{0}$} (m-2-5)
 (m-2-5) edge node [left] {} (m-3-5);
\end{tikzpicture}
\end{center}
By construction, this is a map $\Gamma_{\!\con} \otimes^{\bf L}_{\Gamma} T_i \to \tT_{\!i} \otimes_{\Lambda_{\con}} {\Lambda_{\con}}_{\Lambda}$ in the derived category of $\Gamma_{\!\con}$-$\Lambda$-bimodules. Thus, to prove Theorem \ref{square}, we show the induced maps on homology, $H^i(\partial_i)$, are injective for $i=0,-1$. 

To do this, we will make the assumption that for $i=0,-1$,
\begin{align}
Q_iI \cap \Kernel(d_i)=(\Kernel(d_i))I \label{assumption}
\end{align} 
where $I$ is the two-sided ideal of $\Lambda$ such that $\Lambda_{\con}=\Lambda / I$ (namely, $I$ consists of the endomorphisms of $M$ factoring through $\add R$). This holds trivially when $i=0$ and is proved below in Lemmas \ref{tor} and \ref{assumpholds} for the case $i=-1$. \\

With assumption \eqref{assumption}, take $q + \Image(d_{i-1}) \in H^i(Q)$ to be in the kernel of $H^i(\partial_i)$ (or equivalently such that $q \otimes 1  \in \Image(d_{i-1} \otimes \id))$. Then, for some $p_j \in Q_{i-1}$ and $f_j \in \Lambda$,
\begin{align*}
q \otimes 1  &= \big(d_{i-1} \otimes \id\big)\big(\text{\scalebox{0.8}{$\sum_j$}}\hspace{0.02cm}  (p_j \otimes f_j)\big)\\
&= \text{\scalebox{0.8}{$\sum_j$}} \hspace{0.01cm} (d_{i-1}(p_j) \otimes f_j) \\
&= d_{i-1}\big(\text{\scalebox{0.8}{$\sum_j$}} \hspace{0.02cm}  p_jf_j\big) \otimes 1.
\end{align*}
Thus, $q-d_{i-1}(\text{\scalebox{0.9}{$\sum\limits_j$}} \hspace{0.02cm} p_jf_j) \in Q_iI \cap \Kernel(d_i)$
and so, by assumption \eqref{assumption}, 
\begin{align} \label{qstuff}
q-d_{i-1}\big(\text{\scalebox{0.8}{$\sum_j$}} \hspace{0.03cm} p_jf_j\big) = \text{\scalebox{0.8}{$\sum_k$}} \hspace{0.03cm} q_k \lambda_k  
\end{align}
for some $q_k \in \Kernel(d_i), \lambda_k \in I$. By Proposition \ref{onesided}, there are isomorphisms of $\Lambda$-modules
\begin{align*}
\phi_i \colon  H^i(Q)  \to \left\{  
	\begin{array}{ll}
		 \uHom_R(M,\upnu_i M) & \mbox{if $i= 0$ } \\
		\uHom_R(M, \Omega K_i) & \mbox{if $i= -1$, }
	\end{array}
\right.
\end{align*}
and since the $\phi_i$ are $\Lambda$-module homomorphisms,
\begin{align*}
\phi_i \big( \text{\scalebox{0.8}{$\sum_k$}} \hspace{0.02cm} q_k \lambda_k  +\Image(d_{i-1})\big)=\text{\scalebox{0.8}{$\sum_k$}} \hspace{0.02cm} \phi_i\big(q_k +\Image(d_{i-1})\big)\lambda_k.
\end{align*} 
It is clear that $\uHom_R(M,\upnu_i M)$ and $\uHom_R(M, \Omega K_i)$ are both annihilated on the right by $I$ and hence each of the terms $\phi_i(q_k +\Image(d_{i-1}))\lambda_k$ must be zero since $\lambda_k \in I$. Thus, $\phi_i(\text{\scalebox{0.85}{$\sum\limits_k$}} \hspace{0.01cm} q_k \lambda_k  +\Image(d_{i-1}))=0$ and hence, as $\phi_i$ is an isomorphism,
\begin{align*}
\text{\scalebox{0.8}{$\sum_k$}} \ q_k \lambda_k \in \Image (d_{i-1}).
\end{align*}
Combining this with \eqref{qstuff} then shows that $q \in \Image(d_{i-1})$ and so the map on homology is injective.\\

Thus all that remains to prove Theorem \ref{square} is to verify assumption \eqref{assumption} in the case $i=-1$.
\begin{lemma} \label{tor}
Suppose that $ \dots \to P_{i-2} \xrightarrow{\updelta_{i-2}} P_{i-1} \xrightarrow{\updelta_{i-1}} P_{i} \xrightarrow{\updelta_{i}} P_{i+1} \to \cdots$ is a complex of projective right $\Lambda$-modules and that $I$ is the two-sided ideal of $\Lambda$ such that $\Lambda_{\con}= \Lambda/I$. If $\mathrm{Tor}_2^\Lambda (P_{i+1}/\Image(\updelta_{i}), \Lambda_{\con})=0$, then 
\begin{align*}
(\Kernel(\updelta_i))I  = P_iI \cap \Kernel(\updelta_i).
\end{align*} 
\end{lemma}
\begin{proof}
First of all, note that the inclusion $(\Kernel(\updelta_i))I  \subseteq P_iI \cap \Kernel(\updelta_i)$ is clear and so it is enough to show 
\begin{align*}
P_iI \cap \Kernel(\updelta_i) \subseteq (\Kernel(\updelta_i))I.
\end{align*}
There is an exact sequence
\begin{align*}
0 \to \Image(\updelta_i) \hookrightarrow P_{i+1} \to P_{i+1}/\Image(\updelta_i) \to 0. 
\end{align*}
Applying $- \otimes_\Lambda \Lambda_{\con}$ and using that $P_{i+1}$ is a projective $\Lambda$-module produces an exact sequence
\begin{align*}
0 \to \mathrm{Tor}_2^\Lambda (P_{i+1}/\Image(\updelta_i), \Lambda_{\con}) \to \mathrm{Tor}_1^\Lambda (\Image(\updelta_{i}), \Lambda_{\con}) \to 0
\end{align*}
which, combined with the assumption in the statement, implies that $ \mathrm{Tor}_1^\Lambda (\Image(\updelta_{i}), \Lambda_{\con})=0$. Thus, applying $- \otimes_\Lambda \Lambda_{\con}$ to the exact sequence
\begin{align*}
0 \to \Kernel(\updelta_i) \hookrightarrow P_i \xrightarrow{\updelta_i} \Image(\updelta_i) \to 0
\end{align*}
produces an exact sequence
\begin{align*}
0 \to \Kernel(\updelta_i) \otimes_\Lambda \Lambda_{\con} \hookrightarrow P_i \otimes_\Lambda \Lambda_{\con}\xrightarrow{\updelta_i \otimes 1} \Image(\updelta_i) \otimes_\Lambda \Lambda_{\con}\to 0.
\end{align*}
Now choose $p \in P_iI \cap \Kernel(\updelta_i)$. Then $p \otimes 1$ belongs to $\Kernel(\updelta_i) \otimes_\Lambda \Lambda_{\con}$ and maps to zero in $P_i \otimes_\Lambda \Lambda_{\con}$. Since the map is injective, this implies $p \otimes 1$ is zero in $\Kernel(\updelta_i) \otimes_\Lambda \Lambda_{\con}$ and hence $p \in  (\Kernel(\updelta_i))I$, completing the proof.
\end{proof}
We next apply Lemma \ref{tor} to $Q$ in \eqref{Q} where, by Proposition \ref{onesided}, there are isomorphisms of $\Lambda$-modules
\begin{align*}
P_0/ \Image{\updelta_{-1}} = Q_{0}/\Image(d_{-1}) \cong H^0(Q) \cong \uHom_R(M,\upnu_i M).
\end{align*}
With this in mind, the following completes the proof of Theorem \ref{square}.
\begin{lemma} \label{assumpholds}
Under setup \ref{setupcom}, $\mathrm{Tor}_2^\Lambda (\uHom_R(M,\upnu_i M), \Lambda_{\con})=0$.
\end{lemma}
\begin{proof}
We begin by constructing a projective resolution of $\Lambda_{\con}$ as a left $\Lambda$-module. By Proposition \ref{prop}, there are exact sequences
\begin{align*}
0 \to M \to R^k \to \Omega M \to 0 \\
0 \to \Omega M \to R^k \to M \to 0
\end{align*}
for some $k \in \N$ arising from taking the syzygy of $\Omega M$ and $M$ respectively. Applying $\Hom_R(-,M)$ to these sequences, using that $M$ is rigid and splicing yields the exact sequence
\begin{align*}
 \text{\scalebox{0.935}{$0 \to \Hom_R( M, M) \to \Hom_R( R^k, M) \to \Hom_R( R^k, M) \to \Hom_R( M, M) \to \Ext^1_R(\Omega M,M) \to 0$}}
\end{align*}
where all the terms except $\Ext^1_R(\Omega M,M) \cong \uHom_R(M,M)$ are projective left $\Lambda$-modules. Call this resolution $\mathbb{P}$.

Now $\mathrm{Tor}_2^\Lambda (\uHom_R(M, \upnu_i M), \Lambda_{\con})$ is defined to be the homology of $\uHom_R(M,\upnu_iM) \otimes_{\Lambda} \mathbb{P}$ in degree $-2$ and thus, to show that it is zero, it is enough to show that 
\begin{align*}
\big(\uHom_R(M,\upnu_iM) \otimes_{\Lambda} \mathbb{P}\big)_{-2} \cong \uHom_R(M,\upnu_i M) \otimes_{\Lambda} \Hom_R(R^k,M) =0
\end{align*}
or equivalently, that $ \uHom_R(M,\upnu_i M) \otimes_{\Lambda} \Hom_R(R,M)=0$.

Take $f \otimes g \in \uHom_R(M,\upnu_i M) \otimes_{\Lambda} \Hom_R( R, M)$. 
Then, since $R$ is a summand of $M$, there are maps $i\colon R \to M$ and $p \colon M \to R$ given by inclusion and projection such that $p \circ i = \id$. Therefore,
\begin{align*}
f \otimes g = f \otimes g \circ( p \circ i)  =  f \otimes (g \circ p) \circ i = f \circ g \circ p \otimes i  = 0 
\end{align*}
as $f \circ g \circ p$ factors through $R$. Thus $\uHom_R(M,\upnu_i M) \otimes_{\Lambda} \Hom_R( R, M)=0$ as required.
\end{proof}
To summarise, combining Lemmas \ref{assumpholds} and \ref{tor} shows that the assumption \eqref{assumption} is satisfied and thus we have a map $\Gamma_{\!\con} \otimes^{\bf L}_{\Gamma} T_i \to \tT_{\!i} \otimes_{\Lambda_{\con}} {\Lambda_{\con}}_{\Lambda}$ which is injective on homology in degrees $-1$ and $0$. By Corollary \ref{dimension}, both complexes have the same finite-dimensional homologies which are zero outside degrees $-1$ and $0$. This shows that the map is a quasi-isomorphism and hence completes the proof of Theorem \ref{square}. \qed 
\subsection{Proof of Corollary \ref{equivalence}} \label{equivproof}
In this subsection we will prove that the functor 
\begin{align*}
 - \otimes^{\bf L}_{\Gamma_{\!\con}} \tT_{\!i} \colon \Db(\Gamma_{\!\con}) \to \Db(\Lambda_{\con})
\end{align*}  is an equivalence.
By Theorem \ref{kel}, combined with \cite[6.3]{keller2}, it is enough to show that:
\begin{enumerate}[label=(\alph*)]
\item The map $\Gamma_{\!\con} \to \Hom_{\Dnb(\Lambda_{\con})}(\tT_{\!i},\tT_{\!i})$ induced by $ - \otimes^{\bf L}_{\Gamma_{\!\con}} \tT_{\!i}$ is an isomorphism and further $\Hom_{\Dnb(\Lambda_{\con})}(\tT_{\!i},\tT_{\!i}[n]) =0$ for $n \neq 0$.
\item $\tT_{\!i}$ is quasi-isomorphic to a complex $T \in \Kb(\proj \Lambda_{\con})$.
\item $\thick (T)= \Kb(\proj \Lambda)$.
\end{enumerate}
By Proposition \ref{onesided2}, $\tT_{\!i} \cong P$ in $\Db(\Lambda_{\con})$ where $P$ is the tilting complex for $\Lambda_{\con}$ from Theorem \ref{decon}. Thus, conditions (b), (c) and the latter part of (a) are satisfied. The following lemma uses the commutative diagram \eqref{commutingsquare} to show the first part of condition (a) also holds.

\begin{lemma} \label{naturalmap}
In the setup of \ref{setupcom}, the map $\Gamma_{\!\con} \to \Hom_{\Dnb(\Lambda_{\con})}(\tT_{\!i},\tT_{\!i})$ induced by $- \otimes^{\bf L}_{\Gamma_{\!\con}} \tT_{\!i}$ is an isomorphism.
\end{lemma}
\begin{proof}
The commutative diagram \eqref{commutingsquare} of Theorem \ref{square} induces a commutative diagram
\begin{center}
\begin{tikzpicture}
  \matrix (m) [matrix of math nodes,row sep=4em,column sep=4em,minimum width=2em] {  
\End_{\Dnb(\Lambda_{\con})}(\tT_{\!i})  & \End_{\Dnb( \Lambda)}({\tT_{\!i}}_\Lambda)   \\
\End_{\Dnb(\Gamma_{\!\con})}(\Gamma_{\!\con})  & \End_{\Dnb(\Gamma)}({\Gamma_{\!\con}}_\Gamma)   \\};
 \path[-stealth]
     (m-1-1.east|-m-1-2) edge node [above] {$\scriptstyle \upbeta$ } (m-1-2)
    (m-2-1.east|-m-2-2) edge node [above] {$ \scriptstyle \upgamma$} (m-2-2)
 (m-2-1) edge node [left] {$\scriptstyle  \upalpha$} (m-1-1)
 (m-2-2) edge node [right] {$ \scriptstyle \updelta$} (m-1-2);
\end{tikzpicture}
\end{center}
where $\upalpha$ is the map in the statement and $\updelta$ is an isomorphism since $- \otimes^{\bf L}_\Gamma T_i$ is an equivalence. Moreover, $\upgamma$ is an isomorphism since $\Gamma_{\!\con}$ is a $\Gamma_{\!\con}$-module. Thus, $\updelta \circ \upgamma$ is an isomorphism and so $\upbeta$ must be surjective. 

Now, the restriction and extension of scalars adjunction gives an isomorphism of vector spaces,
\begin{align}
\End_{\Dnb(\Lambda)}({\tT_{\!i}}_\Lambda) \cong \Hom_{\Dnb( \Lambda_{\con})}(\tT_{\!i} \otimes^{\bf L}_\Lambda \Lambda_{\con}, \tT_{\!i}). \label{adjunction}
\end{align}
By Proposition \ref{onesided2}, $\tT_i \cong P$ in $\Db(\Lambda_{\con})$ and \ref{defS} constructs a complex of projective $\Lambda$-modules $\EuScript{P}$ quasi-isomorphic to $P$. Using the explicit form of $\EuScript{P}$, it is easy to calculate that $\EuScript{P} \otimes_\Lambda \Lambda_{\con} \cong P \oplus P[3]$ exactly as in Proposition \ref{onesided2} and hence
\begin{align*}
\tT_{\!i} \otimes^{\bf L}_\Lambda \Lambda_{\con} \cong P \otimes^{\bf L}_\Lambda \Lambda_{\con} \cong \EuScript{P} \otimes_\Lambda \Lambda_{\con} \cong P \oplus P[3].
\end{align*}
This gives vector space isomorphisms,
\begin{align*}
\End_{\Dnb(\Lambda)}(\tT_{\!i}) &\cong \Hom_{\Dnb( \Lambda_{\con})}(P \oplus P[3], P) \tag{using \eqref{adjunction}} \\
&\cong \Hom_{\Dnb( \Lambda_{\con})}(P, P)  \tag{using that $P$ is tilting}\\
&\cong \End_{\Dnb( \Lambda_{\con})}(\tT_i). \tag{using $P \cong \tT_i$ in $\Db(\Lambda_{\con})$}
\end{align*}
Since $\upbeta$ is a surjective morphism of algebras between isomorphic finite dimensional vector spaces, it must therefore be an isomorphism. Thus, as $\upbeta, \upgamma$ and $\updelta$ are isomorphisms, $\upalpha$ must also be. 
\end{proof}
This lemma and the discussion above show $ - \otimes^{\bf L}_{\Gamma_{\!\con}} \tT_{\!i}$ gives an equivalence $\Db(\Gamma_{\!\con}) \to \Db(\Lambda_{\con})$ as required. \qed 

\begin{remark}
In particular, the results of this section show that for any basic rigid object $M \in \ucmr$, the algebras $\uEnd_R(M)$ and $\uEnd_R(\upnu_i M)$ are derived equivalent, recovering the result \cite[5.5]{Dugas}. Moreover, as we provide the two-sided tilting complex $\tT_i$, and hence a standard equivalence between the algebras, we can think of our results as a two-sided version of both \cite[5.5]{Dugas} and Theorem \ref{decon}.
\end{remark}
\section{Composition for Contraction Algebras} \label{composingsection}
Given two minimal models related by a simple flop, the previous section constructed an explicit standard derived equivalence between the associated contraction algebras. In this section, we will use the additional structure of an associated hyperplane arrangement to understand what happens to these standard equivalences under composition. 
\subsection{Deligne Groupoid Preliminaries} \label{delignedef}
For this description of the Deligne groupoid we follow \cite{faithful}. Given a real simplicial hyperplane arrangement $\cH$, construct a directed graph $X_{\cH}$, called the \emph{directed skeleton graph} of $\cH$, which has a vertex for each chamber and an arrow $v \to w$ if the corresponding chambers are separated by a codimension one wall. 

A \emph{positive path} of length $n$ is then a formal symbol
\begin{align*}
p= a_n a_{n-1}  \dots a_1
\end{align*}
such that there exist vertices $v_0, v_1, \dots, v_n$ and arrows $a_i \colon v_{i-1} \to v_i$. The \emph{source} and \emph{target} of such a $p$ are defined to be $s(p) \colonequals v_0$ and $t(p) \colonequals v_n$ respectively. A positive path $p$ is \emph{minimal} if there is no positive path with the same endpoints that has shorter length. Positive minimal paths are called \emph{atoms}.

There is an equivalence relation $\sim$ on the set of positive paths in $X_\cH$ given as the smallest equivalence relation satisfying:
\begin{enumerate}
\item If $p \sim q$ then $s(p)=s(q)$ and $t(p)=t(q)$.
\item If $p$ and $q$ are two atoms starting and ending at the same point then $p \sim q$.
\item If $p \sim q$, then $upr \sim uqr$ for all positive paths $u$ and $r$ with $t(r)=s(p)=s(q)$ and $s(u)=t(p)=t(q)$. 
\end{enumerate}
Write $[p]$ for the equivalence class of a positive path $p$. 
\begin{definition}
Let $\mathrm{Path}_{\cH}(v,w) \colonequals \{ [p]  \mid s(p)=v, t(p)=w \}$. The category $\cG_{\cH}^+$ is defined to have the vertices of $X_{\cH}$ as objects and $\mathrm{Path}_{\cH}(v,w)$ as morphisms. The \emph{Deligne groupoid} of $\cH$, denoted $\cG_{\cH}$, is then defined to be the groupoid completion of $\cG_{\cH}^+$; that is, the objects are the same as for $\cG_{\cH}^+$ but a formal inverse is added for each morphism.
\end{definition}
We will sometimes need to consider paths in $X_\cH$ which are not necessarily positive; namely, we consider formal symbols
\begin{align*}
p= a_n^{\upepsilon_n} a_{n-1}^{\upepsilon_{n-1}}  \dots a_1^{\upepsilon_{1}}
\end{align*}
where $\upepsilon_i \in \{-1,1\}$ and there exists vertices $v_0, v_1, \dots, v_n$ and arrows $a_i$ satisfying
\begin{enumerate}
\item if $\upepsilon_i=1$, then $a_i$ is an arrow $v_{i-1} \to v_i$;
\item if $\upepsilon_i=-1$, then $a_i$ is an arrow $v_{i} \to v_{i-1}$.
\end{enumerate} 
We think of $a_i^{-1}$ as travelling backwards along the arrow $a_i \colon v_i \to v_{i-1}$.

The following key theorem shows that the vertex groups of $\cG_\cH$ are all isomorphic and depend only on the structure of $\cH$.
\begin{theorem} \cite{hyperplane1, hyperplane2} \textup{(see \cite[2.1]{hyperplane3})} \label{vertexgroups}
If $\cH$ is a simplicial hyperplane arrangement, then for each chamber $v$ of $\cH$, there is an isomorphism $\Hom_{\cG_{\cH}}(v,v) \cong \pi_1(\mathbb{C}^n \backslash \cH_{\mathbb{C}})$ where $\cH_\mathbb{C}$ is the complexification of $\cH$.
\end{theorem}
\subsection{Strategy and Result for MMAs} \label{strategy}
For any minimal model of a complete local isolated cDV singularity, there is an associated real hyperplane arrangement, the details of which can be found in \cite[\S 5]{HomMMP}.
\begin{setup}\label{compsetup}
Let $\Spec R$ be a complete local isolated cDV singularity and choose a basic maximal rigid object $M \colonequals \bigoplus\limits_{i=0}^nM_i \in \cmr$ where $M_0 \cong R$. Associated to $M$ as in \cite[\S 5]{HomMMP}, there is a hyperplane arrangement $\cH_M \subset \mathbb{R}^n$ whose chambers are labelled by the minimal models of $\Spec R$ and hence by the maximal rigid objects of $\cmr$ using Theorem \ref{HMMPbij}. 
\end{setup}
To ease notation, we will denote a chamber in $\cH_M$ by $C_N$ if $N$ is the corresponding maximal rigid object. 
The following theorem shows that $\cH_M$ also encodes how the maximal rigid objects are related by mutation. 
\begin{theorem} \cite[6.9(5)]{HomMMP} \label{chambersmutate}
Under the setup of \ref{compsetup}, the directed skeleton graph $X_{\cH_M}$ is isomorphic to the  double of the mutation graph of maximal rigid objects in $\cmr$. 
\end{theorem}
As a consequence of this, given any two basic maximal rigid objects $M,N \in \cmr$, there is an isomorphism between $X_{\cH_M}$ and $X_{\cH_N}$ which fixes the maximal rigid objects. Further, each chamber $C_N$ of $\cH_M$ has $n$ codimension one walls and, by Theorem \ref{chambersmutate}, crossing a wall from $C_N$ is equivalent to mutating $N$ at an indecomposable summand. We will abuse notation and denote an arrow $C_N \to C_L$ in $X_{\cH_M}$ by $s_i$ if $L \cong \upnu_iN$. Note that our choice of indexing on the summands of $M$ fixes the indexing on any other maximal rigid object via mutation and so $s_i$ is well defined.

\begin{example}\label{eightchambers}
Suppose $f \colon X \to \Spec R$ is a minimal model of the complete local isolated cDV singularity $R \colonequals \mathbb{C} \llbracket u,v,x,y \rrbracket / (uv-xy(x+y^2))$. This is an example of a $cA_2$ singularity (see \cite[\S4.5]{rigidequiv}) whose minimal models each contain two curves in the exceptional locus and thus, by \cite[6.6]{HomMMP}, the associated hyperplane arrangement is the root system associated to the Dynkin diagram $A_2$, as shown on the left. If $M$ is the maximal rigid object associated to $f$, then by writing $\upnu_{i_n} \dots \upnu_{i_1}M$ as $M_{{i_n} \dots {i_1}}$, the chambers are labelled as shown.

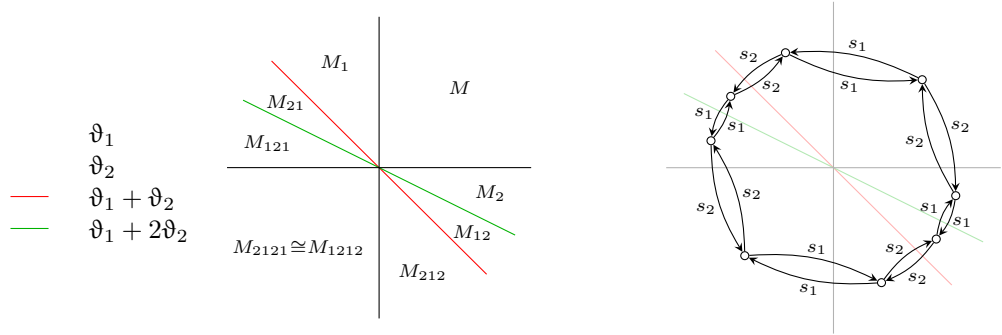
\[
\begin{array}{c}
\begin{tikzpicture}
\node at (6,0) {$\begin{tikzpicture}[scale=1,>=stealth]
\coordinate (A1) at (135:2cm);
\coordinate (A2) at (-45:2cm);
\coordinate (B1) at (153.435:2cm);
\coordinate (B2) at (-26.565:2cm);
\coordinate (C1) at (161.565:2cm);
\coordinate (C2) at (-18.435:2cm);
\draw[red] (A1) -- (A2);
\draw (-2,0)--(2,0);
\draw (0,-2)--(0,2);
\node (C+) at (45:1.5cm)  {$\scriptstyle M$};
\node (C1) at (112.5:1.5cm)  {$\scriptstyle M_1$};
\node (C2) at (157.5:1.5cm)  {$\scriptstyle M_{21}$};
\node (C-) at (225:1.5cm)  {$\scriptstyle M_{121} \cong M_{212}$};
\node (C4) at (-67.5:1.5cm)  {$\scriptstyle M_{12}$};
\node (C5) at (-22.5:1.5cm)  {$\scriptstyle M_2$};
\end{tikzpicture}$};
\node at (12,0) 
{\begin{tikzpicture}[scale=1.1,bend angle=20, looseness=1,>=stealth]
\coordinate (A1) at (135:2cm);
\coordinate (A2) at (-45:2cm);
\coordinate (B1) at (153.435:2cm);
\coordinate (B2) at (-26.565:2cm);
\draw[red!30] (A1) -- (A2);
\draw[black!30] (-2,0)--(2,0);
\draw[black!30] (0,-2)--(0,2);
\node (C+) at (45:1.55cm) [DWs] {};
\node (C1) at (112.5:1.5cm) [DWs] {};
\node (C3) at (157.5:1.5cm) [DWs] {};
\node (C-) at (225:1.55cm) [DWs] {};
\node (C4) at (-67.5:1.5cm) [DWs] {};
\node (C5) at (-22.5:1.5cm) [DWs] {};
\node (C6) at (45:1.85cm) {$\scriptstyle C_M$};
\node (C7) at (225:1.85cm) {$\scriptstyle C_{M_{121}}$};
\draw[->, bend right]  (C+) to (C1);
\draw[->, bend right]  (C1) to (C+);
\draw[->, bend right]  (C1) to (C3);
\draw[->, bend right]  (C3) to (C1);
\draw[->, bend right]  (C3) to (C-);
\draw[->, bend right]  (C-) to  (C3);
\draw[<-, bend right]  (C+) to  (C5);
\draw[<-, bend right]  (C5) to  (C+);
\draw[<-, bend right]  (C5) to  (C4);
\draw[<-, bend right]  (C4) to (C5);
\draw[<-, bend right]  (C4) to  (C-);
\draw[<-, bend right]  (C-) to (C4);
\node at (78.75:0.95cm) {$\scriptstyle s_1$};
\node at (78.75:1.55cm) {$\scriptstyle s_1$};
\node at (130:1.1cm) {$\scriptstyle s_2$};
\node at (130:1.675cm) {$\scriptstyle s_2$};
\node at (198:0.9cm) {$\scriptstyle s_1$};
\node at (196:1.6cm) {$\scriptstyle s_1$};
\node at (258.75:0.9cm) {$\scriptstyle s_2$};
\node at (258.75:1.6cm) {$\scriptstyle s_2$};
\node at (320:1.1cm) {$\scriptstyle s_1$};
\node at (320:1.675cm) {$\scriptstyle s_1$};
\node at (18:0.9cm) {$\scriptstyle s_2$};
\node at (16:1.625cm) {$\scriptstyle s_2$};
\end{tikzpicture}};
\end{tikzpicture}
\end{array}
\]
This results in the directed skeleton graph shown on the right. Notice that the two paths $s_1s_2s_1$ and $s_2s_1s_2$ starting in $C_M$ and traversing clockwise and anti-clockwise respectively to $C_{M_{121}}$ are both atoms, and hence are identified in the Deligne groupoid.
\end{example}
Since $\cH_M$ is a simplicial hyperplane arrangement \cite[3.8]{twists}, there is an associated Deligne groupoid $\cG_{\cH_M}$ whose vertex groups are all isomorphic to $\pi_1(\mathbb{C}^n \backslash \cH_\mathbb{C})$ by Theorem \ref{vertexgroups}. We will also associate to $\Spec R$ the groupoid $\mathbb{F}$, whose vertices are the maximal rigid objects in $\cmr$ and whose morphisms are 
\begin{align*}
\Hom_\mathbb{F}(M,N) \colonequals \{ \text{standard derived equivalences $\uEnd_R(M) \to \uEnd_R(N)$} \}.
\end{align*}

As in \cite{faithful}, our strategy to provide a faithful group action on the derived category of a contraction algebra is to construct a faithful functor
\begin{align*}
\Phi \colon \cG_{\cH_M} \to \mathbb{F}.
\end{align*}
Then it immediately follows that there is an injective group homomorphism 
\begin{align*}
\pi_1(\mathbb{C}^n \backslash \cH_\mathbb{C}) \to \mathrm{Auteq}(\Db(\uEnd_R(N)) 
\end{align*}
for any maximal rigid object $N$. To define such a functor, it is natural to set $\Phi(C_N) \colonequals N$ but then for each arrow $s_i \colon N \to \upnu_i N$ we need to choose an equivalence between the corresponding derived categories. These will be precisely the equivalences $F_i$ we constructed in \S\ref{explicit}.

\begin{notation} \label{funnot}
Suppose that $\Spec R$ is a complete local isolated cDV singularity and that $N \colonequals \bigoplus\limits_{i=0}^n N_i \in \cmr$ is a basic maximal rigid object with $N_0 \cong R$. If $i \neq 0$, consider the following algebras:
\begin{align*}
\Lambda \colonequals \End_R(N), \quad \Lambda_{\con} \colonequals \uEnd_R(N), \quad \Gamma \colonequals \End_R(\upnu_iN), \quad \text{and} \ \ \Gamma_{\!\con} \colonequals \uEnd_R(\upnu_iN).
\end{align*}
Additionally, define
\begin{enumerate}
\item $T_{i} \colonequals \Hom_R(N, \upnu_i N)$
\item $\tT_{\!i} \colonequals \tau_{\scriptscriptstyle{ \geq -1}} ( \Gamma_{\!\con} \otimes^{\bf L}_\Gamma T_i \otimes^{\bf L}_{\Lambda} \Lambda_{\con})$
\end{enumerate}
which induce functors $F_i \colonequals \RHom_{\Lambda_{\con}}(\tT_{\!i},-)$ and $G_i \colonequals \RHom_{\Lambda}(T_i,-)$. By Theorem \ref{square}, these make the diagram
\begin{center}
\begin{tikzpicture} 
  \matrix (m) [matrix of math nodes,row sep=2.5em,column sep=4em,minimum width=2em] {  
\Db(\Lambda_{\con})  & \Db(\Lambda)   \\
\Db(\Gamma_{\!\con})  & \Db(\Gamma)   \\};
 \path[-stealth]
    (m-1-1.east|-m-1-2) edge node [above] {$\scriptstyle - \otimes_{\Lambda_{\con}} \Lambda_{\con}$ } (m-1-2)
    (m-2-1.east|-m-2-2) edge node [above] {$ \scriptstyle - \otimes_{\Gamma_{\!\con}} \Gamma_{\!\con}$} (m-2-2)
 (m-1-1) edge node [left] {$\scriptstyle  F_i$} (m-2-1)
 (m-1-2) edge node [right] {$ \scriptstyle G_i $} (m-2-2);
\end{tikzpicture}
\end{center}
commute. Moreover, $F_i$ and $G_i$ are equivalences by Corollary \ref{equivalence} and Theorem \ref{demma} respectively with inverses $F_i^{-1} \colonequals - \otimes^{\bf L}_{\Gamma_{\!\con}} \tT_{\!i}$ and $G_i^{-1} \colonequals - \otimes^{\bf L}_{\Gamma} T_i$.
\end{notation}
\begin{remark}
Note that we will abuse notation by using $F_i$ to refer to any equivalence
\begin{align*}
\Db(\uEnd_R(N)) \to \Db(\uEnd_R(\upnu_iN))
\end{align*}
constructed as in \S \ref{explicit}, regardless of the choice of maximal rigid object $N$. Similarly, $G_i$ will refer to any standard equivalence between MMAs induced by a tilting bimodule of the form $\Hom_R(N,\upnu_i N)$. 
\end{remark} 

With this notation, we define $\Phi \colon \cG_{\cH} \to \mathbb{F}$ by mapping the arrow $s_i \colon N \to \upnu_i N$ to the corresponding standard equivalence $F_i$. This construction will yield a functor between the groupoids if and only if equivalent paths in $X_{\cH_M}$ give isomorphic functors. In particular, the equivalences $F_i$ must satisfy the relations on paths in the Deligne groupoid. To check this, we need to be able to understand compositions of these functors. For the $G_i$, this is already known.
\begin{theorem}\cite[4.6]{faithful} \label{specialpath}
Under the setup of \ref{compsetup}, write $\Lambda \colonequals \End_R(M)$ and mutate $M$ $m$ times to get
\begin{align*}
N \colonequals \upnu_{i_m} \upnu_{i_{m-1}} \dots \upnu_{i_1} M. 
\end{align*}
 This defines a positive path $s_{i_m} \dots s_{i_1} \colon C_M \to C_N$ in $X_{\cH_M}$. If this path is an atom then
\begin{align*}
G_{i_m} \circ \dots \circ G_{i_1} \simeq \RHom_\Lambda(\Hom_R(M,N), -).
\end{align*}
\end{theorem}
\begin{remark}
\begin{enumerate}[leftmargin=0cm,itemindent=.6cm,labelwidth=\itemindent,labelsep=0cm,align=left]
\item Although \cite{faithful} only prove the above result for NCCRs (or equivalently, when the corresponding minimal model is smooth) the proof works with no modifications for MMAs.
\item The proof in \cite{faithful} heavily relies on using a partial order on the set 
\begin{align*}
\{ \Hom_R(M,N) \mid N \in \cmr \ \text{is maximal rigid}\}
\end{align*}
of tilting bimodules. More generally, for a rigid object $M$ it is not known whether every element of the set
\begin{align*}
\{ \Hom_R(M,N) \mid N \in \cmr \ \text{is obtained from $M$ by iterated mutation}\}
\end{align*}
is a tilting bimodule. If this were to hold, the proof of Theorem \ref{specialpath} would work in this more general setting, and the results of this paper would also generalise to statements about rigid objects, rather than just maximal rigid objects.
\end{enumerate}
\end{remark}

\subsection{Result for Contraction Algebras}
The main result of this section will be to prove the analogue of Theorem \ref{specialpath} for contraction algebras. For this, the following two technical lemmas are required.
\begin{lemma} \label{homlambda}
In the setup of \ref{compsetup}, write $\Lambda \colonequals \End_R(M)$ and $\Lambda_{\con} \colonequals \uEnd_R(M)$. Then the homology of $\Lambda_{\con} \otimes^{\bf L}_{\Lambda} \Lambda_{\con}$ is zero outside degrees $-3$ and $0$. In degree $0$, it is isomorphic as a $\Lambda_{\con}$-bimodule to $\Lambda_{\con}$.
\end{lemma}
\begin{proof}
By Lemma \ref{projresj}, $\Lambda_{\con}$ has a projective resolution as a right $\Lambda$-module of the form
\begin{align*}
0 \to \Lambda \to \Hom_R(M,R^k) \to \Hom_R(M,R^k) \to \Lambda \to 0.
\end{align*}
Applying $- \otimes_{\Lambda} \Lambda_{\con}$ gives
\begin{align*}
0 \to \Lambda_{\con} \to 0 \to 0 \to \Lambda_{\con} \to 0
\end{align*}
and this shows $H^i (\Lambda_{\con} \otimes^{\bf L}_{\Lambda} \Lambda_{\con}) =0$ if $i \neq 0,-3$. To obtain the homology as bimodules in degree $0$, note that it is isomorphic as $\Lambda_{\con}$-$\Lambda_{\con}$-bimodules to
$\Tor^\Lambda_0(\Lambda_{\con}, \Lambda_{\con}) \cong \Lambda_{\con} \otimes_{\Lambda} \Lambda_{\con}$.
Further, it is easily checked that the isomorphism $\Lambda_{\con} \otimes_{\Lambda} \Lambda_{\con} \to \Lambda_{\con}$ defined via $\lambda_1 \otimes \lambda_2 \to \lambda_1\lambda_2$ is an isomorphism of $\Lambda_{\con}$-$\Lambda_{\con}$-bimodules as required.
\end{proof}

\begin{lemma} \label{truncation}
Suppose that $\Delta$ is a ring and $\Lambda_{\con}$ is the contraction algebra of some minimal model of a complete local isolated cDV singularity. Let $X$ be a complex of $\Delta$-$\Lambda_{\con}$-bimodules whose homology vanishes in degrees other than $-1$, $0$ and $1$. Then,
\begin{align*}
 \tau_{\scriptscriptstyle\geq -1} (X \otimes_{\Lambda_{\con}} \Lambda_{\con} \otimes^{\bf L}_{\Lambda} \Lambda_{\con}) \cong X
\end{align*}
in the derived category of $\Delta$-$\Lambda_{\con}$-bimodules.
\end{lemma}
\begin{proof} 
First note that since the homology vanishes above degree 1,  $\tau_{\leq 1} (X)$ is quasi-isomorphic to $X$ and so we can assume $X_i=0$ for all $i >1$ by instead considering the truncation. Now, for any complex Y of $\Lambda_{\con}$-$\Lambda_{\con}$-bimodules, there is a triangle
\begin{align*}
\tau_{ \scriptscriptstyle < -1}(Y) \to Y \to \tau_{\scriptscriptstyle \geq -1}(Y) \to \tau_{ \scriptscriptstyle < -1}(Y)[1].
\end{align*}
Taking $Y \colonequals \Lambda_{\con} \otimes^{\bf L}_{\Lambda} \Lambda_{\con}$, Lemma \ref{homlambda} identifies $\tau_{ \scriptscriptstyle < -1}(Y)$ and  $\tau_{\scriptscriptstyle \geq -1}(Y)$ as complexes in a single degree, and so gives a triangle 
\begin{align*}
M[3] \to \Lambda_{\con} \otimes^{\bf L}_{\Lambda} \Lambda_{\con} \to \Lambda_{\con} \to M[4]
\end{align*}
where $M$ is some $\Lambda_{\con}$-$\Lambda_{\con}$-bimodule. 
Applying the functor $X \otimes^{\bf L}_{\Lambda_{\con}} - $ then results in a triangle
\begin{align} \label{truncationtriangle}
X \otimes_{\Lambda_{\con}}^{\bf L}M[3] \to X \otimes^{\bf L}_{\Lambda_{\con}} \Lambda_{\con} \otimes^{\bf L}_{\Lambda} \Lambda_{\con} \xrightarrow{\phi} X \to X \otimes_{\Lambda_{\con}}^{\bf L}M[4]
\end{align}
in the derived category of $\Delta$-$\Lambda_{\con}$-bimodules. As truncation is a functor, the map $\phi$ induces a map
\begin{align*}
\phi^* \colon \tau_{\scriptscriptstyle \geq-1}(X \otimes^{\bf L}_{\Lambda_{\con}} \Lambda_{\con} \otimes^{\bf L}_{\Lambda} \Lambda_{\con}) \to  \tau_{\scriptscriptstyle \geq-1}(X)
\end{align*}
and further, it easy to check that $\phi$ and $\phi^*$ fit into the following commutative diagram 
\begin{center}

\begin{tikzpicture}
  \matrix (m) [matrix of math nodes,row sep=3em,column sep=2.5em,minimum width=2em] {  
X \otimes_{\Lambda_{\con}}^{\bf L}M[3]  & X \otimes^{\bf L}_{\Lambda_{\con}} \Lambda_{\con} \otimes^{\bf L}_{\Lambda} \Lambda_{\con}  &X  & X \otimes_{\Lambda_{\con}}^{\bf L}M[4] \\
& \tau_{\scriptscriptstyle \geq-1}(X \otimes^{\bf L}_{\Lambda_{\con}} \Lambda_{\con} \otimes^{\bf L}_{\Lambda} \Lambda_{\con}) & \tau_{\scriptscriptstyle \geq-1}(X) & \\};
 \path[-stealth]
(m-1-1.east|-m-1-2) edge node [above] {} (m-1-2)
(m-1-3.east|-m-1-4) edge node [above] { } (m-1-4)
     (m-1-2.east|-m-1-3) edge node [above] {$\scriptstyle \phi$ } (m-1-3)
    (m-2-2.east|-m-2-3) edge node [above] {$ \scriptstyle \phi^*$} (m-2-3)
 (m-1-2) edge node [left] {$\scriptstyle  \upalpha$} (m-2-2)
 (m-1-3) edge node [right] {$ \scriptstyle \upbeta$} (m-2-3);
\end{tikzpicture}
\end{center}
where $\upalpha$ and $\upbeta$ are the natural maps to the truncations.

As $X$ has vanishing homology in degrees other than $-1$, $0$ and $1$, it is clear that $\tau_{\scriptscriptstyle\geq -1} (X) \cong X$, and hence to prove the result, it is enough to show $\phi^*$ is a quasi-isomorphism. 

As homology is functorial, taking homology of the commutative diagram yields another commutative diagram
 \begin{center}
\begin{tikzpicture}
 \matrix (m) [matrix of math nodes,row sep=3em,column sep=1.4em,minimum width=2em] {  
H^i(X \otimes^{\bf L}M[3])  & H^i(X \otimes^{\bf L} \Lambda_{\con} \otimes^{\bf L}_{\Lambda} \Lambda_{\con})  & H^i(X)  & H^i(X \otimes^{\bf L}M[4]) \\
& H^i(\tau_{\scriptscriptstyle \geq-1}(X \otimes^{\bf L} \Lambda_{\con} \otimes^{\bf L}_{\Lambda} \Lambda_{\con})) & H^i(\tau_{\scriptscriptstyle \geq-1}(X)) & \\};
 \path[-stealth]
(m-1-1.east|-m-1-2) edge node [above] {} (m-1-2)
(m-1-3.east|-m-1-4) edge node [above] { } (m-1-4)
     (m-1-2.east|-m-1-3) edge node [above] {$\scriptstyle H^i(\phi)$ } (m-1-3)
    (m-2-2.east|-m-2-3) edge node [above] {$ \scriptstyle H^i(\phi^*)$} (m-2-3)
 (m-1-2) edge node [left] {$\scriptstyle  H^i(\upalpha)$} (m-2-2)
 (m-1-3) edge node [right] {$ \scriptstyle H^i(\upbeta)$} (m-2-3);
\end{tikzpicture}
\end{center}
where unadorned tensors are over $\Lambda_{\con}$. The top row is exact since it is part of the long exact sequence of homology of the triangle \eqref{truncationtriangle}. For $i < -1$, the truncation is constructed so that
\begin{align*}
H^i(\tau_{\scriptscriptstyle \geq-1}(X \otimes^{\bf L}_{\Lambda_{\con}} \Lambda_{\con} \otimes^{\bf L}_{\Lambda} \Lambda_{\con})) = 0 = H^i(\tau_{\scriptscriptstyle \geq-1}(X))
\end{align*}
 and hence $H^i(\phi^*)$ is an isomorphism.
For $i \geq -1$, it is clear that $H^i(\upalpha)$ and $H^i(\upbeta)$ are isomorphisms. Further, as $X_i=0$ for all $i \geq 2$ and $M$ is a module, $(X \otimes_{\Lambda_{\con}}^{\bf L}M)_i=0$ for all $i \geq 2$ as well. Thus, for $i \geq -1$,
\begin{align*}
H^i(X \otimes_{\Lambda_{\con}}^{\bf L}M[3]) = 0 = H^{i}(X \otimes_{\Lambda_{\con}}^{\bf L}M[4])
\end{align*}
and hence $H^i(\phi)$ must be an isomorphism using exactness of the top row. This shows that for $i \geq -1$, the top and vertical maps in the commutative square are isomorphisms and hence $H^i(\phi^*)$ must also be. Thus $H^i(\phi^*)$ is an isomorphism for all $i$ and so $\phi^*$ is a quasi-isomorphism as required.
\end{proof}
The following is the main technical result of this section.
\begin{theorem} \label{composing}
Under the setup of \ref{compsetup}, write $\Lambda \colonequals \End_R(M)$ and $\Lambda_{\con} \colonequals \uEnd_R(M)$. Mutate $M$ $m$ times to get
\begin{align*}
N \colonequals \upnu_{i_m} \upnu_{i_{m-1}} \dots \upnu_{i_1} M. 
\end{align*}
This defines a positive path $\upalpha \colonequals s_{i_m} \dots s_{i_1} \colon C_M \to C_N$ in $X_{\cH_M}$. If this path is an atom then, writing $\Gamma \colonequals \End_R(N)$ and $\Gamma_{\!\con} \colonequals \uEnd_R(N)$,
\begin{align*}
F_{i_m} \circ \dots \circ F_{i_1} \simeq \RHom_{\Lambda_{\con}}(\tau_{\scriptscriptstyle\geq 1}\big(\Gamma_{\!\con} \otimes^{\bf L}_{\Gamma} \Hom_R(M,N) \otimes^{\bf L}_{\Lambda} \Lambda_{\con}\big) , -).
\end{align*}
\end{theorem}
\begin{proof}
We begin by setting notation. For $1 \leq j \leq m-1$, define $\Gamma^j \colonequals \End_R(\upnu_{i_j}\dots \upnu_{i_1}M)$ and $\Gamma^j_{\!\con} \colonequals \uEnd_R(\upnu_{i_j}\dots \upnu_{i_1}M)$. Further, with $T_{i_j}$ and $\tT_{i_j}$ as in notation \ref{funnot}, set
\begin{align}
\tT_{\upalpha} \colonequals \tT_{i_m} \otimes^{\bf L}_{\Gamma^{m-1}_{\!\con}} \tT_{i_{m-1}} \otimes^{\bf L}_{\Gamma^{m-2}_{\!\con}} \dots \otimes^{\bf L}_{\Gamma^2_{\!\con}} \tT_{i_2} \otimes^{\bf L}_{\Gamma^1_{\!\con}} \tT_{i_1} \label{ttalpha}
\end{align}
and
\begin{align}
T_\upalpha \colonequals T_{i_m} \otimes^{\bf L}_{\Gamma^{m-1}} T_{i_{m-1}} \otimes^{\bf L}_{\Gamma^{m-2}} \dots \otimes^{\bf L}_{\Gamma^2} T_{i_2} \otimes^{\bf L}_{\Gamma^1} T_{i_1} \label{talpha}
\end{align}
so that $ F_{i_1}^{-1} \circ \dots \circ  F_{i_m}^{-1} \cong - \otimes^{\bf L}_{\Gamma_{\!\con}} \tT_\upalpha$ and $ G_{i_1}^{-1} \circ \dots \circ  G_{i_m}^{-1} \cong - \otimes^{\bf L}_{\Gamma} T_\upalpha$. Then, for each $F_i$ there is a commutative diagram \eqref{commutingsquare}, and combining them shows that 
\begin{align*}
G_{i_1}^{-1} \circ \dots \circ G_{i_m}^{-1} \circ (- \otimes_{\Gamma_{\!\con}} {\Gamma_{\!\con}}_{\Gamma}) \cong  (- \otimes_{\Lambda_{\con}} {\Lambda_{\con}}_{\Lambda}) \circ  F_{i_1}^{-1} \circ \dots \circ  F_{i_m}^{-1}.
\end{align*}
In other words, there is an isomorphism
\begin{align}
\tT_\upalpha \otimes_{\Lambda_{\con}} \Lambda_{\con}  \cong \Gamma_{\!\con} \otimes^{\bf L}_\Gamma T_\upalpha \label{bimod}
\end{align}
in the derived category of $\Gamma_{\!\con}$-$\Lambda$-bimodules. 

By Theorem \ref{specialpath}, $T_\upalpha \cong \Hom_R(M,N)$ as $\Gamma$-$\Lambda$-bimodules and hence $T_\upalpha$ is a tilting module of projective dimension one by Theorem \ref{demma}. Thus the right hand side of \eqref{bimod} has nonzero homology in at most degrees $-1$ and $0$. This shows $\tT_\upalpha$ also has nonzero homology in at most degrees $-1$ and $0$, as the tensor with $\Lambda_{\con}$ does not change the homology. 

Now, applying $- \otimes_\Lambda^{\bf L} \Lambda_{\con}$ to both sides gives an isomorphism
\begin{align*}
\tT_\upalpha \otimes_{\Lambda_{\con}} \Lambda_{\con} \otimes_\Lambda^{\bf L} \Lambda_{\con} \cong \Gamma_{\!\con} \otimes^{\bf L}_\Gamma \Hom_R(M,N) \otimes_\Lambda^{\bf L} \Lambda_{\con}
\end{align*}
in the derived category of $\Gamma_{\!\con}$-$\Lambda_{\con}$-bimodules. Thus, applying the truncation $\tau_{\scriptscriptstyle\geq-1}$ and Lemma \ref{truncation} gives
\begin{align*}
\tT_\upalpha \cong \tau_{\scriptscriptstyle\geq -1} \big( \Gamma_{\!\con} \otimes^{\bf L}_\Gamma \Hom_R(M,N) \otimes_\Lambda^{\bf L} \Lambda_{\con} \big)
\end{align*}
in the derived category of $\Gamma_{\!\con}$-$\Lambda_{\con}$-bimodules, completing the proof.
\end{proof}

\begin{corollary} \label{functordef}
Suppose that $\Spec R$ is a complete local isolated cDV singularity with maximal rigid object $M \in \cmr$ and associated hyperplane arrangement $\cH$. Then there is a well defined functor \begin{align*}
\Phi \colon \cG_{\cH} \to \mathbb{F}
\end{align*}
which sends a chamber $C_N$ to $N$ and an arrow $s_i: N \to \upnu_i N$ to the standard equivalence $F_i$ (as in notation \ref{funnot}). In particular, for any contraction algebra $\Lambda_{\con}$, there is a group homomorphism
\begin{align*}
\pi_1(\mathbb{C}^n \backslash \cH_{\mathbb{C}}) \to \mathrm{Auteq}(\Db(\Lambda_{\con}))
\end{align*}
where  $\cH_{\mathbb{C}}$ is the complexification of $\cH$.
\end{corollary}
\begin{proof}
It is enough to show that for any two atoms $\upalpha, \upbeta : C_L \to C_N$, we have that $\Phi(\upalpha) \cong \Phi(\upbeta)$. Viewing $\upalpha$ and $\upbeta$ as paths in $X_{\cH_L}$ via the isomorphism in the remark after Theorem \ref{chambersmutate}, applying Theorem \ref{composing} in that setting gives the desired result. 
\end{proof}

We will return to this result in \S\ref{faithfulness} where we will further show that the functor is faithful. For now, note that the relations from the Deligne groupoid in fact imply that the $F_i$ satisfy higher length braid relations.
\begin{corollary} \label{braidthm}
With notation as above, then there is a functorial isomorphism
\begin{align}
\underbrace{\hdots F_j \circ F_i\circ F_j}_{m} \cong \underbrace{\hdots F_i \circ F_j\circ F_i}_{m}  \label{braid}
\end{align}
for some $m$ with $2\leq m\leq 8$.
\end{corollary}
\begin{proof}
By Theorem \ref{composing}, traversing one way round a codimension two wall is isomorphic to traversing the other way round, so this establishes \eqref{braid} for some $m$.  The fact $2\leq m\leq 8$ follows from the bound for flops in \cite[\S1]{twists}.
\end{proof}
Another corollary of Theorem \ref{composing} is the following analogue of Theorem \ref{demma}(2) in the case of contraction algebras; namely, it provides a direct standard equivalence between any two contraction algebras.
\begin{corollary} \label{directequivalences}
Let $\Spec R$ be a complete local isolated cDV singularity and suppose that $M,N \in \cmr$ are two maximal rigid objects. Writing $\Lambda \colonequals \End_R(M)$, $\Lambda_{\con} \colonequals \uEnd_R(M)$, $\Gamma \colonequals \End_R(N)$ and $\Gamma_{\!\con} \colonequals \uEnd_R(N)$ there is a standard derived equivalence
\begin{align*}
\RHom_{\Lambda_{\con}}\left(\tau_{\scriptscriptstyle\geq 1}\big(\Gamma_{\!\con} \otimes^{\bf L}_{\Gamma} \Hom_R(M,N) \otimes^{\bf L}_{\Lambda} \Lambda_{\con}\big) , -\right) \colon \Db(\Lambda_{\con}) \to \Db(\Gamma_{\con}).
\end{align*}
\end{corollary}
\begin{proof}
Given two maximal rigid objects $M,N \in \cmr$, there must be an atom between their corresponding chambers in the hyperplane arrangement $\cH_M$. Applying Theorem \ref{composing} to this atom gives the required functor.
\end{proof}

\section{Two-Term Tilting Complexes} \label{twotermsection}
Given any two maximal rigid objects $M,N \in \cmr$, Corollary \ref{directequivalences} gives an explicit standard derived equivalence between their contraction algebras. In this section we show that these derived equivalences are precisely the derived equivalences induced by \textit{two-term} tilting complexes.
\subsection{Background}
Throughout this section, let $A$ be a finite dimensional symmetric algebra.  
\begin{remark}
Many of the results cited here are originally stated for silting complexes; a weaker notion than tilting. However, our assumption that $A$ is symmetric means that silting and tilting are equivalent \cite[2.8]{siltingmutation}.
\end{remark}
The set of tilting complexes for $A$ comes with a partial order \cite{siltingmutation} which allows us to define two-term tilting complexes.
\begin{definition} Let $P$ and $Q$ be tilting complexes for $A$.  \begin{enumerate}[leftmargin=0cm,itemindent=.6cm,labelwidth=\itemindent,labelsep=0cm,align=left]
\item If $\Hom_{\Kb(\proj A)}(P,Q[i])=0$ for all $i >0$, then we say $P \geq Q$. Further,  we write $P > Q$ if $P \geq Q$ and $P \ncong Q$.
\item $P$ is called two-term if $A \geq P \geq A[1]$, or equivalently by \cite[2.9]{tiltingconnectedness}, if the terms of $P$ are zero in every degree other than $0$ and $-1$.
\end{enumerate} 
\end{definition}
We will write $\twotilt A$ for the set of two-term tilting complexes for $A$. The link with maximal rigid objects arises from the following theorem.
\begin{theorem}\cite[4.7]{[AIR]} \label{ctbij}
Let $\cC$ be a Hom-finite, Krull-Schmidt, 2-CY triangulated category and $M$ be a basic maximal rigid object of $\cC$. If $A \colonequals \End_\cC(M)$ is a symmetric algebra, then there is a bijection 
\begin{align*}
\{\text{basic maximal rigid objects in $\cC$}\}  \longleftrightarrow \twotilt A
\end{align*}
which preserves the number of summands.
\end{theorem}
Tilting complexes have a notion of mutation, defined similarly to that of maximal rigid objects. 
\begin{definition} \cite[2.31]{siltingmutation} \label{mutdef}
Let $P \in \Kb(\proj \Lambda)$ be a basic tilting complex for $A$, and write $P \colonequals \bigoplus\limits_{i=1}^n P_i$ where each $P_i$ is indecomposable. Consider a triangle
\begin{align*}
P_i \xrightarrow{f} P' \to Q_i \to P_i[1] 
\end{align*} 
where $f$ is a minimal left $\add(P/P_i)$-approximation of $P_i$. Then $\upmu_i(P)\colonequals P/P_i \oplus Q_i$ is also a tilting complex, known as the left mutation of $P$ with respect to $P_i$. 
\end{definition}
Further, the partial order can determine when two complexes are related by a single mutation.
\begin{theorem}\cite[2.35]{siltingmutation} \label{siltingmutation2}
If $P$ and $Q$ are basic tilting complexes for $A$, then the following are equivalent:
\begin{enumerate}
\item $Q=\upmu_i(P)$ for some summand $P_i$ of $P$.
\item $P > Q$ and there is no tilting complex $T$ such that $P > T > Q$.
\end{enumerate}
\cite[3.9]{[AIR]} Further, if $P, Q \in \twotilt A$ then the second condition can be replaced with:
\begin{enumerate}
\item[\normalfont(2')] $P > Q$ and they differ by exactly the $i^{th}$ summand. 
\end{enumerate}
\end{theorem}
The last point shows that a pair of two-term tilting complexes are connected by a single mutation if and only if they differ by exactly one indecomposable summand. As the same is true for maximal rigid objects in $\cC$, it is easy to see the bijection of Theorem \ref{ctbij} preserves mutation and hence the mutation graphs of both sets of objects will be the same \cite[4.8]{[AIR]}.
\subsection{New Results}
We now return to the setting of complete local isolated cDV singularities. Recall that under the setup of \ref{compsetup}, there exists some hyperplane arrangement $\cH_M$ with directed skeleton graph $X_{\cH_{M}}$. In Corollary \ref{functordef}, we associated to any arrow $s_i \colon N \to \upnu_i N$ in $X_{\cH_M}$ a standard derived equivalence $ \Phi(s_i) = F_i$, and hence to any positive path a derived equivalence between the starting and ending contraction algebras as follows. 
\begin{notation} \label{compfun}
Under the setup of \ref{compsetup}, choose a positive path $\upalpha \colonequals s_{i_m} \dots s_{i_1}$ starting in chamber $C_N$. Let $\upnu_\upalpha N \colonequals \upnu_{i_m} \dots \upnu_{i_1} N$ and write $\Lambda_{\con} \colonequals \uEnd_R(N)$ and $\Gamma_{\!\con} \colonequals \uEnd_R(\upnu_\upalpha N)$. Consider
\begin{enumerate}
\item $F_\upalpha \colonequals F_{i_m} \circ \dots \circ F_{i_1} = \RHom_{\Lambda_{\con}}(\tT_\upalpha, -) \colon \Db(\Lambda_{\con}) \to \Db(\Gamma_{\!\con})$, where $\tT_\upalpha$ is defined as in \eqref{ttalpha}.
\item $\upmu_\upalpha \Lambda_{\con} \colonequals \upmu_{i_m} \dots \upmu_{i_1} \Lambda_{\con}$.
\end{enumerate}
\end{notation}
\begin{remark}
In this notation, $\Phi(\upalpha) =F_\upalpha$, where $\Phi$ is the functor from Corollary \ref{functordef} and further, if $\upalpha$ is an atom then, by Theorem \ref{composing} applied to $\cH_N$,
\begin{align}
\tT_\upalpha \cong \tau_{\scriptscriptstyle\geq 1}\big(\Gamma_{\!\con} \otimes^{\bf L}_{\Gamma} \Hom_R(N,\upnu_\upalpha N) \otimes^{\bf L}_{\Lambda} \Lambda_{\con}\big). \label{bimodulecomplex}
\end{align}
\end{remark}
By Proposition \ref{rouq}, tilting complexes for $\Lambda_{\con}$ induce a unique (up to algebra isomorphism) standard derived equivalence and so we can compare the derived equivalence induced by $\upmu_\upalpha \Lambda_{\con}$ to $F_\upalpha$. 
When the path is of length one, suppose that $\upnu_i N$ is obtained via the exchange sequence
\begin{align*}
0 \to N_i \xrightarrow{b_i} V_i \xrightarrow{d_i} K_i \to 0.
\end{align*}
Then Proposition \ref{onesided2} shows that $\tT_i$ is isomorphic to the tilting complex
\begin{align*}
P \colonequals \big(\uHom_R(N,N_i) \xrightarrow{b_i \circ -} \uHom_R(N,V_i)\big) \oplus \big(0 \to \bigoplus_{j\neq i} \uHom_R(N,N_j)\big),
\end{align*}
in $\Db(\Lambda_{\con})$ and hence $F_i$ is induced by $P$. Now $P$ is clearly two-term and further, as it differs from $\Lambda_{\con} \colonequals \uEnd_R(N)$ by exactly one summand, Theorem \ref{siltingmutation2} shows that $P$ must be $\upmu_i \Lambda_{\con}$. Thus, $F_i$ is precisely the equivalence (up to algebra isomorphism) induced by $\upmu_i \Lambda_{\con}$. The following shows this holds more generally for longer positive paths.  

\begin{proposition} \cite[3.10]{rigidequiv} \label{mutationpaths}
Under the setup of \ref{compsetup}, choose a basic maximal rigid object $N$ in $\cmr$ and let $\upalpha \colonequals s_{i_m} \dots s_{i_1}$ be a positive path in $X_{\cH_M}$ starting at $N$. Writing $\Lambda_{\con} \colonequals \uEnd_R(N)$ and $\Gamma_{\con} \colonequals \uEnd_R(\upnu_{\upalpha}N)$, the following hold.
\begin{enumerate}
\item $F_\upalpha( \upmu_\upalpha \Lambda_{\con}) \cong \Gamma_{\!\con}$ in $\Db(\Gamma_{\!\con})$.
\item $\tT_\upalpha \cong \upmu_\upalpha \Lambda_{\con}$ in $\Db(\Lambda_{\con})$.
\end{enumerate}
\end{proposition}
\noindent Our goal is to determine which paths correspond to the two-term tilting complexes.

\begin{lemma} \label{twotermatoms}
Under the setup of \ref{compsetup}, choose a chamber $C_N$ and let $\Lambda_{\con} \colonequals \uEnd_R(N)$. If $T \colonequals \upmu_{i_m} \dots \upmu_{i_1} \Lambda_{\con}$ is a two-term tilting complex for $\Lambda_{\con}$, then the path $\upalpha \colonequals s_{i_m} \dots s_{i_1}$ starting in chamber $C_N$ must be an atom.
\end{lemma}
\begin{proof}
We prove this by induction on $m$. When $m=1$ this is clear as any path of length one must be an atom.\\
Now assume $m \geq 2$ and let $\upbeta \colonequals s_{i_{m-1}} \dots s_{i_1}$. By Theorem \ref{siltingmutation2}, we have 
\begin{align*}
\Lambda_{\con} > \upmu_\upbeta \Lambda_{\con}>T
\end{align*}
and hence, as $T$ is two-term,
\begin{align*}
\Lambda_{\con} > \upmu_\upbeta \Lambda_{\con}>T \geq \Lambda_{\con}[1]
\end{align*}
so that $\upmu_\upbeta \Lambda_{\con}$ is also two-term. Thus, by the inductive hypothesis, $\upbeta$ is an atom. Let us suppose that $\upalpha$ is not.
By \cite[5.1]{faithful}, $\upbeta$ must end (up to relations) with $s_{i_m}$, and hence there exists a positive path $\upgamma$ such that $\upbeta \sim s_{i_m}\upgamma$. By the relations on paths in the Deligne groupoid, this implies that $\upalpha \sim s_{i_m}s_{i_m}\upgamma$. Now, as the assignment $\upalpha \mapsto F_\upalpha$ was shown in Corollary \ref{functordef} to give a functor $\cG_\cH \to \mathbb{F}$, we must have $F_\upalpha \cong F_{s_{i_m}s_{i_m}\upgamma}$ and hence
\begin{align*}
T=\upmu_\upalpha \Lambda_{\con} \cong \tT_\upalpha \cong \tT_{{i_m}{i_m}\upgamma} \cong \upmu_{{i_m}{i_m}\upgamma} \Lambda_{\con}.
\end{align*}
Using Theorem \ref{siltingmutation2} and the fact that $T$ is two-term,
\begin{align} \label{order}
\Lambda_{\con} > \upmu_\upgamma \Lambda_{\con}>  \upmu_{i_m} \upmu_\upgamma \Lambda_{\con} > T \geq \Lambda_{\con}[1]
\end{align}
and hence $\upmu_\upgamma \Lambda_{\con}$ is a two-term complex. But $\upmu_\upgamma \Lambda_{\con}$ and $T$ differ by at most one summand as $T$ is obtained from $\upmu_\upgamma \Lambda_{\con}$ by mutating at the same summand twice and thus, by Theorem \ref{siltingmutation2}(2'), $T$ and $\upmu_\upgamma \Lambda_{\con}$ are either isomorphic or related by a single mutation. However, combining \eqref{order} with Theorem \ref{siltingmutation2} gives a contradiction in both cases, and thus $\upalpha$ must be an atom. 
\end{proof}

\begin{theorem} \label{atomstwoterm}
Under the setup of \ref{compsetup}, choose a basic maximal rigid object $N \in \cmr$, and let $\Lambda_{\con} \colonequals \uEnd_R(N)$. Then there is a bijection 
\begin{align*}
\{ \text{atoms starting in $C_{N}$}\} \longrightarrow  \twotilt \Lambda_{\con},
\end{align*}
sending an atom $\upalpha$ to the tilting complex $\upmu_\upalpha \Lambda_{\con}$.
\end{theorem}
\begin{proof}
We first need to check this map is well defined, namely:
\begin{enumerate}
\item If $\upalpha \sim \upbeta$, then $\upmu_\upalpha\Lambda_{\con} \cong \upmu_\upbeta \Lambda_{\con}$;
\item For any atom $\upalpha$, $\upmu_\upalpha\Lambda_{\con}$ is a two-term tilting complex.
\end{enumerate}
The first statement follows easily from Corollary \ref{functordef}, where the assignment $\upalpha \to F_\upalpha$ is shown to yield a functor from the Deligne groupoid to the groupoid $\mathbb{F}$. Indeed, if $\upalpha \sim \upbeta$, then $F_\upalpha \cong F_\upbeta$ and hence, using part $(2)$ of Proposition \ref{mutationpaths}
\begin{align*}
\upmu_\upalpha\Lambda_{\con} \cong \tT_\upalpha \cong \tT_\upbeta \cong \upmu_\upbeta \Lambda_{\con}.
\end{align*}
For the second part, note that, since $\upalpha$ is an atom, Theorem \ref{composing} shows $\tT_\upalpha$ is zero outside degrees $0$ and $-1$ and hence has zero homology outside these degrees as well. Then, using $\upmu_\upalpha\Lambda_{\con} \cong \tT_\upalpha$ as above,
\begin{align*}
\Hom_{\Kb(\proj \Lambda_{\con})}(\Lambda_{\con}, \upmu_\upalpha\Lambda_{\con}[i]) &\cong  \Hom_{\Db( \Lambda_{\con})}(\Lambda_{\con}, \tT_\upalpha[i]) \\
&= H^i(\tT_\upalpha) =0 \ \text{if $i \neq 0,-1$}.
\end{align*}
Hence, $\Hom_{\Kb(\proj \Lambda_{\con})}(\Lambda_{\con}, \upmu_\upalpha\Lambda_{\con}[i]) =0 $ for all $i >0$ and so $\Lambda_{\con} \geq \upmu_\upalpha \Lambda_{\con}$. Similarly, using \cite[2.7]{tiltingconnectedness} and that $\Lambda_{\con}$ is symmetric, there is an isomorphism
\begin{align*}
\Hom_{\Kb(\proj \Lambda_{\con})}(\upmu_\upalpha\Lambda_{\con}, \Lambda_{\con}[i+1]) &\cong   D\Hom_{\Kb(\proj \Lambda_{\con})}(\Lambda_{\con}[i+1], \upmu_\upalpha\Lambda_{\con})
\end{align*}
which leads to isomorphisms
\begin{align*}
\Hom_{\Kb(\proj \Lambda_{\con})}(\upmu_\upalpha\Lambda_{\con}, \Lambda_{\con}[i+1]) 
&\cong D\Hom_{\Db( \Lambda_{\con})}(\Lambda_{\con}[i+1], \tT_\upalpha) \\
&= DH^{-(i+1)}(\tT_\upalpha) =0 \ \text{if $i \neq 0, -1$}.
\end{align*}
This shows that $\upmu_\upalpha \Lambda_{\con} \geq \Lambda_{\con}[1]$. Combining these gives that $\Lambda_{\con} \geq \upmu_\upalpha \Lambda_{\con} \geq \Lambda_{\con}[1]$ and hence $\upmu_\upalpha \Lambda_{\con}$ is a two-term tilting complex, as required.

Next we show the map is bijective. Recall from Theorem \ref{ctbij} that two-term tilting complexes are in bijection with maximal rigid objects in $\ucmr$ and hence with chambers in $\cH_M$ (of which there are finitely many). Further it is clear that atoms starting in a given chamber are also in bijection with the chambers. Hence,
\begin{align*}
\# \{ \text{atoms starting in $C_{N}$}\} = \# \twotilt \Lambda_{\con} 
\end{align*}
and so it is enough to show the given map is surjective.

As there are finitely many two-term complexes for $\Lambda_{\con}$, \cite[3.5]{tiltingconnectedness} shows that any two-term complex can be obtained from $\Lambda_{\con}$ by iterated left mutation. So given any $P \in \twotilt \Lambda_{\con}$, $P \cong \upmu_\upalpha \Lambda_ {\!\con}$ for some positive path $\upalpha$ starting in $C_N$. Lemma \ref{twotermatoms} shows that $\upalpha$ must be an atom. In particular, $\upalpha$ maps to $P$ under the given map, and so it is surjective.
\end{proof}
\begin{corollary} \label{twotermareatoms}
In the setup of \ref{compsetup}, choose a basic maximal rigid object $N$, and let $\Lambda_{\con} \colonequals \uEnd_R(N)$. Then, for any atom in $X_{\cH_M}$ starting in chamber $C_N$, the equivalence $F_\upalpha$ is induced by a two-term tilting complex. Moreoever, for any standard derived equivalence 
\begin{align*}
F \colonequals \RHom_{\Lambda_{\con}}(\tT, -) \colon \Db(\Lambda_{\con}) \to \Db(\Gamma)
\end{align*}
induced by a two-term tilting complex, there exists an atom $\upalpha$ starting in chamber $C_N$, and an isomorphism $\upgamma \colon \Gamma \to \Gamma_{\con} \colonequals \uEnd_R(\upnu_\upalpha N)$ such that $F \cong \RHom_{\Gamma}(_\upgamma\Gamma_{\con},-)\circ F_\upalpha$.
\end{corollary}
\begin{proof}
By Proposition \ref{mutationpaths}, $F_{\upalpha}$ is induced by $\upmu_\upalpha \Lambda_{\con}$ and by Theorem \ref{atomstwoterm}, this is a two-term tilting complex when $\upalpha$ is an atom, so the first statement follows. For the second statement, the assumption on $F$ means that $\tT$ is isomorphic to a two-term tilting complex for $\Lambda_{\con}$ and hence, by Theorem \ref{atomstwoterm}, $\tT \cong \upmu_{\upalpha}\Lambda_{\con}$ for some atom $\upalpha$ starting at $C_N$. But $\tT_\upalpha$ is also isomorphic to $\upmu_{\upalpha}\Lambda_{\con}$ by Proposition \ref{mutationpaths}, and thus if $\Gamma_{\con} \colonequals  \uEnd_R(\upnu_\upalpha N)$, Proposition \ref{rouq} shows that there is an algebra isomorphism $\upgamma \colon \Gamma \to \Gamma_{\con}$ such that
\begin{align*}
\tT \cong {}_\upgamma \Gamma_{\con} \otimes_{\Gamma_{\con}} \tT_{\upalpha} 
\end{align*}
which is equivalent to the statement.
\end{proof}
In other words, the standard equivalences from $\Db(\Lambda_{\con})$ induced by two-term tilting complexes of $\Lambda_{\con}$ are precisely (up to algebra isomorphism) the $F_\upalpha$ given by atoms $\upalpha$ in $X_{\cH_M}$ starting in chamber $C_N$. 

\begin{example} Continuing Example \ref{eightchambers}, the left hand diagram shows that all the atoms starting in chamber $C_M$. The two going to the opposite chamber are identified in $\cG_\cH$. Writing $\Lambda_{i_n \dots i_1} \colonequals \uEnd_R(M_{i_n \dots i_1})$, Corollary \ref{twotermareatoms} shows that the functors on the right hand side are all induced by two-term tilting complexes and further, they are the only standard equivalences from $\Db(\Lambda_{\con})$ (up to algebra isomorphism) which are induced by two-term tilting complexes. 
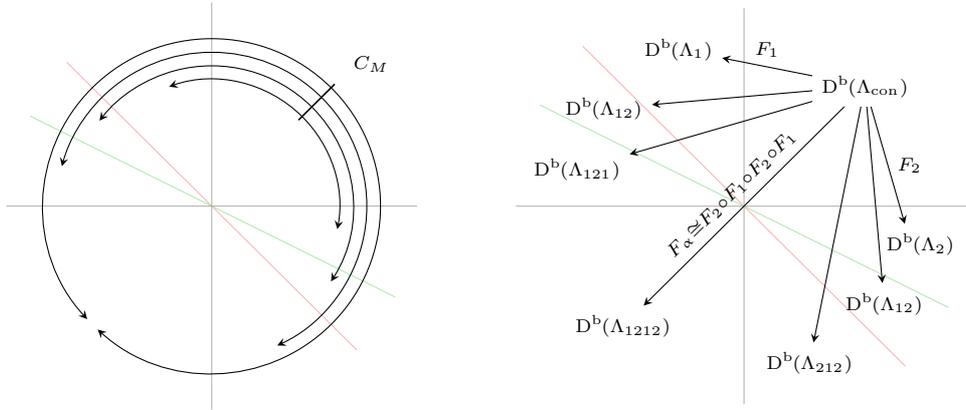
\[
\begin{tikzpicture}
\node at (-3.5,0) 
{\begin{tikzpicture}[scale=1.2,bend angle=15, looseness=1,>=stealth]
\coordinate (A1) at (135:2.25cm);
\coordinate (A2) at (-45:2.25cm);
\coordinate (B1) at (153.435:2.225cm);
\coordinate (B2) at (-26.565:2.25cm);
\draw[red!30] (A1) -- (A2);
\draw[black!30] (-2.25,0)--(2.25,0);
\draw[black!30] (0,-2.25)--(0,2.25);
\draw[->] ([shift=(45:1.4cm)]0,0) arc (45:109.5:1.4cm);
\draw[->] ([shift=(45:1.55cm)]0,0) arc (45:157.5:1.55cm);
\draw[->] ([shift=(45:1.7cm)]0,0) arc (45:222.5:1.7cm);
\draw[->] ([shift=(45:1.4cm)]0,0) arc (45:-22.5:1.4cm);
\draw[->] ([shift=(45:1.55cm)]0,0) arc (45:-65:1.55cm);
\draw[->] ([shift=(45:1.7cm)]0,0) arc (45:-132.5:1.7cm);
\filldraw[fill=black] (45.1:1.35cm) -- (45.1:1.75cm) -- (44.9:1.75cm) -- (44.9:1.35cm) -- cycle;
\node at (42:2.25cm) {$\scriptstyle C_M$};
\end{tikzpicture}};
\node at (3.5,0) 
{\begin{tikzpicture}[scale=1.5,bend angle=15, looseness=1,>=stealth]
\coordinate (A1) at (135:2cm);
\coordinate (A2) at (-45:2cm);
\coordinate (B1) at (153.435:2cm);
\coordinate (B2) at (-26.565:2cm);
\draw[red!30] (A1) -- (A2);
\draw[black!30] (-2,0)--(2,0);
\draw[black!30] (0,-1.75)--(0,1.75);
\node (C+) at (45:1.5cm) {$\scriptstyle \Db(\Lambda_{\con})$};
\node (C1) at (112.5:1.5cm) {$\scriptstyle \Db(\upnu_1\Lambda_{\con})$};
\node (C2) at (161:1.52cm) {$\scriptstyle \Db(\upnu_2\upnu_1\Lambda_{\con})$};
\node (C-) at (225:1.5cm) {$\scriptstyle \Db(\upnu_1\upnu_2\upnu_1\Lambda_{\con})$};
\node (C4) at (-65:1.5cm) {$\scriptstyle \Db(\upnu_1\upnu_2\Lambda_{\con})$};
\node (C5) at (-22.5:1.5cm) {$\scriptstyle \Db(\upnu_2 \Lambda_{\con})$};
\node at (-0.4,0.95) {$\scriptstyle F_2 \circ F_1$};
\node at (0.53,-0.2) {$\scriptstyle F_1 \circ F_2$};
\draw[->]  (C+) -- node[above] {$\scriptstyle F_1$} (C1);
\draw[->]  (C+) --  (C2);
\draw[->]  (C+)-- node[above,rotate=45] {$ \scriptstyle F_\upalpha \cong  F_1 \circ  F_2 \circ  F_1$} (C-);
\draw[->]  (C+) --  (C4);
\draw[->]  (C+) -- node[right] {$\scriptstyle F_2$} (C5);
\end{tikzpicture}};
\end{tikzpicture}
\]
\end{example}

Finally, the following shows that we can view the results of this section as a two-sided improvement of the bijection of Theorem \ref{ctbij}.

\begin{corollary} \label{twosidedctbij}
In the setup of \ref{compsetup}, choose a basic maximal rigid object $N$, and let $\Lambda_{\con} \colonequals \uEnd_R(N)$. Given an atom $\upalpha \colon N \to \upnu_{\upalpha}N$, the bimodule complex \eqref{bimodulecomplex} is isomorphic in $\Db(\Lambda_{\con})$ to the two-term tilting complex $P$ of $\Lambda_{\con}$ associated to $\upnu_{\upalpha}N$ via the bijection in Theorem \ref{ctbij}. Then, 
\begin{align*}
\End_{\Lambda_{\con}}(P) \cong \uEnd_R(\upnu_\upalpha N)
\end{align*}
so that, in this case, the bijection of Theorem \ref{ctbij} preserves endomorphism rings.
\end{corollary}
\begin{proof}
By Proposition \ref{mutationpaths}, \eqref{bimodulecomplex} is isomorphic in $\Db(\Lambda_{\con})$ to $\upmu_\upalpha\Lambda_{\con}$, which is two-term by Theorem \ref{atomstwoterm}. As the bijection of Theorem \ref{ctbij} preserves mutation, this shows that $\upmu_\upalpha\Lambda_{\con}$ is precisely the two-term tilting complex associated to $\upnu_\upalpha N$, completing the proof of the first statement. Now, by Proposition \ref{mutationpaths}, $F_{\upalpha}(\upmu_\upalpha\Lambda_{\con}) \cong \uEnd_R(\upnu_\upalpha N)$ and so using that $F_\upalpha$ is an equivalence gives the second statement.
\end{proof}

\section{Faithfulness} \label{faithfulness}

In this section we show the functor from Corollary \ref{functordef} is in fact faithful by adapting the strategy used in \cite[\S 6]{faithful}. By \cite[2.11]{faithful}, this problem can be immediately reduced to checking the functor is faithful on positive paths. In particular, if $\upalpha, \upbeta$ are positive paths with $\Phi(\upalpha) \cong \Phi(\upbeta)$ (or equivalently $F_\upalpha \cong F_\upbeta$ in the notation of \ref{compfun}), then we need to show that $\upalpha \sim \upbeta$. For this, we need an effective way of telling when two positive paths are equivalent, for which we will use the Deligne normal form.  

\subsection{Deligne Normal Form}
As with the Deligne groupoid, our description of the Deligne normal form will follow \cite{faithful}. Take a hyperplane arrangement $\cH$ and its directed skeleton graph $X_\cH$. For positive paths $p,q$ in $X_{\cH}$ with $s(p)=s(q)$, we say $p$ begins with $q$ if there exists a positive path $r$ such that $s(r)=t(q)$, $t(r)=t(p)$ and $p \sim rq$. For a positive path $p$ let $\mathrm{Begin}(p)$ denote the set of all atoms with which $p$ begins.
\begin{lemma} \cite[2.2]{hyperplane3}
For each positive path $p$ in $X_{\cH}$, there exists a unique (up to the relations) atom $\upalpha$ such that $\mathrm{Begin}(p)=\mathrm{Begin}(\upalpha)$.
\end{lemma}
\begin{definition}
Take $p$ to be any positive path in $X_{\cH}$ and let $\upalpha_1$ be the unique atom such that $\mathrm{Begin}(p)=\mathrm{Begin}(\upalpha_1)$. Then $p$ begins with $\upalpha_1$ and so there exists a positive $\upbeta$ such that 
\begin{align*}
p \sim \upbeta \upalpha_1.
\end{align*}
Continuing this process with $\upbeta$, we decompose $p$ as 
\begin{align*}
p \sim \upalpha_n \dots \upalpha_1
\end{align*}
which we refer to as the \emph{Deligne normal form} of $p$.
\end{definition}
If $p$ and $q$ are positive paths with Deligne normal forms $\upalpha_n \dots \upalpha_1$ and $\upbeta_m \dots \upbeta_1$ respectively, it is clear that $p \sim q$ if and only if $n=m$ and $\upalpha_i \sim \upbeta_i$ for each $i$. The following well-known lemma is useful for determining the Deligne normal form of a given path.
\begin{lemma}\cite[5.1]{faithful} \label{delignehelp}
If $p \sim \upalpha_n \dots  \upalpha_1$ is in Deligne normal form then for each $k \in \{2, \dots, n\}$, $\upalpha_k$ must start (up to relations) by crossing the wall that $\upalpha_{k-1}$ crosses through last. 
\end{lemma}
\begin{example}\label{deligneformexample}
Continuing Example \ref{eightchambers}, the positive path
\[
\begin{array}{cc}
\begin{array}{c}
\begin{tikzpicture}[scale=0.6,bend angle=15, looseness=1,>=stealth]
\coordinate (A1) at (135:2cm);
\coordinate (A2) at (-45:2cm);
\coordinate (B1) at (153.435:2cm);
\coordinate (B2) at (-26.565:2cm);
\draw[red!30] (A1) -- (A2);
\draw[black!30] (-2,0)--(2,0);
\draw[black!30] (0,-2)--(0,2);
\draw[->] ([shift=(50:1.4cm)]0,0) arc (50:290:1.4cm);
\node at (50:1.4cm) [DWs] {};
\node at (292.5:1.4cm) [DWs] {};
\end{tikzpicture}
\end{array}
&
p= s_2 s_1 s_2 s_1
\end{array}
\]
starting in chamber $C_M$ has Deligne normal form $(s_2) (s_1 s_2 s_1)$. Although $s_1 s_2 s_1$ does not appear to end with $s_2$, under the relations it is equivalent to $s_2 s_1 s_2$ which clearly does. 
\end{example}
\subsection{Tracking Simples}
The key idea of this section is that it is possible to compute the Deligne normal form of a positive path $\upalpha$ by tracking where the functor $F_\upalpha$ sends simple modules. 

\begin{notation}
Under the setup of \ref{compsetup}, given any basic maximal rigid object $N \colonequals \bigoplus\limits_{i=0}^n N_i$, the contraction algebra $\Lambda_{\con} \colonequals \uEnd_R(N)$ has $n$ simple modules. We will abuse notation and denote these as $S_1, \dots, S_n$, where the projective cover of $S_i$ is $\uHom_R(N,N_i)$. Note that each $S_i$ is also a simple module when considered as a module over $\Lambda \colonequals \End_R(N)$ and if we wish to view $S_i$ in that way, we will write it as $(S_i)_{\Lambda}$.   
\end{notation}

The following technical lemma will be used repeatedly.
\begin{lemma} \label{technical}
Under setup \ref{compsetup}, let $N \in \cmr$ be a basic maximal rigid object and write $\Lambda \colonequals \End_R(N)$ and $\Lambda_{\con} \colonequals \uEnd_R(N)$. Suppose that $X$ is a $\Lambda$-module, and $Y \in \Db(\Lambda_{\con})$ is such that
\begin{align*}
X[n] \cong Y \otimes_{\Lambda_{\con}} (\Lambda_{\con})_{\Lambda}
\end{align*}
in $\Db(\Lambda)$. Then, $X$ is a $\Lambda_{\con}$-module and $X \cong Y[-n]$ in $\Db(\Lambda_{\con})$.
\end{lemma}
\begin{proof}
Since $- \otimes_{\Lambda_{\con}} (\Lambda_{\con})_{\Lambda}$ is an exact functor, it preserves cohomology and hence there is an isomorphism of $\Lambda$-modules
\begin{align*}
\upphi_i \colon H^i(X[n]) \to H^{i}(Y)
\end{align*}
for all $i\in \mathbb{Z}$. When $i=-n$ this gives an isomorphism
\begin{align*}
\upphi \colon X \to H^{-n}(Y).
\end{align*}
But as $H^{-n}(Y)$ is a $\Lambda_{\con}$-module, it is annihilated by $I$, where $\Lambda_{\con} \cong \Lambda/I$. Hence, if $f \in I$, and $x \in X$, then 
\begin{align*}
\upphi(x . f)=\upphi(x) . f=0
\end{align*}
and hence, as $\upphi$ is an isomorphism, $x. f=0$. Thus, $X$ is annihilated by $I$ and so is a $\Lambda_{\con}$-module. If $i \neq -n$ then the above shows that the homology of $Y$ in degree $i$ is zero and hence
\begin{align*}
Y[-n] \cong H^{-n}(Y) \cong X
\end{align*}
in $\Db(\Lambda_{\con})$ as required. 
\end{proof}
Recall that associated to any arrow $s_i \colon N \to \upnu_i N$ in $X_{\cH_M}$, there are equivalences $F_i$ and $G_i$ as in notation \ref{funnot}. For a positive path $\upalpha \colonequals s_{i_m} \dots s_{i_1}$, $F_\upalpha$ and $G_\upalpha$ will denote the composition of the corresponding functors, as in notation \ref{compfun}. Further, $T_\upalpha$ will be the two-sided tilting complex inducing $G_\upalpha$, defined as in \eqref{talpha}. The following known result tracks simple modules through the $G_\upalpha$. 

\begin{lemma} \cite[\S 5]{faithful}\label{trackbigsimple}
Under the setup of \ref{compsetup}, let $\upalpha \colon C_L \to C_N$ be an atom in $X_{\cH_M}$. Writing $\Lambda \colonequals \End_R(L)$ and $\Gamma \colonequals \End_R(N)$, the following statements hold.
\begin{enumerate}\item $
G_{\upalpha}^{-1}((S_i)_{\Gamma}) \cong 
\left\{
	\begin{array}{ll}
		\Tor^{\Gamma}_1(S_i, T_{\upalpha})[1]  & \mbox{if $\upalpha$ ends (up to relations) with $s_i$;} \\
		S_i \otimes_{\Gamma}T_{\upalpha} & \mbox{otherwise. }
	\end{array}
\right.
$
\item If $\upalpha$ ends (up to relations) with $s_i$, then there exists $j \in \{1, \dots, n\}$ such that $S_j \hookrightarrow \Tor^{\Gamma}_1(S_i, T_{\upalpha})$ and $\upalpha$ starts (up to relations) with $s_j$. 
\end{enumerate}
\end{lemma}
The commutative diagram \eqref{commutingsquare} allows us to prove the corresponding result for contraction algebras.
\begin{corollary} \label{tracksimple}
Under the setup of \ref{compsetup}, let $\upalpha \colon C_L \to C_N$ be an atom in $X_{\cH_M}$. Writing $\Lambda \colonequals \End_R(L)$ and $\Gamma \colonequals \End_R(N)$, then  \begin{align*}
F_{\upalpha}^{-1}(S_i) \cong 
\left\{
	\begin{array}{ll}
		\Tor^{\Gamma}_1(S_i, T_{\upalpha})[1]  & \mbox{if $\upalpha$ ends (up to relations) with $s_i$;} \\
		S_i \otimes_{\Gamma}T_{\upalpha} & \mbox{otherwise. }
	\end{array}
\right.
\end{align*}
\end{corollary}
\begin{proof}
The commutative diagram \eqref{commutingsquare} shows that
\begin{align*}
F_{\upalpha}^{-1}(S_i) \otimes_{\Lambda_{\con}} (\Lambda_{\con})_{\Lambda} \cong G_{\upalpha}^{-1}((S_i)_{\Gamma}).
\end{align*}
Combining  Lemmas \ref{trackbigsimple} and \ref{technical} then gives the result.
\end{proof}

The following technical lemma is needed for the main result of this section.

\begin{lemma} \label{generalmodules}
Under the setup of \ref{compsetup}, let $N \in \cmr$ be a basic maximal rigid object and let $\Lambda_{\con}\colonequals \uEnd_R(N)$. Suppose that $X \in \mod \Lambda_{\con}$ is nonzero.
\begin{enumerate}
\item If $\Ext^{\geq p}_{\Lambda_{\con}}(S_i,X)=0$ for all $1 \leq i \leq n$, then $\Ext^{\geq p}_{\Lambda_{\con}}(-,X)=0$. \label{generalmodules1}
\item $\Ext^i_{\Lambda_{\con}}(-, \Lambda_{\con}) =0$ if $i \geq 1$. \label{generalmodules2}
\item $\Hom_{\Lambda_{\con}}(X, \Lambda_{\con})  \neq 0$. \label{generalmodules3}
\end{enumerate}
\end{lemma}
\begin{proof}
\begin{enumerate}[leftmargin=0cm,itemindent=.6cm,labelwidth=\itemindent,labelsep=0cm,align=left]
\item Choose $Y \in \mod \Lambda_{\con}$. We need to show that $\Ext^{\geq p}_{\Lambda_{\con}}(Y,X)=0$. Filtering $Y$ by simple modules, an easy induction on the length of $Y$ establishes the result. 
\item $\Lambda_{\con}$ is a symmetric algebra by Proposition \ref{symmalg}, and thus is self injective.
\item Since $\Lambda_{\con}$ is a symmetric algebra there is an isomorphism
\begin{align*}
\Hom_{\Lambda_{\con}}(X, \Lambda_{\con}) \cong \Hom_k(X,k)
\end{align*} 
for all $X \in \mod \Lambda_{\con}$ \cite[2.7]{broue}. As $\Hom_k(-,k) \colon \mod \Lambda_{\con} \to \mod \Lambda_{\con}^{\mathrm{op}}$ is a duality, the statement follows. \qedhere
\end{enumerate}
\end{proof}
\vspace{-0.2cm}
The following is the main technical result of this section and it mirrors \cite[3.1]{bravthomas} and \cite[6.3]{faithful}.
\begin{proposition} \label{faithful2}
Under the setup of $\ref{compsetup}$, let $\upalpha \colon C_L \to C_N$ be a positive path in $X_{\cH_M}$ with Deligne normal form $\upalpha = \upalpha_k  \dots \upalpha_1$. Writing $\Lambda_{\con} \colonequals \uEnd_R(L)$ and $\Gamma_{\!\con} \colonequals \uEnd_R(N)$, then the following statements hold.
\begin{enumerate}
\item $\Ext_{\Gamma_{\!\con}}^{> k}(-,F_{\upalpha}(\Lambda_{\con}))=0$.
\item $\Ext_{\Gamma_{\!\con}}^{k}(S_i,F_{\upalpha}(\Lambda_{\con})) \neq 0$ if and only if $\upalpha_k$ ends (up to relations) with $s_i$.  
\item $\mathrm{max}\Big\{p \Bigm| \Ext_{\Gamma_{\!\con}}^{p}(\bigoplus\limits_{i=1}^n S_i,F_{\upalpha}(\Lambda_{\con})) \neq 0\Big\} = k$.
\end{enumerate}
\end{proposition}

\begin{proof}
Part $(3)$ is clearly a consequence of the first two parts. We prove parts $(1)$ and $(2)$ together using induction on $k$.

\noindent \textbf{Base case: $\mathbf{k=1}$.}
There are two cases to consider, namely if $\upalpha$ ends (up to relations) with $s_i$ or not. From now on, for ease of reading, we will omit the statement `up to relations'. 
\begin{enumerate}[label=(\alph*),leftmargin=0cm,itemindent=.6cm,labelwidth=\itemindent,labelsep=0cm,align=left]
\item If $\upalpha$ does not end with $s_i$, then as $k=1$, $\upalpha$ is an atom and so by Corollary \ref{tracksimple}, 
\begin{align*}
F^{-1}_{\upalpha}(S_i) \cong X
\end{align*} 
for some $\Lambda_{\con}$-module $X$. Then,
\begin{align*}
\Ext_{\Gamma_{\!\con}}^{\geq 1}(S_i,F_{\upalpha}(\Lambda_{\con})) &\cong \Ext_{\Lambda_{\con}}^{\geq1}(F_{\upalpha}^{-1}(S_i),\Lambda_{\con})\\
&\cong \Ext_{\Lambda_{\con}}^{\geq 1}(X,\Lambda_{\con}) = 0
\end{align*}
where the last equality holds by Lemma \ref{generalmodules}\ref{generalmodules2}.
\item If $\upalpha$ ends with $s_i$, then by Corollary \ref{tracksimple}, 
\begin{align*}
F^{-1}_{\upalpha}(S_i) \cong Y[1]
\end{align*} 
for some $\Lambda_{\con}$-module $Y$. Then,
\begin{align*}
\Ext_{\Gamma_{\!\con}}^1(S_i,F_{\upalpha}(\Lambda_{\con})) &\cong \Ext_{\Lambda_{\con}}^1(F_{\upalpha}^{-1}(S_i),\Lambda_{\con})\\
&\cong \Ext_{\Lambda_{\con}}^1(Y[1],\Lambda_{\con})\\
&\cong \Hom_{\Lambda_{\con}}(Y,\Lambda_{\con}) \neq 0
\end{align*}
where the last part comes from Lemma \ref{generalmodules}\ref{generalmodules3}. A similar calculation gives
\begin{align*}
\Ext_{\Gamma_{\!\con}}^{\geq 2}(S_i,F_{\upalpha}(\Lambda_{\con})) \cong \Ext_{\Lambda_{\con}}^{\geq 1}(Y,\Lambda_{\con}) =0.
\end{align*}
\end{enumerate}
Combining the two cases shows part $(2)$ of the result when $k=1$. Further, it shows that $\Ext_{\Gamma_{\!\con}}^{ \geq 2}(S_i,F_{\upalpha}(\Lambda_{\con}))=0$ for all $i \in \{1, \dots, n\}$. Applying Lemma \ref{generalmodules}\ref{generalmodules1} gives $\Ext_{\Gamma_{\!\con}}^{ \geq 2}(-,F_{\upalpha}(\Lambda_{\con}))=0$, also proving part $(1)$ for $k=1$. \\

\noindent \textbf{Inductive Step.} We now assume that the result is true for all paths with less than $k$ Deligne factors. Write $\upalpha = \upalpha_k \upbeta$ where $\upbeta \colonequals \upalpha_{k-1}  \dots \upalpha_1$. Write $\Delta_{\con}$ for the contraction algebra associated to the chamber at the end of $\upbeta$. By the inductive hypothesis,
\begin{align}
\Ext_{\Delta_{\con}}^{\geq k}(-,F_{\upbeta}(\Lambda_{\con})) = 0 \label{induct}
\end{align}
and further, $\Ext_{\Delta_{\con}}^{ k-1}(S_i,F_{\upbeta}(\Lambda_{\con})) \neq 0$ if and only if $\upbeta$ ends with $s_i$. Again, we consider two cases:
\begin{enumerate}[label=(\alph*),leftmargin=0cm,itemindent=.6cm,labelwidth=\itemindent,labelsep=0cm,align=left]
\item If $\upalpha_k$ does not end with $s_i$, then by Corollary \ref{tracksimple}, 
\begin{align*}
F^{-1}_{\upalpha_k}(S_i) \cong X
\end{align*} 
for some $\Delta_{\con}$-module $X$. Then,
\begin{align*}
\Ext_{\Gamma_{\!\con}}^{\geq k}(S_i,F_{\upalpha}(\Lambda_{\con})) &\cong \Ext_{\Delta_{\con}}^{\geq k}(F_{\upalpha_k}^{-1}(S_i),F_{\upbeta}(\Lambda_{\con}))\\
&\cong \Ext_{\Delta_{\con}}^{\geq k}(X,F_{\upbeta}(\Lambda_{\con}))\\
&= 0
\end{align*}
where the last equality holds by the inductive hypothesis \eqref{induct}.
\item If $\upalpha_{k}$ ends with $s_i$, then by Corollary \ref{tracksimple}, 
\begin{align*}
F^{-1}_{\upalpha_k}(S_i) \cong Y[1]
\end{align*} 
where $Y \colonequals \Tor^{\Gamma}_1(S_i, T_{\upalpha})$ is a $\Delta_{\con}$-module. Then
\begin{align*}
\Ext_{\Gamma_{\!\con}}^{\geq k+1}(S_i,F_{\upalpha}(\Lambda_{\con})) \cong \Ext_{\Delta_{\con}}^{\geq k}(Y,F_{\upbeta}(\Lambda_{\con}))=0,
\end{align*}
again using the inductive hypothesis \eqref{induct}. Further,
\begin{align*}
\Ext_{\Gamma_{\!\con}}^{k}(S_i,F_{\upalpha}(\Lambda_{\con})) \cong \Ext_{\Delta_{\con}}^{k-1}(Y,F_{\upbeta}(\Lambda_{\con}))
\end{align*}
and so it suffices to show that $\Ext_{\Delta_{\con}}^{k-1}(Y,F_{\upbeta}(\Lambda_{\con})) \neq 0$. As $Y \colonequals \Tor^{\Gamma}_1(S_i, T_{\upalpha})$, Lemma \ref{trackbigsimple}$(2)$ shows there exists a simple module $S_j$ of $\Delta_{\con}$ such that $S_j \hookrightarrow Y$ and $\upalpha_k$ starts with $s_j$. Applying $\Hom_{\Delta_{\con}}(-,F_\upbeta(\Lambda_{\con}))$ to the short exact sequence
\begin{align*}
0 \to S_j \to Y \to Y/S_j \to 0,
\end{align*}
gives a long exact sequence
\begin{align*}
\dots \to \Ext_{\Delta_{\con}}^{k-1}(Y,F_{\upbeta}(\Lambda_{\con})) \to \Ext_{\Delta_{\con}}^{k-1}(S_j,F_{\upbeta}(\Lambda_{\con})) \to \Ext_{\Delta_{\con}}^{k}(Y/S_j,F_{\upbeta}(\Lambda_{\con})) \to \cdots
\end{align*}
where the last term is zero by the inductive hypothesis \eqref{induct}. Since $\upalpha_k$ starts with $s_j$, Lemma \ref{delignehelp} shows $\upbeta$ must also end with $s_j$ otherwise $\upalpha$ would not be in Deligne form. Thus, by the inductive hypothesis,
\begin{align*}
\Ext_{\Delta_{\con}}^{k-1}(S_j,F_{\upbeta}(\Lambda_{\con})) \neq 0
\end{align*}
and hence by the exact sequence, $\Ext_{\Delta_{\con}}^{k-1}(Y,F_{\upbeta}(\Lambda_{\con})) \neq 0$, completing the proof. \qedhere
\end{enumerate} 
\end{proof}
\vspace{0.2cm}
\begin{corollary} \label{faithful}
The functor $\Phi$ from Corollary \ref{functordef} is a faithful functor.
\end{corollary}
\begin{proof}
This follows exactly as in \cite[6.5]{faithful} or \cite[3.1]{bravthomas}.
\end{proof}
The following is an immediate consequence of Corollary \ref{faithful} and is the main result of this section.
\begin{corollary} \label{groupaction}
Suppose that $f \colon X \to \Spec R$ is a minimal model of a complete local isolated cDV singularity, that $\cH$ in $\mathbb{R}^n$ is the associated hyperplane arrangement, and that $\Lambda_{\con}$ is the associated contraction algebra. Then there is an injective group homomorphism
\begin{align*}
\pi_1(\mathbb{C}^n \backslash \cH_{\mathbb{C}}) \to \mathrm{Auteq}(\Db(\Lambda_{\con}))
\end{align*}
where  $\cH_{\mathbb{C}}$ is the complexification of $\cH$.
\end{corollary}
This shows the group structure of the autoequivalences from compositions of flop functors (which is shown to be $\pi_1(\mathbb{C}^n \backslash \cH_{\mathbb{C}})$ in \cite[6.7]{faithful}) is the same as the group structure of the autoequivalences from compositions of the $F_i$. Thus, the symmetry group coming from the geometry can be seen by studying the derived category of the contraction algebras, giving evidence towards the conjecture \cite[1.3]{rigidequiv} of Donovan--Wemyss. 


\section{Visualising the Derived Equivalence Class} \label{picture}
In this final section, we first show that the hyperplane arrangement $\cH_M$ from Setup \ref{compsetup} can be constructed from the two-term tilting theory of a contraction algebra. Then we combine the main results of this paper with the author's previous work \cite{rigidequiv} to obtain a full picture of the derived equivalence class of a contraction algebra.
To construct the hyperplane arrangement from a contraction algebra, we use \emph{g-vectors} as studied in \cite{dij}. Recall that if $A$ is a finite dimensional $k$-algebra with indecomposable projective modules $P_1, \dots, P_n$, then the $g$-vector of a complex 
\begin{align*}
P \colonequals \bigoplus_{i=1}^n P_i^{b_i} \to \bigoplus_{i=1}^n P_i^{a_i}
\end{align*}
is defined as $g(P)\colonequals (a_i -b_i)_{i=1}^n \in \mathbb{R}^n$. Thus, given a two-term silting complex $T \colonequals \bigoplus_{i=1}^n T_i$, there are $n$ associated $g$-vectors $g(T_1), \dots g(T_n)$ and the positive cone of these is
\begin{align*}
C_T \colonequals \Big\{ \sum_{i=1}^{n} \upvartheta_ig(T_i) \Bigm| \upvartheta_i >0 \ \forall 1 \leq i \leq n \Big\} \subset \mathbb{R}^n.
\end{align*}
The union of these cones across all two-term silting complexes, which is known to be disjoint from \cite[6.7]{dij}, is called the \emph{$g$-vector fan} of $A$.  
\begin{theorem} \label{intrinsic}
Suppose that $f \colon X \to \Spec R$ is a minimal model of a complete local isolated cDV singularity, that $\cH$ in $\mathbb{R}^n$ is the associated hyperplane arrangement, and that  $\Lambda_{\con}$ is the associated contraction algebra. \begin{enumerate}
\item The chambers of $\cH$ are precisely those in the $g$-vector fan of $\Lambda_{\con}$.
\item The subgroup of $\mathrm{Auteq}(\Db(\coh X))$ formed by flop functors, their compositions and inverses can be recovered from $\Lambda_{\con}$.
\end{enumerate}
\end{theorem}
\begin{proof} 
\begin{enumerate}[leftmargin=0cm,itemindent=.6cm,labelwidth=\itemindent,labelsep=0cm,align=left]
\item Suppose that $M \colonequals \bigoplus_{i=0}^n M_i \in \cmr$, with $M_0 \cong R$, is the basic maximal rigid object associated to $f$ so that $\Lambda_{\con} \cong \uEnd_R(M)$. Label the projective $\Lambda_{\con}$-modules $P_i \colonequals \uHom_R(M,M_i)$. 
By the bijection in Theorem \ref{ctbij}, each two-term tilting complex $T \colonequals \bigoplus_{i=1}^nT_i$ corresponds to a basic maximal rigid object $N \colonequals \bigoplus_{i=0}^n N_i$ in $\cmr$ with $N_0 \cong R$. For each $0 \leq i \leq n$, \cite[4.12]{mmas} shows that there exists an exact sequence
\begin{align}
0 \to  \bigoplus_{j=0}^n M_j^{{b_{ij}}} \xrightarrow{f_i}  \bigoplus_{j=0}^n M_j^{{a_{ij}}} \xrightarrow{g_i} N_i \to 0. \label{gvector}
\end{align}
As $M$ is rigid, applying $\Hom_R(M,-)$ to each sequence yields a projective resolution of $\Hom_R(M,N_i)$ and hence, by \cite[4.6(2)]{faithful}, the chamber associated to $N$ in the hyperplane arrangement $\cH$ is
\begin{align*}
C_N = \Big\{ \sum_{i=1}^{n} \upvartheta_i \Big( \sum_{j=1}^{n}(a_{ij}-b_{ij}) \mathbf{e_j} \Big) \Bigm| \upvartheta_i >0 \ \forall 1 \leq i \leq n \Big\} \subset \mathbb{R}^n,
\end{align*}
where $\mathbf{e_j}$ is the $j^{th}$ standard basis vector of $\mathbb{R}^n$. However, the exact sequences \eqref{gvector} descend to the triangles
\begin{align*}
\bigoplus_{j=1}^n M_j^{{b_{ij}}} \xrightarrow{f_i}  \bigoplus_{j=1}^n M_j^{{a_{ij}}} \xrightarrow{g_i} N_i 
\end{align*}
in $\ucmr$ which shows that, under the bijection of Theorem \ref{ctbij} (see \cite[4.7]{[AIR]} for the details), $N_i$ corresponds to the two-term complex
\begin{align*}
\bigoplus_{j=1}^n \uHom_R(M,M_j)^{{b_{ij}}} \xrightarrow{}  \bigoplus_{j=1}^n \uHom_R(M,M_j)^{{a_{ij}}}.
\end{align*} 
By assumption, $N_i$ corresponds to $T_i$ and so we deduce that $g(T_i)=(a_{ij}-b_{ij})_{j=1}^n$. Thus, $C_T=C_N$ and so, as the $C_N$ sweep out the chambers of $\cH$, so do the $C_T$. 
\item If $\cH$ is the hyperplane arrangement associated to $f$, then \cite[6.7]{faithful} shows that the subgroup of $\mathrm{Auteq}(\Db(\coh X))$ formed by flop functors, their compositions and inverses is precisely $\fundgp(\mathbb{C}^n \backslash \cH_{\mathbb{C}})$. By part (1), we can recover $\cH$ from $\Lambda_{\con}$ and hence also the group $\fundgp(\mathbb{C}^n \backslash \cH_{\mathbb{C}})$ as required. \qedhere
\end{enumerate}
\end{proof}
This provides further evidence to suggest that the contraction algebras of $\Spec R$ control the geometry. However, perhaps more importantly, it also shows that the hyperplane arrangement, which we have seen controls the derived equivalence class of the contraction algebras, is intrinsic to the algebra. Thus, starting with just a contraction algebra, we can recover a `picture' of its entire derived equivalence class in the following sense.

Given a contraction algebra $\Lambda_{\con}$ of a complete local isolated cDV singularity $\Spec R$, let $\cH$ be the hyperplane arrangement determined by two-term tilting complexes. By Theorem \ref{intrinsic}, this matches the hyperplane arrangement $\cH_M$ for some maximal rigid object $M \in \cmr$ such that $\Lambda_{\con} \cong \uEnd_R(M)$. In particular, the chambers of $\cH$ naturally correspond to maximal rigid objects and paths correspond to the derived equivalences $F_\upalpha$. Further, by Theorem \ref{chambersmutate}, the directed skeleton graph of $\cH$ is the double of the mutation graph of maximal rigid objects. The latter object was studied in \cite{rigidequiv} and we now combine the results from there with the results from this paper to obtain the following summary theorem, where all the results are contained in other parts of these two papers, but here we group them all together to give the overall picture. For clarity, note that if an arrow $s_i$ is assigned the functor $F_i$, the path corresponding to travelling along this arrow backwards, namely $s_i^{-1}$, is assigned the functor $F_i^{-1}$.

\begin{theorem} \label{summary} 
Given a contraction algebra $\Lambda_{\con}$ of a complete local isolated cDV singularity $\Spec R$, let $\cH$ be the hyperplane arrangement determined by two-term tilting complexes. Choosing a maximal rigid object $M \in \cmr$ such that $\Lambda_{\con} \cong \uEnd_R(M)$, then the following hold.
\begin{enumerate}
\item The only basic algebras in the derived equivalence class of $\Lambda_{\con}$ are the contraction algebras of $\Spec R$; or equivalently, the endomorphism algebras of two-term tilting complexes of $\Lambda_{\con}$. In particular, there are finitely many such algebras.
\item For any standard derived equivalence 
\begin{align*}
F \colonequals \RHom_{\Lambda_{\con}}(\tT, -) \colon \Db(\Lambda_{\con}) \to \Db(\Gamma),
\end{align*}
there exists a (not necessarily positive) path $\upalpha$ in $X_{\cH}$ starting in chamber $C_M$, and an algebra isomorphism $\upgamma \colon \Gamma \to \Gamma_{\con} \colonequals \uEnd_R(\upnu_\upalpha M)$ such that 
\begin{align*}
F \cong \RHom_{\Gamma}(_\upgamma\Gamma_{\con},-)\circ F_\upalpha,
\end{align*}
where $F_\upalpha$ is as in notation \ref{compfun}. In other words, any standard derived equivalence (up to algebra isomorphism) from $\Lambda_{\con}$ is obtained as a path starting at $C_M$. 
\item All standard derived equivalences from $\Lambda_{\con}$ are the composition of two-term tilts and their inverses.
\item The standard autoequivalences of $\Db(\Lambda_{\con})$ are determined precisely (up to algebra isomorphism) by paths $C_M \to C_N$ in the mutation graph, where $N$ satisfies $\uEnd_R(N) \cong \Lambda_{\con}$. 
\item If $\upalpha, \upbeta \colon C_M \to C_N$ are two paths, then $F_\upalpha \cong F_\upbeta$ if and only if $\upalpha \sim \upbeta$ in the Deligne groupoid. In other words, the $F_i$ satisfy precisely the Deligne groupoid relations.
\item For any atom $\upalpha \colon C_M \to C_N$, the two-sided tilting complex giving $F_\upalpha$ is given in Corollary \ref{directequivalences}.
\item The atoms starting at $C_M$ determine precisely (up to algebra isomorphism) the standard equivalences induced by two-term tilting complexes of $\Lambda_{\con}$.
\end{enumerate}
\end{theorem} 
\begin{proof}
The first statement of part $(1)$ is \cite[4.12]{rigidequiv} while the second follows as the bijection between maximal rigid objects and two-term tilting complexes preserves endomorphism rings by Corollary \ref{twosidedctbij}. Part $(2)$ is \cite[3.16]{rigidequiv}, which uses Proposition \ref{rouq} and the fact that $\tT$ must be isomorphic to $\upmu_\upalpha \Lambda_{\con}$ for some path $\upalpha$, similar to the proof of Corollary \ref{twotermareatoms}. Parts $(3)$ and $(4)$  follow directly from $(2)$, while part $(5)$ follows from Corollary \ref{faithful}. Part $(6)$ is Corollary \ref{directequivalences} and part $(7)$ is Corollary \ref{twotermareatoms}.
\end{proof}
\noindent In this way, the mutation graph of maximal rigid objects (or equivalently of two-term tilting complexes of $\Lambda_{\con}$) can be viewed as a `picture' of the derived equivalence class; the contraction algebras (the basic members of the equivalence class) sit at the vertices and paths determine all standard derived equivalences. Using the Deligne groupoid, which is also completely determined by the two-term tilting complexes of $\Lambda_{\con}$, we are further able to control the composition of these equivalences and thus obtain a complete understanding of the members of the derived equivalence class and of the standard equivalences between them. In particular, part (6) of Theorem \ref{summary} allows us to determine all the autoequivalences of a contraction algebra and the relations between them.

\begin{remark}
Using part $(6)$ of Theorem \ref{summary}, we would like to say the injective group homomorphism 
\begin{align*}
\pi_1(\mathbb{C}^n \backslash \cH_{\mathbb{C}}) \to \mathrm{Auteq}(\Db(\Lambda_{\con}))
\end{align*}
from Corollary \ref{groupaction} is almost surjective in the sense that it hits every standard equivalence (up to algebra isomorphism). This might allow us to obtain a similar result to that of \cite[4.4]{Miz}; determining the group of standard equivalences, otherwise known as the \emph{derived picard group} of $\Lambda_{\con}$, as a semi-direct product. However, for a maximal rigid object $M$, the object $\Omega M$ is another maximal rigid object with the same stable endomorphism algebra as $M$. In particular, each contraction algebra appears at least twice in our `picture' of the derived equivalence class and thus some autoequivalences are obtained as paths between their corresponding vertices, rather than as a loop at a single vertex. This shows the need for the groupoid picture when visualising the derived equivalence class.
\end{remark}

\appendix

\section{Tracking Through Derived Equivalences} 

This appendix is devoted to proving some of the technical results referred to in the proofs of \S\ref{onesidedsection}. 

Recall that for a complete local isolated cDV singularity $\Spec R$, the category of maximal Cohen-Macaulay modules, denoted $\cmr$, is a Krull-Schmidt, Frobenius category. Further, by Proposition \ref{prop}, the stable category $\ucmr$ is a $k$-linear, Hom-finite, Krull-Schmidt, 2-Calabi--Yau triangulated category whose shift functor is given by the inverse syzygy functor $\Omega^{-1}$, which satisfies $\Omega^2 \cong \id$. In particular, for any $M,N \in \cmr$ there are isomorphisms
\begin{align*}
\Ext_R^1(M,\Omega N) \cong \uHom_R(M, N) \cong  \Ext_R^1(\Omega M, N)
\end{align*}
which we will use freely throughout.
\begin{setup} \label{appendixsetup}
Let $\Spec R$ be a complete local isolated cDV singularity and choose a basic rigid object $M \colonequals\bigoplus\limits_{i=0}^nM_i \in \cmr$ with $M_0 \cong R$. Write $\Lambda \colonequals \End_R(M)$. Choosing $i\neq 0$ and mutating at $M_i$ via the exchange sequence
\begin{eqnarray}
0 \to M_i \xrightarrow{b_i} V_i \xrightarrow{d_i} K_i \to 0 \label{exchange1}
\end{eqnarray}
we obtain $\upnu_iM \colonequals M/M_i\oplus K_i$ and $\Gamma \colonequals \End_R(\upnu_iM)$. Further, we set $\Lambda_{\con} \colonequals \uEnd_R(M)$ and $\Gamma_{\!\con} \colonequals \uEnd_R(\upnu_i M)$. 
\end{setup}

By Theorem \ref{demma}, the $\Gamma$-$\Lambda$-bimodule $T \colonequals \Hom_R(M,\upnu_iM)$ induces an equivalence 
\begin{align}
\RHom_{\Lambda}(T, -) \colon \Db(\Lambda) \to \Db(\Gamma) \label{mmaequiv}
\end{align}
with inverse given by $ - \otimes_{\Gamma}^{\bf L} T$. The purpose of this section is to track $\Gamma_{\!\con} \in \Db(\Gamma)$ back through this equivalence. 

\begin{theorem}[Theorem \ref{proofofappenmain}] \label{appenmain}
With the setup of \ref{appendixsetup}, there is an isomorphism
\begin{align*}
\Gamma_{\!\con} \otimes_{\Gamma}^{\bf L} T \cong \big( 0 \to \bigoplus_{j \neq i} \uHom_R(M,M_j) \big) \oplus \big( \uHom_{R}(M,M_i) \xrightarrow{b_i \circ -} \uHom_{R}(M,V_i) \big)
\end{align*}
in $\Db(\Lambda)$.
\end{theorem}

To set notation, consider the following. 
 
\begin{enumerate}[itemsep=1.2pt, leftmargin=0cm,itemindent=.6cm,labelwidth=\itemindent,labelsep=0cm,align=left]
\item The projective $\Lambda$-modules are $P_i :=\Hom_R(M,M_i)$ for $i=0,\hdots, n$.
\item The projective $\Lambda_{\con}$-modules are $A_i :=\uHom_R(M,M_i)$ for $i=1,\hdots,n$.
\item The projective $\Gamma$-modules are $Q_j:=\Hom_R(\upnu_iM,M_j)$ for $j=0, \dots, n$ and $j\neq i$, and $Q_i :=\Hom_R(\upnu_iM,K_i)$.
\item The projective $\Gamma_{\!\con}$-modules are $B_j:=\uHom_R(\upnu_iM,M_j)$ for $j=1,\hdots,n$ and $j\neq i$, and $B_i:=\uHom_R(\upnu_iM,K_i)$.
\end{enumerate}
To prove Theorem \ref{appenmain}, we begin by tracking projective $\Gamma$-modules through the functor \eqref{mmaequiv}.

\begin{lemma} \label{projequiv}
The equivalence \eqref{mmaequiv} restricts to an equivalence $\proj \Gamma \to \add(T)$, and thus
\begin{align*}
Q_j \otimes_{\Gamma}^{\bf L} T \cong \left\{
	\begin{array}{ll}
		 P_j & \mbox{if $j \neq i$, } \\
		\Hom_R(M, K_i) & \mbox{if $j=i$.} \\
	\end{array}
\right.
\end{align*}
Moreover, if $f \colon N \to N'$ is any morphism in $\add(\upnu_iM)$, then $\Hom_R(\upnu_iM,f)$ maps to $\Hom_R(M,f)$ under the equivalence.
\end{lemma}
\begin{proof}
As $\Gamma$ is a projective $\Gamma$-module, it is clear $\Gamma \otimes_{\Gamma}^{\bf L} T \cong \Gamma \otimes_{\Gamma} T \cong T$ and thus, as $-\otimes_{\Gamma}^{\bf L} T$ is an additive equivalence, it must restrict to an equivalence $\add(\Gamma) \to \add(T)$. Recalling that $\proj\Gamma = \add(\Gamma)$ completes the proof of the first statement.

For an indecomposable module $N \in \add(\upnu_iM)$, the isomorphism can be described explicitly as 
\begin{align*}
\phi_N \colon \Hom_R( \upnu_i M, N) \otimes_{\Gamma} \Hom_R(M, \upnu_iM) &\to \Hom_R(M, N)\\
g \otimes g' &\mapsto g \circ g'
\end{align*}
with inverse given by $h \mapsto \mathrm{pr} \otimes i \circ h$ where $\mathrm{pr}\colon \upnu_iM \to N$ and $i \colon N \to \upnu_iM$ are the natural projection and inclusion maps. Thus, if $f\colon N \to N'$ is any map between indecomposables in $\add(\upnu_iM)$, the equivalence \eqref{mmaequiv} maps $\Hom_R(\upnu_iM,f)$ to 
\[\phi_{N'} \circ (\Hom_R(\upnu_iM, f) \otimes \id) \circ \phi_N^{-1} \colon \Hom_R(M,N) \to \Hom_R(M,N')\]  
satisfying
\begin{align*}
\phi_{N'} \circ (\Hom_R(\upnu_iM, f) \otimes \id) \circ \phi_N^{-1} (s) &= \phi_{N'} \circ (\Hom_R(\upnu_iM, f) \otimes \id) (\mathrm{pr} \otimes i \circ s) \\
&= \phi_{N'} (f \circ \mathrm{pr} \otimes i \circ s) \\
&= f \circ \mathrm{pr} \circ i \circ s \\
&= f \circ s \\
&= \Hom_R(M, f)(s)
\end{align*}
for all $s \colon M \to N$. Recalling that $ -\otimes_{\Gamma}^{\bf L} T$ is an additive functor then completes the proof.
\end{proof}
Using Lemma \ref{projequiv}, $\Gamma_{\!\con}$-modules can be tracked back through the equivalence \eqref{mmaequiv} if their projective resolution as a $\Gamma$-module can be computed. The following lemma helps with this.
\begin{lemma} \label{projresj}
Suppose that $N$ is a basic rigid object in $\cmr$ with $R \in \add(N)$ and that $N_j \ncong R$ is an indecomposable summand of $N$. If $k=\mathrm{rk}(N_j)+\mathrm{rk}( \Omega N_j)$, then there is a projective resolution of $\uHom_{R}(N,N_j)$ as an $\End_R(N)$-module of the form
\[
0 \to\Hom_R(N,N_j)\to\Hom_R(N,R^{k})\to \Hom_R(N,R^{k})\to \Hom_R(N,N_j) \to 0.
\]
\end{lemma}
\begin{proof}
Given $N_j$, Proposition \ref{prop} shows that $\Omega^2 \cong \id$ and hence there are exact sequences 
\begin{align}
0 \to \Omega N_j \to R^k \to N_j \to 0 \label{syzygy11} \\
0 \to N_j \to R^k \to \Omega N_j \to 0 \label{syzygy12}
\end{align}
which come from taking the syzygy of $N_j$ and $\Omega N_j$ respectively. Using that $R$ is injective in $\cmr$ to get $\Ext^1_{R}(N,R)=0$, applying $\Hom_{R}(N,-)$ to \eqref{syzygy11} gives the exact sequence
\begin{align*}
0 \to \Hom_{R}(N, \Omega N_j) \to \Hom_{R}(N,R^k) &\to \Hom_{R}(N,N_j) \to \uHom_{R}(N,N_j) \to 0.
\end{align*}
Similarly, since $N$ is rigid in $\cmr$, applying $\Hom_{R}(N,-)$ to \eqref{syzygy12} gives the exact sequence
\begin{align*}
0 \to \Hom_{R}(N, N_j) \to &\Hom_{R}(N,R^k) \to \Hom_{R}(N,\Omega N_j) \to 0.
\end{align*}
Splicing these two together gives the required result.
\end{proof}
Since $\Gamma_{\con} = \bigoplus_{j=1}^n B_j$, to prove Theorem \ref{appenmain}, it will be enough to track each $B_j$ through the equivalence \eqref{mmaequiv}. We start with the case when $j \neq i$. 
\begin{lemma} \label{jneqi}
Under the setup of \ref{appendixsetup}, when $j \neq i$, $B_j \otimes_{\Gamma}^{\bf{L}} T \cong A_j$ in $\Db(\Lambda)$. 
\end{lemma}
\begin{proof}
By Lemma \ref{projresj} there is a projective resolution of $B_j$ as a $\Gamma$-module of the form 
\begin{align*}
0 \to Q_j \to Q_0^{n_j} \to Q_0^{n_j} \to Q_j \to 0.
\end{align*}
Applying $- \otimes_{\Gamma} T$ termwise to this complex gives $B_j \otimes_{\Gamma}^{\bf{L}} T$ but, by Lemma \ref{projequiv}, this is precisely 
\begin{align*}
0 \to P_j \to P_0^{n_j} \to P_0^{n_j} \to P_j \to 0
\end{align*}
where the maps are precisely those giving a projective resolution of $A_j$ as a $\Lambda$-module, by Lemma \ref{projresj}. Thus, $B_j \otimes_{\Gamma}^{\bf{L}} T$ is quasi-isomorphic to $A_j$ and hence isomorphic in $\Db(\Lambda)$. 
\end{proof}

To deal with the $j=i$ case, we consider the exact sequences
\begin{eqnarray}
0 \to \Omega K_i \xrightarrow{f\textcolor{white}{`}} R^k \xrightarrow{g\textcolor{white}{`}} \hspace{0.12cm}K_i \hspace{0.12cm}\to 0 \label{syzygy1}\\
0 \to \hspace{0.1cm}K_i \hspace{0.11cm} \xrightarrow{f'} R^k \xrightarrow{g' }\Omega K_i \to 0 \label{syzygy2}
\end{eqnarray}
coming from taking syzygies of $K_i$ and $\Omega K_i$ respectively.

\begin{lemma} \label{homology}
Under the setup of \ref{appendixsetup}, the complex $B_i \otimes_{\Gamma}^{\bf{L}} T$ is isomorphic in $\Db(\Lambda)$ to 
\begin{align}
\scalebox{0.98}{$0\to \Hom_{R}(M,K_i)\xrightarrow{f' \circ -} \Hom_{R}(M,R^k)\xrightarrow{f \circ g' \circ -} \Hom_{R}(M,R^k)\xrightarrow{g \circ -} \Hom_{R}(M,K_i)\to 0$}\label{candidate}
\end{align}
and hence has homology; 
\begin{enumerate}[itemsep=0.8pt]
\item $\uHom_{R}(M,K_i)$ in degree $0$;
\item $\uHom_{R}(M,\Omega K_i)$ in degree $-1$;
\item $0$ elsewhere.
\end{enumerate}
\end{lemma}
\begin{proof}
By Lemma \ref{projresj} the sequence
\begin{align*}
0 \to Q_i \xrightarrow{f' \circ -}  Q_0^{k} \xrightarrow{f \circ g' \circ  -}  Q_0^{k} \xrightarrow{g \circ -}  Q_i \to 0
\end{align*}
is a $\Gamma$-projective resolution of $B_i$ and thus to get $B_i \otimes_{\Gamma}^{\bf{L}} T$, we can apply $- \otimes_{\Gamma} T$ termwise to this complex. By Lemma \ref{projequiv} this is exactly \eqref{candidate}.

To compute the homology, apply $\Hom_{R}(M, -)$ to \eqref{syzygy1} to obtain the exact sequence
\begin{align*}
0\to \Hom_{R}(M, \Omega K_i) \xrightarrow{f \circ -} \Hom_{R}(M,R^k) \xrightarrow{g \circ -} \Hom_{R}(M,K_i) \to \Ext^1_{R}(M,\Omega K_i)\to 0
\end{align*} 
which shows 
\begin{align*}
\Hom_{R}(M,K_i)/ \Image(g \circ -) \cong \Ext^1_{R}(M,\Omega K_i) \cong \uHom_{R}(M,K_i).
\end{align*}
This shows the homology in degree $0$. Since the sequence is exact 
\begin{align*}
\Kernel(g \circ -)=\Image(f \circ -) \cong \Hom_{R}(M, \Omega K_i)
\end{align*}
and further, $\Image(f \circ g' \circ - ) \cong \Image( g' \circ - )$ as $f$ is injective. Thus,
\begin{align*}
\Kernel(g \circ -)/\Image(f \circ g' \circ - ) \cong \Hom_{R}(M, \Omega K_i)/\Image (g' \circ -) \cong \uHom_{R}(M, \Omega K_i)
\end{align*}
where the last isomorphism comes from the exact sequence
\begin{align*}
0\to \Hom_{R}(M, K_i) \xrightarrow{f' \circ -} \Hom_{R}(M,R^k) \xrightarrow{g' \circ -} \Hom_
{R}(M, \Omega K_i) \to \Ext^1_{R}(M, K_i)\to 0
\end{align*}
obtained by applying $\Hom_{R}(M, -)$ to the exact sequence \eqref{syzygy2}. This sequence also shows that the complex \eqref{candidate} is exact elsewhere.
\end{proof}

Since $B_i \otimes_\Gamma^{\bf L} T$ has zero homology outside degrees $-1$ and $0$, the complex can be truncated appropriately to give a quasi-isomorphic complex.
\begin{corollary} \label{truncquasi}
The complex $B_i \otimes_{\Gamma}^{\bf{L}} T$ is quasi-isomorphic to the truncated complex
\begin{equation}
0 \to \Hom_{R}(M,R^k)/ \Image(f \circ g' \circ -) \xrightarrow{g \circ -} \Hom_{R}(M,K_i)\to 0. \label{trunc}
\end{equation} 
\end{corollary}
In particular, to prove Theorem \ref{appenmain}, we now only need to show the complexes of $\Lambda$-modules, $\eqref{trunc}$ and 
\begin{align*}
\uHom_{R}(M,M_i) \xrightarrow{b_i \circ -} \uHom_{R}(M,V_i),
\end{align*}
are quasi-isomorphic. To do this, we will construct a complex of projective $\Lambda$-modules which is quasi-isomorphic to both. This involves first finding projective resolutions of $\Hom_{R}(M,K_i)$ and $\Hom_{R}(M,R^k)/ \Image(f \circ g' \circ -) $.
\begin{lemma} \label{projectiveresolutions}
\begin{enumerate}[leftmargin=0cm,itemindent=.6cm,labelwidth=\itemindent,labelsep=0cm,align=left]
\item The sequence
\begin{equation}
0 \to \Hom_{R}(M,M_i) \xrightarrow{b_i \circ -} \Hom_{M}(M,V_i) \xrightarrow{d_i \circ -} \Hom_{R}(M,K_i) \to 0 \label{proj1}
\end{equation}
is exact and so is a projective resolution of $\Hom_{R}(M,K_i)$ as a $\Lambda$-module.
\item The sequence
\vspace{-0.6cm}
\begin{center}
\begin{tikzpicture}[scale=0.95]
\node (1) at (0,0) {$0$};
\node (2) at (2,0) {$\Hom_{R}(M,M_i)$};
\node (3) at (5.3,0) {$\Hom_{R}(M,V_i)$};
\node (4) at (9,0) {$\Hom_{R}(M,R^k)$};
\node (5) at (6,-1) {$\Hom_{R}(M,R^k)$};
\node (6) at (10.5,-1) {$\Hom_{R}(M,R^k)/\Image(f \circ g' \circ -)$};
\node (7) at (13.8,-1) {$0$};
\node (8) at (10.995,0.008) {};
\node (9) at (11.305,0) {};
\draw[->] (1) to (2);
\draw[->] (2) to node[above] {$\scriptstyle b_i \circ -$} (3); 
\draw[->] (3) to node[above] {$\scriptstyle f' \circ d_i \circ -$} (4); 
\draw[-,out=0, in=179,looseness=0.2] (4) to node[above] {$\scriptstyle f \circ g' \circ -$} (9); 
\draw[->] (5) to (6);
\draw[->] (6) to (7);
\draw[->,out=-3, in=177, looseness=1.8] (8) to (5);
\end{tikzpicture}
\end{center}
is a projective resolution of $\Hom_{R}(M,R^k)/ \Image(f \circ g' \circ -)$ as a $\Lambda$-module.
\end{enumerate}
\end{lemma}
\begin{proof}
\begin{enumerate}[leftmargin=0cm,itemindent=.6cm,labelwidth=\itemindent,labelsep=0cm,align=left]
\item Applying $\Hom_{R}(M,-)$ to the exchange sequence \eqref{exchange1} and using the rigidity of $M$ shows $\Ext_{R}^1(M,M_i)=0$ so that \eqref{proj1} is exact.
\item
The sequence 
\begin{align*}
0 \to \Hom_{R}(M,K_i) \xrightarrow{f' \circ -}\Hom_{R}(M,R^k) \xrightarrow{f \circ g' \circ  -} \Hom_{R}(M,R^k)
\end{align*}
is exact using the proof of Lemma \ref{homology}. Taking the cokernel and splicing this sequence with \eqref{proj1} gives the result. \qedhere
\end{enumerate}
\end{proof}

To construct a projective complex quasi-isomorphic to the complex \eqref{trunc}, we need maps between the projective resolutions constructed in Lemma \ref{projectiveresolutions}. For this, the following lemma is useful.
\begin{lemma} \label{xs}
There exists maps $s\colon R^k \to V_i$ and $x\colon \Omega K_i \to M_i$ such that the following diagram commutes.
\begin{center}
\begin{tikzpicture}
  \matrix (m) [matrix of math nodes,row sep=2.5em,column sep=3em,minimum width=2em] {  
0 &  M_i & V_i & K_i & 0 \\
0 &  \Omega K_i & R^k & K_i & 0\\};
 \path[-stealth]
    (m-1-1.east|-m-1-2) edge node [above] {} (m-1-2)
    (m-1-2.east|-m-1-3) edge node [above] {$ \scriptstyle b_i$} (m-1-3)
    (m-1-3.east|-m-1-4) edge node [above] {$\scriptstyle d_i$} (m-1-4)
    (m-1-4.east|-m-1-5) edge node [above] {} (m-1-5)
    (m-2-1.east|-m-2-2) edge node [above] {} (m-2-2)
    (m-2-2.east|-m-2-3) edge node [above] {$\scriptstyle f$} (m-2-3)
    (m-2-3.east|-m-2-4) edge node [above] {$\scriptstyle g$} (m-2-4)
    (m-2-4.east|-m-2-5) edge node [above] {} (m-2-5)
 (m-2-2) edge node [right] {$\scriptstyle x$} (m-1-2)
 (m-2-3) edge node [right] {$ \scriptstyle s$} (m-1-3)
 (m-2-4) edge node [right] {$\scriptstyle \id$} (m-1-4);
\end{tikzpicture}
\end{center} 
Further, these maps give an exact sequence
\begin{align}
0 \to \Omega K_i \xrightarrow{\scriptsize \begin{pmatrix} -f \\ x \end{pmatrix}} R^k \oplus M_i \xrightarrow{\scriptsize \begin{pmatrix}  s & b_i \end{pmatrix}}  V_i \to 0. \label{a10sequence}
\end{align}
\end{lemma}
\begin{proof}
The map $s$ exists because $R^k$ is projective and $d_i$ is a surjective map. Then the map $x$ exists using the universal property of kernels.

Viewing the rows of the commutative diagram as complexes, this construction gives a map between two exact complexes which therefore must be a quasi-isomorphism. Thus, the mapping cone,
\begin{align*}
0 \to \Omega K_i \xrightarrow{\scriptsize \begin{pmatrix} -f \\ x \end{pmatrix}} R^k \oplus M_i \xrightarrow{\scriptsize \begin{pmatrix} -g & 0 \\ s & b_i \end{pmatrix}} K_i \oplus V_i, \xrightarrow{\scriptsize \begin{pmatrix} \mathrm{id} & d_i \end{pmatrix}} K_i \to 0.
\end{align*}
is necessarily exact. In the commutative diagram below,
\begin{center}
\begin{tikzpicture}
  \matrix (m) [matrix of math nodes,row sep=2.8em,column sep=3em,minimum width=2em] {  
\ &0 & 0 & 0 & 0 & \  \\
0 & \Omega K_i & R^k \oplus M_i & V_i & 0 & 0 \\
0 & \Omega K_i & R^k \oplus M_i & K_i \oplus V_i & K_i & 0 \\
0 & 0 & 0 & K_i & K_i & 0 \\
\ &0 & 0 & 0 & 0 & \  \\};
 \path[-stealth]
    (m-2-1.east|-m-2-2) edge node [above] {} (m-2-2)
    (m-2-2.east|-m-2-3) edge node [above] {$\scriptsize{\begin{pmatrix} -f \\ x \end{pmatrix} }$} (m-2-3)
    (m-2-3.east|-m-2-4) edge node [above] {$\scriptsize{\begin{pmatrix}  s & b_i \end{pmatrix}} $} (m-2-4)
    (m-2-4.east|-m-2-5) edge node [above] {$\scriptstyle 0$} (m-2-5)
    (m-2-5.east|-m-2-6) edge node [above] {} (m-2-6)
(m-3-1.east|-m-3-2) edge node [above] {} (m-3-2)
    (m-3-2.east|-m-3-3) edge node [above] {$\scriptsize{\begin{pmatrix} -f \\ x \end{pmatrix} }$} (m-3-3)
    (m-3-3.east|-m-3-4) edge node [above] {$\scriptsize{\begin{pmatrix} -g & 0 \\ s & b_i \end{pmatrix}} $} (m-3-4)
    (m-3-4.east|-m-3-5) edge node [above] {$\scriptsize{\begin{pmatrix} \mathrm{id} & d_i \end{pmatrix}}$} (m-3-5)
    (m-3-5.east|-m-3-6) edge node [above] {} (m-3-6)
(m-4-1.east|-m-4-2) edge node [above] {} (m-4-2)
    (m-4-2.east|-m-4-3) edge node [above] {$\scriptstyle{0 }$} (m-4-3)
    (m-4-3.east|-m-4-4) edge node [above] {$\scriptstyle{0} $} (m-4-4)
    (m-4-4.east|-m-4-5) edge node [above] {$\scriptstyle{\mathrm{id}}$} (m-4-5)
    (m-4-5.east|-m-4-6) edge node [above] {} (m-4-6)
(m-1-2) edge node [right] {} (m-2-2)
 (m-2-2) edge node [right] {$\scriptstyle \mathrm{id}$} (m-3-2)
 (m-3-2) edge node [right] {} (m-4-2)
 (m-4-2) edge node [right] {} (m-5-2)
(m-1-3) edge node [right] {} (m-2-3)
 (m-2-3) edge node [right] {$\scriptstyle \mathrm{id}$} (m-3-3)
 (m-3-3) edge node [right] {} (m-4-3)
 (m-4-3) edge node [right] {} (m-5-3)
(m-1-4) edge node [right] {} (m-2-4)
 (m-2-4) edge node [right] {$\scriptsize{\begin{pmatrix} -d_i \\ \mathrm{id} \end{pmatrix}} $} (m-3-4)
 (m-3-4) edge node [right] {$\scriptsize{\begin{pmatrix} \mathrm{id} & d_i \end{pmatrix}} $} (m-4-4)
 (m-4-4) edge node [right] {} (m-5-4)
(m-1-5) edge node [right] {} (m-2-5)
 (m-3-5) edge node [right] {$\scriptstyle \mathrm{id}$} (m-4-5)
 (m-2-5) edge node [right] {} (m-3-5)
 (m-4-5) edge node [right] {} (m-5-5);
\end{tikzpicture}
\end{center}
the first, second and fourth columns are obviously exact while the third column is exact as it is the mapping cone of the map:
\begin{center}
\begin{tikzpicture}
  \matrix (m) [matrix of math nodes,row sep=2em,column sep=3em,minimum width=2em] {  
0 &  V_i & V_i & 0 \\
0 & K_i & K_i & 0.\\};
 \path[-stealth]
    (m-1-1.east|-m-1-2) edge node [above] {} (m-1-2)
    (m-1-2.east|-m-1-3) edge node [above] {\scriptsize $-\mathrm{id}$} (m-1-3)
    (m-1-3.east|-m-1-4) edge node [above] {} (m-1-4)
    (m-2-1.east|-m-2-2) edge node [above] {} (m-2-2)
    (m-2-2.east|-m-2-3) edge node [above] {\scriptsize $\mathrm{id}$} (m-2-3)
    (m-2-3.east|-m-2-4) edge node [above] {} (m-2-4)
 (m-1-2) edge node [right] {\scriptsize $-d_i$} (m-2-2)
 (m-1-3) edge node [right] {\scriptsize $d_i$} (m-2-3);
\end{tikzpicture}
\end{center}
Thus all the columns are exact and so we have a short exact sequence of complexes. Considering the long exact sequence of homology associated to this short exact sequence shows that the first row is exact since the second and third are.
\end{proof}
This gives us everything we need to construct a complex of projective $\Lambda$-modules quasi-isomorphic to $B_i \otimes_\Gamma^{\bf L} T$.
\begin{lemma} \label{quasiso1}
With notation as above, and writing ${}_R(X,Y)=\Hom_R(X,Y)$, the chain map 
\begin{center}
\begin{tikzpicture}
  \matrix (m) [matrix of math nodes,row sep=2.5em,column sep=2.2em,minimum width=2em] {  
\scalebox{0.92}{$_{R}(M,M_i) $} &\scalebox{0.92}{$ _{R}(M,V_i) $} & \scalebox{0.92}{$ _{R}(M,R^k) $}&\scalebox{0.92}{$  _{R}(M,R^k) \oplus  \hspace{-0.2cm}\ _{R}(M,M_i) $}& \scalebox{0.92}{$ _{R}(M,V_i) $}\\
\scalebox{0.92}{$ 0 $}& \scalebox{0.92}{$  0$} & \scalebox{0.92}{$  0$} & \scalebox{0.92}{$  _{R}(M,R^k)/\Image(f \circ g' \circ -)$} & \scalebox{0.92}{$ _{R}(M,K_i)$}\\};
 \path[-stealth]
    (m-1-1.east|-m-1-2) edge node [above] {$\scriptstyle b_i \circ -$ } (m-1-2)
    (m-1-2.east|-m-1-3) edge node [above] {$\scriptstyle f'_i \circ d_i \circ -$} (m-1-3)
    (m-1-3.east|-m-1-4) edge node [above] {$ \scriptstyle{f\circ g'\circ - \choose -x\circ g'\circ-}$} (m-1-4)
    (m-1-4.east|-m-1-5) edge node [above] { \scalebox{0.94}{$\scriptstyle \big( \scriptstyle s \circ -, \  \scriptstyle b_i \circ -\big)$}} (m-1-5)
    (m-2-1.east|-m-2-2) edge node [above] {} (m-2-2)
    (m-2-2.east|-m-2-3) edge node [above] {} (m-2-3)
    (m-2-3.east|-m-2-4) edge node [above] {} (m-2-4)
    (m-2-4.east|-m-2-5) edge node [above] {$ \scriptstyle g \circ - $} (m-2-5)
 (m-1-1) edge node [right] {$\scriptstyle 0$} (m-2-1)
 (m-1-2) edge node [right] {$\scriptstyle 0$} (m-2-2)
 (m-1-3) edge node [right] {$\scriptstyle 0$} (m-2-3)
 (m-1-4) edge node [right] {$\scriptsize{\begin{pmatrix}0 & \mathrm{pr}  \end{pmatrix}}$} (m-2-4)
 (m-1-5) edge node [right] {$\scriptstyle d_i \circ -$} (m-2-5);
\end{tikzpicture}
\end{center}
 is a quasi-isomorphism where $\mathrm{pr} \colon _R(M,R^k) \to _R(M,R^k)/\Image(f \circ g' \circ -)$ denotes the natural surjection. In particular, $B_i \otimes_{\Gamma}^{\bf{L}} T$ is isomorphic in $\Db(\Lambda)$ to the complex in the top row.
\end{lemma}
\begin{proof} First note that the diagram is a chain map because $d_i \circ b_i=0$ and $d_i \circ s=g$ by construction. Also, by Lemma \ref{projectiveresolutions} and Lemma \ref{xs}, there is a double complex
\begin{center}
\begin{tikzpicture}
  \matrix (m) [matrix of math nodes,row sep=2em,column sep=3em,minimum width=2em] {  
\Hom_{R}(M,M_i) & 0  \\
\Hom_{R}(M,V_i) & 0\\
\Hom_{R}(M,R^k) & \Hom_{R}(M,M_i)\\
\Hom_{R}(M,R^k) & \Hom_{R}(M,V_i)\\
\Hom_{R}(M,R^k)/\Image(f \circ g' \circ -)& \Hom_{R}(M,K_i)\\};
 \path[-stealth]
    (m-1-1.east|-m-1-2) edge node [above] {$\scriptstyle 0$} (m-1-2)
    (m-2-1.east|-m-2-2) edge node [above] {$\scriptstyle 0$} (m-2-2)
    (m-3-1.east|-m-3-2) edge node [above] {$\scriptstyle x \circ g' \circ -$} (m-3-2)
    (m-4-1.east|-m-4-2) edge node [above] {$\scriptstyle -(s \circ -)$} (m-4-2)
    (m-5-1.east|-m-5-2) edge node [above] {$\scriptstyle g \circ -$} (m-5-2)
(m-1-1) edge node [right] {$\scriptstyle b_i \circ -$} (m-2-1)
 (m-2-1) edge node [right] {$\scriptstyle f' \circ d_i \circ -$} (m-3-1)
 (m-3-1) edge node [right] {$\scriptstyle f \circ g' \circ -$} (m-4-1)
 (m-4-1) edge node [right] {$\scriptstyle \mathrm{pr}$} (m-5-1)
(m-1-2) edge node [right] {$\scriptstyle 0$} (m-2-2)
 (m-2-2) edge node [right] {$\scriptstyle 0$} (m-3-2)
 (m-3-2) edge node [right] {$\scriptstyle b_i \circ -$} (m-4-2)
 (m-4-2) edge node [right] {$\scriptstyle d_i \circ -$} (m-5-2);
\end{tikzpicture}
\end{center}
where the columns are acyclic. A standard result from homological algebra says the total complex of a bounded double complex with acyclic columns is acyclic \cite[1.2.5]{weibel}. However, the total complex of this double complex is precisely (up to $\pm$ signs on the maps) the mapping cone of the given chain map and so the chain map must be a quasi-isomorphism. The final statement then follows by combining this quasi-isomorphism with that of Lemma \ref{truncquasi}.
\end{proof}
Finally, it needs to be shown that the complex of projectives constructed in Lemma \ref{quasiso1} is also quasi-isomorphic to 
\begin{align*}
\uHom_{R}(M,M_i) \xrightarrow{b_i \circ -} \uHom_{R}(M,V_i). 
\end{align*}
\begin{lemma} \label{quasiso}
Writing $_R(X,Y) \colonequals \Hom_R(X,Y)$, the following chain map is a quasi-isomorphism:

\vspace*{-0.2em}
\hspace{-0.5cm}\begin{tikzpicture}
  \matrix (m) [matrix of math nodes,row sep=3em,column sep=2.2em,minimum width=2em] {  
\scalebox{0.95}{$_{R}(M,M_i)$} & \scalebox{0.95}{$ _{R}(M,V_i)$} & \scalebox{0.95}{$_{R}(M,R^k)$} & \scalebox{0.95}{$ _{R}(M,R^k) \oplus \hspace{-0.2cm}\ _{R}(M,M_i)$} &  \scalebox{0.95}{$ _{R}(M,V_i)$} \\
 \scalebox{0.95}{$ 0$} &  \scalebox{0.95}{$ 0$} &  \scalebox{0.95}{$0$} & \scalebox{0.95}{$\uHom_{R}(M,M_i)$} & \scalebox{0.95}{$\uHom_{R}(M,V_i).$}\\};
 \path[-stealth]
    (m-1-1.east|-m-1-2) edge node [above] {$\scriptstyle b_i \circ -$ } (m-1-2)
    (m-1-2.east|-m-1-3) edge node [above] {$\scriptstyle f' \circ d_i \circ -$} (m-1-3)
    (m-1-3.east|-m-1-4) edge node [above] {$\scriptstyle{f\circ g'\circ - \choose -x\circ g'\circ-}$} (m-1-4)
    (m-1-4.east|-m-1-5) edge node [above] { \scalebox{0.94}{$\scriptstyle \big( \scriptstyle s \circ -, \  \scriptstyle b_i \circ -\big)$}} (m-1-5)
    (m-2-1.east|-m-2-2) edge node [above] {} (m-2-2)
    (m-2-2.east|-m-2-3) edge node [above] {} (m-2-3)
    (m-2-3.east|-m-2-4) edge node [above] {} (m-2-4)
    (m-2-4.east|-m-2-5) edge node [above] {$ \scriptstyle b_i \circ - $} (m-2-5)
 (m-1-1) edge node [right] {$\scriptstyle 0$} (m-2-1)
 (m-1-2) edge node [right] {$\scriptstyle 0$} (m-2-2)
 (m-1-3) edge node [right] {$\scriptstyle 0$} (m-2-3)
 (m-1-4) edge node [right] {$\scriptsize{\begin{pmatrix}0 & \mathrm{pr}  \end{pmatrix}}$} (m-2-4)
 (m-1-5) edge node [right] {$\scriptstyle \mathrm{pr}$} (m-2-5);
\end{tikzpicture}
\vspace*{-1em}
\end{lemma}
\begin{proof}
Combining Lemma \ref{quasiso1} and Lemma \ref{homology}, the top complex has homology:
\begin{enumerate}
\item $\uHom_{R}(M,K_i)$ in degree $0$;
\item $\uHom_{R}(M,\Omega K_i)$ in degree $-1$;
\item $0$ elsewhere.
\end{enumerate}
The exchange sequence \eqref{exchange1} induces a triangle 
\begin{align*}
M_i \xrightarrow{b_i} V_i &\xrightarrow{d_i} K_i \to \Omega M_i
\end{align*}
in $\ucmr$. Applying $\uHom_{R}(M,-)$ and using rigidity of $M$ gives an exact sequence
\begin{align*}
0 \to \uHom_{R}(M,\Omega K_i) \xrightarrow{} \uHom_{R}(M,M_i) \xrightarrow{b_i \circ -} \uHom_{R}(M,V_i) \xrightarrow{d_i \circ -} \uHom_{R}(M,K_i) \to 0 
\end{align*}
which shows the homology of the two complexes are the same. Thus, to prove the claim it is enough to show the maps on homology are surjective in degrees $0$ and $-1$.\\
In degree $0$, the map
\begin{align*}
\Hom_{R}(M,V_i) \xrightarrow{\mathrm{pr}} \uHom_{R}(M,V_i) \to \uHom_{R}(M,V_i)/\Image(b_i \circ -)
\end{align*}
is the composition of two surjective maps and hence is surjective. Also, any map in the image of $\begin{pmatrix}s \circ - & b_i \circ - \end{pmatrix}$ is clearly in the kernel of this map and so the map on homology must be surjective.\\

In degree $-1$, take $\upalpha\colon M \to M_i$ such that $b_i \circ \upalpha$ factors through $\add(R)$; that is, $\mathrm{pr}(\upalpha) \in \Ker ( \uHom_R(M,M_i) \xrightarrow{b_i \circ -} \uHom_R(M,V_i)) $. In particular, $b_i \circ \upalpha = \updelta_2 \circ \updelta_1$ for some $\updelta_1\colon M \to R^m$ and $\updelta_2 \colon R^m \to V_i$ where $m \in \mathbb{N}$. To show the map on homology is surjective we need to show there exists $\phi\colon M \to R^k$ and $\phi'\colon M \to M_i$ such that $(\phi, \phi')$ is in the kernel of $ \begin{pmatrix} s \circ - & b_i \circ - \end{pmatrix}$ and $\mathrm{pr}(\phi')=\mathrm{pr}(\upalpha)$. 

Since $R^m$ is projective, $\Hom_{R}(R^m,-)$ is exact and hence applying this to the exact sequence \eqref{a10sequence} shows that
\begin{align*}
\Hom_{R}(R^m,R^k) \oplus \Hom_{R}(R^m,M_i) \xrightarrow{\scriptsize{\begin{pmatrix} s \circ - & b_i \circ - \end{pmatrix}}} \Hom_{R}(R^m,V_i)
\end{align*} 
is surjective. Thus there exists $\upbeta\colon R^m \to R^k$ and $\upgamma \colon R^m \to M_i$ such that
\begin{align*}
\updelta_2 =s \circ \upbeta +b_i \circ \upgamma.
\end{align*}
This gives
\begin{align*}
b_i \circ \upalpha= s \circ \upbeta \circ \updelta_1 + b_i \circ \upgamma \circ \updelta_1
\end{align*}
and so $(-\upbeta \circ \updelta_1, \upalpha- \upgamma \circ \updelta_1)$ belongs to the kernel of
 \begin{align*}
\Hom_{R}(M,R^k) \oplus \Hom_{R}(M,M_i) \xrightarrow{\scriptsize \begin{pmatrix} s \circ - & b_i \circ - \end{pmatrix}} \Hom_{R}(M,V_i).
\end{align*} 
Moreover, since $\mathrm{pr} \colon \Hom_R(M,M_i) \to \uHom_R(M,M_i)$ is the natural surjection, applying $\begin{pmatrix} 0 & \mathrm{pr}\end{pmatrix}$ to $(-\upbeta \circ \updelta_1, \upalpha- \upgamma \circ \updelta_1)$ gives
\begin{align*}
\mathrm{pr}( \upalpha- \upgamma \circ \updelta_1)= \mathrm{pr}(\upalpha)
\end{align*} 
as $\upgamma \circ \updelta_1$ factors through $\add(R)$. This shows the map on homology is surjective and hence is an isomorphism.
\end{proof}
We now have all the results required to prove Theorem \ref{appenmain}.

\begin{theorem}[Theorem \ref{appenmain}] \label{proofofappenmain}
With the setup of \ref{appendixsetup}, there is an isomorphism
\begin{align*}
\Gamma_{\!\con} \otimes_{\Gamma}^{\bf L} T \cong \big( 0 \to \bigoplus_{j \neq i} \uHom_R(M,M_j) \big) \oplus \big( \uHom_{R}(M,M_i) \xrightarrow{b_i} \uHom_{R}(M,V_i) \big). 
\end{align*}
in $\Db(\Lambda)$.
\end{theorem}
\begin{proof}
Since $\Gamma_{\!\con} \cong \bigoplus_{j=1}^n B_j$, it is enough to show that
\begin{align*}
B_j \otimes_{\Gamma}^{\bf L} T \cong \uHom_R(M,M_j)
\end{align*}
when $j \neq i$, which holds by Lemma \ref{jneqi} and that
\begin{align*}
B_i \otimes_{\Gamma}^{\bf L} T \cong  \big(\uHom_{R}(M,M_i) \xrightarrow{b_i \circ -} \uHom_{R}(M,V_i)\big)
\end{align*}
which follows by combining Lemmas \ref{quasiso1} and \ref{quasiso}.
\end{proof}
\begin{remark} \label{defS}
Note that in this process we have constructed a complex of projective $\Lambda$-modules, which is quasi-isomorphic to the tilting complex,
\begin{align}
\big( 0 \to \bigoplus_{j \neq i} \uHom_R(M,M_j) \big) \oplus \big( \uHom_{R}(M,M_i) \xrightarrow{b_i} \uHom_{R}(M,V_i) \big), \label{twotermcomplexP}
\end{align}
from Theorem \ref{decon}. In particular, when $j \neq i$ take $\EuScript{P}_j$ to be the projective resolution  
\begin{align*}
0 \to\Hom_R(M,M_j)\to\Hom_R(M,R^{n_j})\to \Hom_R(M,R^{n_j})\to \Hom_R(M,M_j) \to 0.
\end{align*}
of $\uHom_R(M,M_j)$ as a $\Lambda$-module, as in Lemma \ref{projresj}. Further, set $\EuScript{P}_i$ to be the complex of projective $\Lambda$-modules
\begin{align*}
\scalebox{0.9}{$0 \to \Hom_R(M,M_i) \to \Hom_R(M,V_i) \to \Hom_R(M,R^k) \to \Hom(M,R^k \oplus M_i) \to \Hom_R(M,V_i) \to 0$},
\end{align*}
constructed in \ref{quasiso1}, which, by \ref{quasiso}, is quasi-isomorphic to 
\begin{align*}
\uHom_R(M,M_i) \xrightarrow{b_i \circ -} \uHom_R(M,V_i).
\end{align*}
Then $\EuScript{P} \colonequals \bigoplus\limits_{j=1}^n \EuScript{P}_j$ is clearly quasi-isomorphic to \eqref{twotermcomplexP}. This is a tool useful for calculations in the main paper.
\end{remark}


\begin{thebibliography}{BIKR}
\addcontentsline{toc}{chapter}{Bibliography}




\bibitem[AIR]{[AIR]}
T.~Adachi, O.~Iyama, and I.~Reiten, \emph{$\tau$-tilting theory}, Compositio
  Mathematica \textbf{150} (2014), no.~3, 415--452.

\bibitem[A]{tiltingconnectedness}
T.~Aihara, \emph{Tilting-connected symmetric algebras}, Algebras and
  Representation Theory \textbf{16} (2013), no.~3, 873--894.

\bibitem[AI]{siltingmutation}
T.~Aihara and O.~Iyama, \emph{Silting mutation in triangulated categories},
  Journal of the London Mathematical Society \textbf{85} (2012), no.~3,
  633--668.




\bibitem[Au]{rigidequiv} J.~August, \emph{On the finiteness of the derived equivalence classes of some stable endomorphism rings}, Math. Z. \url{https://doi.org/10.1007/s00209-020-02475-y}.


\bibitem[BT]{bravthomas}
C.~Brav and H.~Thomas, \emph{Braid groups and {K}leinian singularities}, Math.
  Ann. \textbf{351} (2011), no.~4, 1005--1017.
  
  \bibitem[B]{bridgeland}
T.~Bridgeland, \emph{Flops and derived categories}, Invent. Math. \textbf{147}
  (2002), no.~3, 613--632.

\bibitem[Br]{broue}
M.~Brou\'{e}, \emph{Higman's criterion revisited}, Michigan Math. J.
  \textbf{58} (2009), no.~1, 125--179.



\bibitem[BIKR]{symmetric}
I.~Burban, O.~Iyama, B.~Keller, and I.~Reiten, \emph{Cluster tilting for
  one-dimensional hypersurface singularities}, Advances in Mathematics
  \textbf{217} (2008), no.~6, 2443--2484.


\bibitem[C]{chen}
J.-C. Chen, \emph{Flops and equivalences of derived categories for threefolds
  with only terminal {G}orenstein singularities}, J. Differential Geom.
  \textbf{61} (2002), no.~2, 227--261.

\bibitem[DIJ]{dij}
L.~Demonet, O.~Iyama and G.~Jasso, \emph{{$\tau$}-tilting finite algebras, bricks, and {$g$}-vectors}, 
Int. Math. Res. Not. IMRN (2019), no.~3, 852--892. 

\bibitem[DW1]{DefFlops}
W.~Donovan and M.~Wemyss, \emph{Noncommutative deformations and flops}, Duke Math. J.
  \textbf{165} (2016), no.~8, 1397--1474.

\bibitem[DW2]{ConDef}
W.~Donovan and M.~Wemyss, \emph{Contractions and deformations},
 Amer. J. Math. \textbf{141} (2019), no.~3, 563--592.

\bibitem[DW3]{twists}
W.~Donovan and M.~Wemyss, \emph{Twists and braids for general threefold flops},
J. Eur. Math. Soc. (JEMS) \textbf{21} (2019), no.~6, 1641--1701.


\bibitem[D]{Dugas}
A.~Dugas, \emph{A construction of derived equivalent pairs of symmetric
  algebras}, Proceedings of the American Mathematical Society \textbf{143}
  (2015), no.~6, 2281--2300.


\bibitem[HT]{invsing}
Z.~Hua and Y.~Toda, \emph{Contraction algebra and invariants of singularities},
  Int. Math. Res. Not. IMRN (2018), no.~10, 3173--3198.

\bibitem[HW]{faithful}
Y.~Hirano and M.~Wemyss, \emph{Faithful actions from hyperplane arrangements},
  Geom. Topol. \textbf{22} (2018), no.~6, 3395--3433.

\bibitem[IW]{mmas}
O.~Iyama and M.~Wemyss, \emph{Maximal modifications and Auslander-Reiten
  duality for non-isolated singularities}, Invent. Math. \textbf{197} (2014),
  no.~3, 521--586.

%
%

\bibitem[IY]{yoshino}
O.~Iyama and Y.~Yoshino, \emph{Mutation in triangulated categories and rigid
  Cohen--Macaulay modules}, Invent. Math. \textbf{172} (2008), no.~1, 117--168.

\bibitem[K]{[K]}
Y.~Kawamata, \emph{Flops connect minimal models}, Publ. Res. Inst. Math. Sci.
  \textbf{44} (2008), no.~2, 419--423.

\bibitem[KM]{[KM]}
Y.~Kawamata and K.~Matsuki, \emph{The number of the minimal models for a 3-fold
  of general type is finite}, Math. Ann. \textbf{267} (1987), no.~4, 595--598.

\bibitem[Ke1]{keller}
B.~Keller, \emph{On the construction of triangle equivalences}, Derived
  equivalences for group rings, Lecture Notes in Math., vol. 1685, Springer,
  Berlin, 1998, pp.~155--176. 

\bibitem[Ke2]{keller2}
B.~Keller, \emph{Derived categories and tilting}, Handbook of tilting theory,
  London Math. Soc. Lecture Note Ser., vol. 332, Cambridge Univ. Press,
  Cambridge, 2007, pp.~49--104.

\bibitem[Ke3]{kellerhomotopy}
B.~Keller, \emph{Bimodule complexes via strong homotopy actions}, Algebra and
  Representation Theory \textbf{3} (2000), no.~4, 357--376.


\bibitem[Ko]{kollar}
J.~Koll{\'a}r, \emph{Flops}, Nagoya Math. J. \textbf{113} (1989), 15--36.


\bibitem[Mi]{Miz}
Y.~Mizuno, \emph{Derived picard groups of preprojective algebras of Dynkin
  type}, International Mathematics Research Notices, \url{https://doi.org/10.1093/imrn/rny299}.


\bibitem[P1]{hyperplane1}
L.~Paris, \emph{Universal cover of {S}alvetti's complex and topology of
  simplicial arrangements of hyperplanes}, Trans. Amer. Math. Soc. \textbf{340}
  (1993), no.~1, 149--178.

\bibitem[P2]{hyperplane3}
L.~Paris, \emph{On the fundamental group of the complement of a complex
  hyperplane arrangement}, Arrangements--{T}okyo 1998, Adv. Stud. Pure Math.,
  vol.~27, Kinokuniya, Tokyo, 2000, pp.~257--272.

\bibitem[R]{reid}
M.~Reid, \emph{Minimal models of canonical 3-folds}, Algebraic Varieties and
  Analytic Varieties (Tokyo, Japan), Mathematical Society of Japan, 1983,
  pp.~131--180.


\bibitem[Ri]{morita}
J.~Rickard, \emph{Morita theory for derived categories}, J. London Math. Soc.
  \textbf{39} (1989), no.~2, 436--456.


\bibitem[RZ]{rouquier}
R.~Rouquier and A.~Zimmermann, \emph{Picard groups for derived module
  categories}, Proc. London Math. Soc. \textbf{87} (2003), no.~3, 197--225.

\bibitem[S]{hyperplane2}
M.~Salvetti, \emph{Topology of the complement of real hyperplanes in {${\mathbb
  C}^N$}}, Invent. Math. \textbf{88} (1987), no.~3, 603--618.

\bibitem[Sc]{Schroer}
S.~Schr{\"o}er, \emph{A characterization of semiampleness and contractions of
  relative curves}, Kodai Math. J. \textbf{24} (2001), no.~2, 207--213.

\bibitem[W]{weibel}
Charles~A. Weibel, \emph{An introduction to homological algebra}, Cambridge
  Studies in Advanced Mathematics, vol.~38, Cambridge University Press,
  Cambridge, 1994.


\bibitem[We]{HomMMP}
M.~Wemyss, \emph{Flops and clusters in the homological minimal model programme},
  Invent. Math. \textbf{211} (2018), no.~2, 435--521.



\end{thebibliography}
\end{document}